%% file: main.tex
\documentclass[11pt, reqno]{amsart}
\usepackage{extarrows}
\usepackage{newtxtext}
\usepackage[scale=1]{inconsolata}
\usepackage{bm}

\renewcommand{\leq}{\leqslant}
\renewcommand{\geq}{\geqslant}

\renewcommand{\c}{\mathsf{c}}

\newcommand{\AS}{\mathsf{ApproxSampler}}
\newcommand{\US}{\mathsf{UniformSampler}}

\input{preamble/preamble.tex}

\input{preamble/additional_macros.tex}

\begin{document}
\title{Decoupling of clusters in independent sets in a percolated hypercube}

\author{Mriganka Basu Roy Chowdhury, Shirshendu Ganguly and Vilas Winstein}
\address{Mriganka Basu Roy Chowdhury\\ University of California, Berkeley}
\email{mriganka\_brc@berkeley.edu}
\address{Shirshendu Ganguly\\ University of California, Berkeley}
\email{sganguly@berkeley.edu}
\address{Vilas Winstein\\ University of California, Berkeley}
\email{vilas@berkeley.edu}

\begin{abstract}
Independent sets are sets of vertices in a graph which contain no neighbors, and are well-studied in statistical mechanics as well as computer science, owing to their
antiferromagnetic nature as well as their ability to encode problems involving hard constraints.
Of particular interest is the setting of the nearest-neighbor hypercube $\Q = \{0,1\}^d$, where understanding the number of independent sets was the original motivator 
for Sapozhenko's famous \emph{graph container method} \cite{container}.
A modern perspective on constraint satisfaction problems and Gibbs measures, both of which feature independent sets as a prime example, is to consider the effect of disorder.
This was first done by Kronenberg and Spinka \cite{ks} who initiated the study of independent sets in a \emph{percolated hypercube} $\Qp$ obtained by removing
each edge from $\Q$ with probability $1-p$ independently.

Recent work by the authors \cite{brcgw} resolved multiple open questions posed by \cite{ks} regarding the distribution
of the number of independent sets in $\Qp$ for $p \geq \f23$, and also elucidated the structure of a typical independent set via a sampling algorithm.
These results relied heavily on the relatively simple behavior of the model for large enough $p$, but as $p$ decreases, typical independent sets become larger
and feature more intricate clustering behavior.
In the present article we overcome many of the challenges presented by this phenomenon and analyze the model for a wider range of $p$.
In particular, for all $p> 0.465$ we obtain a sharp in-probability approximation for the number of independent sets in $\Qp$ in terms of explicit random variables, as well as provide a sampling algorithm. Note that this shows, curiously, that $p = \f12$ does not pose as a natural barrier for this problem unlike in many other problems such as the existence of isolated vertices,
existence of a perfect matching, and Hamiltonicity where it appears as a point of a phase transition. 

A key contribution of this work is the introduction of a new probabilistic framework to handle the clustering behavior for these low values of $p$.
Although our analysis is restricted to $p > 0.465$, our arguments are expected to be helpful for studying this model at even lower values of $p$, and possibly for other related problems.
In addition, although we only consider the total number of independent sets (corresponding to sampling independent sets uniformly)
our techniques should readily generalize to the \emph{partition function} of the hard-core model.
\end{abstract}
\maketitle{}

\setcounter{tocdepth}{1}
\tableofcontents

\parindent=0pt
\parskip=5pt

\input{sections/intro.tex}
\input{sections/iop.tex}
\input{sections/polymerdecomp.tex}
\input{sections/smallpolymers.tex}
\input{sections/nodimers.tex}
\input{sections/dimers.tex}
\input{sections/clt.tex}
\input{sections/sampling.tex}

\appendix

\input{sections/moments.tex}
\input{sections/entropy.tex}

\bibliographystyle{plain}

\bibliography{ref}

\end{document}

%% file: preamble/preamble.tex
\usepackage[letterpaper, hmargin=1.0in, vmargin=1.0in]{geometry}
\usepackage{import}
\usepackage{xifthen}
\usepackage{pdfpages}
\usepackage{transparent}
\usepackage{mleftright}
\def\left{\mleft}
\def\right{\mright}
\usepackage{cancel}
\usepackage[normalem]{ulem}

\usepackage{tikz}
\usetikzlibrary{positioning,shapes.geometric,fit,decorations.pathreplacing}
\pgfdeclarelayer{background}
\pgfsetlayers{background,main}
\tikzset{
  binnode/.style={
    rectangle, rounded corners=6pt,
    draw=black, thick,
    inner sep=5pt,
    outer sep=-2pt,
    minimum width=8mm,
    font=\ttfamily\small,
    align=center
  },
  bitedge/.style={- , thick}
}
\usepackage{float}
\usepackage{mathtools, amsmath, amsfonts, amssymb, amsthm, mathrsfs}
\usepackage{pgffor}
\usepackage{etoolbox}
\usepackage[shortlabels]{enumitem}
\usepackage{xcolor}
\definecolor{CustomBlue}{RGB}{23, 86, 118}
\usepackage[unicode=true, colorlinks=true, linkcolor=CustomBlue, citecolor=CustomBlue]{hyperref}
\hypersetup{pdfauthor={Name}}
\usepackage{setspace}

\usepackage{algorithm, algpseudocodex}

\usepackage[T1]{fontenc}
\newcommand{\sse}{\subseteq}

\newcommand{\R}{\mathbb{R}}

\newcommand{\N}{\mathbb{N}}

\newcommand{\eps}{\varepsilon}

\newcommand{\Ber}{\mathrm{Ber}}

\newcommand{\weakto}{\xLongrightarrow{w}}

\newcommand{\ts}{\widetilde{\theta}}

\newcommand{\Nor}[2]{\mathcal{N}\left(#1, #2\right)}
\newcommand{\Exp}[1]{\exp\left(#1\right)}

\renewcommand{\Box}[1]{\left[#1\right]}
\newcommand{\Rnd}[1]{\left(#1\right)}

\renewcommand{\bar}[1]{\overline{#1}}

\makeatletter
\newcommand{\E}[1][\@nil]{%
	\def\tmp{#1}%
	\ifx\tmp\@nnil
		\mathbb{E}
	\else
		\mathbb{E}\left[#1\right]
\fi}

\renewcommand{\P}[1][\@nil]{%
	\def\tmp{#1}%
	\ifx\tmp\@nnil
		\mathbb{P}
	\else
		\mathbb{P}\Rnd{#1}
\fi}
\makeatother

\newcommand{\Cov}{\mathbf{Cov}}
\newcommand{\Var}{\mathbf{Var}}
\renewcommand{\hat}[1]{\widehat{#1}}

\newcommand{\mnorm}[1]{{\left\vert\kern-0.25ex\left\vert\kern-0.25ex\left\vert #1 
\right\vert\kern-0.25ex\right\vert\kern-0.25ex\right\vert}}

\newcommand{\1}{\mathbf{1}}

\newcommand{\M}{\mathcal{M}}
\renewcommand{\S}{\mathcal{S}}
\newcommand{\Bin}{\mathrm{Bin}}

\newcommand{\cE}{\mathcal{E}}

\newcommand{\f}[2]{\frac{#1}{#2}}
\newcommand{\F}[2]{{\left(\frac{#1}{#2}\right)}}

\newcommand{\df}{\coloneq}

\newcommand{\poly}{\mathrm{poly}}

\renewcommand{\c}{\mathsf{c}}

\newcommand{\A}{\mathbb{A}}

\newcommand{\x}{\bar{x}}

\usepackage{mathtools}
\mathtoolsset{showonlyrefs}

\newtheorem{theorem}{Theorem}[section]
\newtheorem*{theorem*}{Theorem}
\newtheorem{lemma}[theorem]{Lemma}
\newtheorem{proposition}[theorem]{Proposition}
\newtheorem{corollary}[theorem]{Corollary}
\newtheorem{definition}[theorem]{Definition}
\theoremstyle{definition}
\newtheorem{remark}[theorem]{Remark}

\numberwithin{theorem}{section}

\numberwithin{equation}{section}
\numberwithin{figure}{section}

\newcommand{\cC}{\ensuremath{\mathcal C}} 
 
\newcommand{\cH}{\ensuremath{\mathcal H}}

\newcommand{\cO}{\ensuremath{\mathcal O}}

\newcommand{\cX}{\ensuremath{\mathcal X}} 
 
\newcommand{\cZ}{\ensuremath{\mathcal Z}}

\newcommand{\frL}{\mathfrak{L}}

\newcommand{\scrC}{\mathscr{C}}

\newcommand{\scrN}{\mathscr{N}}

\newcommand{\sfN}{\mathsf{N}}
\newcommand{\sfO}{\mathsf{O}}

\def\({\left(}
\def\){\right)}

\newcommand{\e}{\varepsilon}
 
\newcommand{\ind}[1]{\mathbf{1}_{\{#1\}}}

%% file: preamble/additional_macros.tex
\newcommand{\Cnt}{i(Q_{d,p})}
\newcommand{\Cntd}{i(Q_d)}
\newcommand{\Cntcrit}{i(Q_{d,2/3})}
\newcommand{\fr}{\frac}
\newcommand{\Fr}[2]{\Rnd{\frac{#1}{#2}}}
\renewcommand{\P}{\mathbb{P}}
\newcommand{\Cpt}{\mathsf{Cpt}}
\newcommand{\pto}{\xlongrightarrow{\P}}
\newcommand{\dto}{\xLongrightarrow{w}}
\renewcommand{\epsilon}{\varepsilon}
\newcommand{\TooBig}{\mathsf{TooBig}}

\newcommand{\TVbin}[2]{\mathrm{TV}\left( #1, #2\right)}

\let\temp\phi
\let\phi\varphi
\let\varphi\temp

\renewcommand{\emptyset}{\varnothing}

\newcommand{\Pois}{\mathrm{Poisson}}

\newcommand{\Even}{\mathsf{Even}}
\newcommand{\Odd}{\mathsf{Odd}}

\renewcommand{\N}{N_p}
\newcommand{\Q}{Q_d}
\newcommand{\Qp}{Q_{d,p}}

\def\Cpt{\mathsf{Compat}}

\def\Adj{\mathsf{Adj}}
\def\Dim{\mathsf{Dimers}}

\def\cE{\mathsf{Even}}
\def\cO{\mathsf{Odd}}

\newcommand{\psim}{\stackrel{\P}{\sim}}

\newcommand{\polymers}{\mathsf{Polymers}}
\newcommand{\fakepolymers}{\widetilde{\polymers}}

\newcommand{\Sep}[1][]{%
    \def\side{#1}%
    \ifx\side\empty
        \mathsf{PD}%
    \else
        \mathsf{PD}^{\side}%
    \fi
}

\newcommand{\fp}{\mathbf{p}}
\newcommand{\fs}{v}
\newcommand{\fd}{\mathbf{d}}

\usepackage{mleftright}
\renewcommand{\left}{\mleft}
\renewcommand{\right}{\mright}

\newcommand{\fphi}{\tilde{\phi}}

\renewcommand{\tilde}{\widetilde}

\def\tAdj{\widetilde{\Adj}}

\newcommand{\polymerpf}{\mathcal{Z}}
\newcommand{\fakepolymerpf}{\tilde{\mathcal{Z}}}

%% file: sections/intro.tex
%!TEX root =../main.tex

\def\AS{\mathsf{ApproxSampler}}
\def\US{\mathsf{UniformSampler}}
\def\A{\mathsf{A}}
\def\AE{\mathsf{A}^\Even}
\def\AO{\mathsf{A}^\Odd}
\def\M{\mathsf{M}}
\def\ME{\mathsf{M}^\Even}
\def\MO{\mathsf{M}^\Odd}
\def\sbnd{2^d / d^2}
\def\Proper{\mathsf{Proper}}
\def\PolyDec{\mathsf{PD}}

\def\tD{\tilde{\Delta}}
\def\SS{\mathsf{SingletonSampler}}

\section{Introduction}
\label{sec:intro}
This article is our second work investigating independent sets in a disordered hypercube.
Referring the reader to the first paper \cite{brcgw} for a more detailed account of the relevant background,
in this introduction we will be brief, simply introducing the pertinent objects and giving a short record of immediately adjacent results.

We begin with the notion of independent sets in a graph $G=(V,E)$ which are sets of vertices containing no neighbors.
Independent sets are of interest in statistical mechanics where they arise in models of hard-core lattice gases, and the constraint that no two
vertices can be neighbors can be viewed as \emph{antiferromagnetic}, leading to connections with other models of interest such as the antiferromagnetic Ising model.
As is common in statistical mechanics, it is useful to consider a suitable \emph{Gibbs measure} on the collection of independent sets in a graph.
This is known as the \emph{hard-core model}, with a parameter $\lambda > 0$ called the \emph{fugacity}.
Given $G$, the hard-core model assigns to any independent set $I\sse V$ the probability
\begin{equation}\label{hardcoredef342}
	\frac{\lambda^{|I|}}{\cZ},
\end{equation}
where $|I|$ denotes the size of $I$ and $\cZ$, the normalizing constant, is called the partition function.

While the problem of independent sets can be studied for any graph, a particularly interesting case is when the base graph is the
high dimensional lattice given by the hypercube $\Q\df\{0,1\}^d$, endowed with the nearest-neighbor graph structure,
where two vertices $u$ and $v$ are neighbors if and only if they differ at exactly one bit.
Independent sets in $\Q$ correspond to error-detecting binary codes of distance $2$, and this was the initial motivation for studying
such sets in the classic work of Korshunov and Sapozhenko \cite{ksbinary}.
The largest independent sets are $\cE$ and $\cO$, the sets of vertices with an even or odd number of $1$s in their binary-string representation.
Indeed, $\Q$ is bipartite, and $\cE$ and $\cO$ form a bipartition.
These maximal independent sets correspond to the \emph{parity bit} code used at the hardware level in most modern computers.

More recently, the study of independent sets in bipartite graphs like $\Q$ or subgraphs thereof has become especially popular in theoretical computer science,
due to the the well-known \#BIS problem of counting independent sets in bipartite graphs.
This problem is known to be as difficult to solve exactly as the corresponding problem for general graphs (both are \#P-complete \cite{x}),
but whereas the general problem is known also to be intractible to approximation via randomized algorithms \cite{sly}, it is still unknown whether or
not an efficient randomized approximation scheme exists for \#BIS \cite{count}.

Our goal in the present work is to estimate the partition function $\cZ$ of the hard-core model on \emph{random subgraphs} of $\Q$,
obtained via percolation.
As is often the case in such settings, our results concerning the approximation of the partition function will
also allow us to elucidate the structure of a typical independent set in the hard-core model on a percolated hypercube.

To keep the presentation focused, we only consider the case $\lambda = 1$ in this article, although we expect our results to generalize readily
to other values of $\lambda$, as was done in \cite{brcgw}.
When $\lambda=1$, the hard-core model on a graph $G$ is the uniform distribution on all independent sets in $G$, and thus the 
partition function is the total number of such independent sets; this number is denoted by $i(G)$.
We will denote a hypercube that has been percolated with parameter $p$ (meaning that each edge is deleted independently with probability $1-p$)
by $\Qp$.
So our main focus will be to approximate $\Cnt$.

\subsection{Background}

Broadly speaking, as $p$ decreases and more edges are removed, there are fewer constraints for a set to be independent and thus typical independent
sets become larger.
When $p=1$ so that $\Qp =\Q$, much is known, going back to the work of Korshunov and Sapozhenko \cite{ksbinary} who discovered the leading-order behavior
of the number of independent sets:
\begin{equation}
\label{eq:ksleading}
    \Cntd \sim 2^{2^{d-1}} \cdot 2 \cdot \sqrt{e}.
\end{equation}
Intuitively, this approximation holds because most independent sets in $\Q$ are very close to simply being a subset of either the even or odd side
(leading to the factor of $2^{2^{d-1}} \cdot 2$),
with a small number of \emph{defect vertices} in the opposite side.
Each included defect vertex $v$ precludes $d$ non-defect vertices and thus eliminates $d$ binary choices (whether or not to include each neighbor of $v$ in the independent set),
effectively reducing the number of possibile independent sets by a factor of roughly $2^d$.
This means the probability that a uniformly random independent set includes a \emph{particular} defect vertex $v$ is roughly $2^{-d}$.
Using the powerful \emph{graph container} method, it can be shown that the $2^{d-1}$ possible defect vertices do not interact and behave roughly independently,
and thus that the number of defect vertices in a
uniformly random independent set is approximately $\Bin(2^{d-1},2^{-d}) \approx \Pois(\f12)$, leading to the factor of $\sqrt{e}$.

The approximation \eqref{eq:ksleading} was sharpened by Jenssen and Perkins \cite{jp} who provided an algorithm which computes arbitrary lower-order
terms in the approximation.
This is based on a more careful analysis of the defects, which can be formalized via the language of \emph{abstract polymer models} and which can
subsequently be examined via the \emph{cluster expansion}.
A similar technique was used by Kronenberg and Spinka \cite{ks} who initiated the study of $\Cnt$ for $p < 1$.
They adapted the cluster expansion technique to an \emph{annealed} version of the hard-core model on $\Qp$, and were able to exhibit arbitrarily sharp
approximations of the moments $\E[\Cnt^k]$ via an algorithm.
These moment calculations revealed a possible phase transition at $p = \f23$, and they conjectured a central limit theorem
for $\Cnt$ when $p > \f23$ (with a small variance leading to concentration),
as well as a lognormal limit theorem for $p = \f23$, with anticoncentration for $p < \f23$.
Further, their moment calculations allowed them to prove a central limit theorem for $\Cnt$ as long as $p=p_d$ tended to $1$ with $d$ at a certain rate. The behavior for a fixed $p < 1$ remained completely open.

In \cite[Theorem 1.1]{brcgw}, we addressed the conjectures of \cite{ks},
exhibiting explicit deterministic quantities $M_{d,p}$ and $S_{d,p}$ such that, for fixed $p \in (\f23, 1)$,
\begin{equation}
\label{eq:clt23}
    \fr{\Cnt - M_{d,p}}{S_{d,p}} \weakto \Nor{0}{1},
\end{equation}
and for $p = \f23$,
\begin{equation}
\label{eq:lognormal23}
    \fr{\Cnt}{M_{d,\f23}} \weakto e^{W^\cE / \sqrt{2}} + e^{W^\cO / \sqrt{2}},
\end{equation}
where $W^\cE, W^\cO \stackrel{i.i.d.}{\sim} \Nor{0}{1}$.
Here and in the remainder of this article, $\weakto$ denotes convergence in distribution as $d \to \infty$.

Both \eqref{eq:clt23} and \eqref{eq:lognormal23} derive from a more fundamental result of \cite{brcgw} which gives an
in-probability approximation for $\Cnt$ as follows, for $p \geq \f23$:
\begin{equation}
\label{eq:approxinprob}
    \fr{\Cnt}{2^{2^{d-1}}} \psim e^{\Psi^\cE} + e^{\Psi^\cO},
\end{equation}
where $\Psi^\cE$ and $\Psi^\cO$ are random variables which can be written explicitly in terms of the degrees of vertices in $\Qp$,
with each vertex contributing once in a sum of independent terms.
Here and in the rest of the article, $A \psim B$ means that $\fr{A}{B} \pto 1$, where $\pto$ denotes convergence in probability
with respect to the randomness of the percolation configuration, as $d \to \infty$.
Moreover, $\Psi^\cE$ and $\Psi^\cO$ satisfy a joint central limit theorem, converging after proper rescaling to a pair of
independent normal random variables.
The phase transition at $p = \f23$ is explained by the variance of $\Psi^\cE$ and $\Psi^\cO$ (which can easily be explicitly computed)
decreasing to $0$ for $p > \f23$, staying constant for $p = \f23$, and increasing to $\infty$ for $p < \f23$.

The approximation \eqref{eq:approxinprob} should be seen as an extension of \eqref{eq:ksleading} to $p < 1$, and at its core is the same fundamental
mechanism, where independent sets are close to the \emph{ground states} (uniformly random subsets of the even or odd side) with a small number
of defects.
The regime $p > \f23$ is exactly where the defects are still sparse enough that they can be
effectively approximated by sampling them \emph{with replacement}, as no repeat samples will arise
(via an analysis akin to the classical \emph{birthday problem}).
When $p \leq \f23$, this sampling scheme fails due to the presence of defect vertices with outsized probability of occurring, which
also leads to the failure of concentration and the central limit theorem \eqref{eq:clt23}.

In addition to results about the \emph{number} of independent sets, the approximation \eqref{eq:approxinprob} allowed 
to obtain various structural results about typical independent sets in $\Qp$, sampled uniformly \cite[Theorem 1.2 and Proposition 1.4]{brcgw}.
For $p \geq \f23$ this approximation scheme allowed us to conclude that the number of defect vertices in a uniformly random
independent set in $\Qp$ is approximately $\Pois(\f12(2-p)^d)$.

\subsection{Results}
\label{sec:intro_results}

In our previous work \cite{brcgw}, we focused on the case $p \geq \f23$ due to various technicalities that arise as $p$ decreases and the size of a typical
independent set increases.
However, one very interesting question asked by \cite{ks} was about the behavior of this model near $p = \f12$ in light of the natural emergence of the latter as an important point of phase transition for many other features of the percolated hypercube $\Qp$ such as the existence of isolated vertices,
existence of a perfect matching, and Hamiltonicity \cite{be, bo, bp, ch, ee}.

Motivated by the above, in the present article we introduce many new techniques which are organized via a convenient probabilistic framework for studying this problem.
We expect that this will be useful for continued study of this model for lower values of $p$, and that it may also provide insights into new ways to study other related problems as well.

Notably, as a consequence of our results, \textit{a posteriori}, it turns out that $p = \f12$ does not appear to be a point of phase transition for any qualitative behavior
about independent sets in $\Qp$, and we obtain essentially the same result just above $p = \f12$ as just below.
Instead, $p = \f12$ only corresponds to a change in the way we may present the formulas involved in our result (see Remark \ref{rmk:onehalf}).
Of more interest and relevance for the qualitative behavior of this model, as evidenced by our results to be presented below,
is the phase transition at $p = 2 - \sqrt{2} \approx 0.586$, below which the defects start exhibiting
nontrivial clustering behavior.
This leads to a loss of independence between defects, and as such, our analysis in this article depends crucially on the interactions between
nearby vertices in $\Qp$.

We are now in a position to state our first main result.
We remind the reader that here and throughout the rest of the article,
$A \psim B$ means that $\fr{A}{B} \pto 1$, where $\pto$ denotes convergence in probability (with respect to randomness of the percolation $\Qp$ of the hypercube)
as $d \to \infty$.
In addition, we will occasionally use phrases like ``$o(1)$-in-probability'' to describe a random variable $A$ with $A \pto 0$.
Further, when we say that a collection of random variables $\{ A_S \}_{S \in \mathcal{S}}$ indexed by some collection $\mathcal{S}$
is ``\emph{uniformly} $o(1)$-in-probability'' if there is some random variable $M \pto 0$ such that $|A_S| \leq M$ for all $S \in \mathcal{S}$.
Similarly, we will say that $A_S \psim B_S$ uniformly if $\fr{A_S}{B_S} - 1$ is uniformly $o(1)$-in-probability.

\begin{theorem}
\label{thm:main} 
There is some $\gamma > 0$ such that for $p > \fr{1}{2} - \gamma$ the following holds.
There are random variables $\Psi^\cE, \Psi^\cO$, which can be expressed explicitly as functions of
the percolation configuration, such that
\begin{equation}\label{eq:main_approx}
    \fr{\Cnt}{2^{2^{d-1}}} \psim e^{\Psi^\cE} + e^{\Psi^\cO}.
\end{equation}
Moreover, there are explicit deterministic quantities $\mu = \mu_{d,p}$ and $\sigma = \sigma_{d,p}$ (defined below)
such that
\begin{equation}\label{eq:main_clt}
    \Rnd{ \fr{\Psi^\cE - \mu}{\sigma}, \fr{\Psi^\cO - \mu}{\sigma} }
    \dto \Nor{0}{1} \otimes \Nor{0}{1}.
\end{equation}
\end{theorem}

Let us quickly see some consequences of this theorem, for which it is helpful to know that the value of $\sigma$ in the statement is given by
\begin{equation}
\label{eq:sigmadefintro}
    \sigma^2 \coloneqq \fr{1}{2} \Rnd{2 - \fr{3}{2} p}^d.
\end{equation}
Notice that for $p > \f23$ we have $\sigma \to 0$, while if $p < \f23$ we have $\sigma \to \infty$.
The above suggests that for $p < \f23,$ as is the case for lognormal variables
with large fluctuations in the exponent, {the typical value of $\Cnt$ will be much smaller than its mean}. This can indeed be made precise using sharp estimates about the mean from \cite{ks}.
In addition, for much the same reason, there is no way to shift and scale $\Cnt$ by deterministic quantities to obtain a scaling limit for $p < \f23$.
However, we do obtain a scaling limit for the \emph{free energy}, which is defined as $\log \Cnt$:
from Theorem \ref{thm:main} and \eqref{eq:sigmadefintro} it is easy to see that
\begin{equation}
    \fr{\log \Cnt - \log 2^{2^{d-1}} - \mu}{\sigma} \dto \max\left\{ W^\cE , W^\cO \right\}
\end{equation}
for $\f12 - \gamma < p < \f23$,
where $W^\cE, W^\cO \stackrel{i.i.d.}{\sim} \Nor{0}{1}$.

We now briefly describe the variables $\Psi^\cE$ and $\Psi^\cO$ appearing in the statement of
Theorem \ref{thm:main}, as well as the deterministic quantity $\mu$ (the quantity $\sigma$ having just been defined in \eqref{eq:sigmadefintro} above).
{Further elaborations on the form of these quantities arising from relatively simple moment calculations are provided later in Section \ref{sec:clt}.}

Recall from the previous subsection that $\N(v)$ denotes the (random) number of neighbors of a single vertex $v$ in $\Qp$.
We will need to consider
sets containing two distinct vertices which are ``adjacent'' in the sense that they share a neighbor in the opposite side of the (unpercolated) hypercube. We will call a set of vertices 
\emph{2-linked} if they are connected through the above adjacency structure. 

Note that if two vertices share a neighbor in $\Q$, they must share exactly two neighbors as demonstrated by the following diagram.
\begin{equation}
\begin{tikzpicture}[node distance=8mm and 8mm, every edge/.style={bitedge, shorten >=0pt, shorten <=0pt}, baseline=(current bounding box.center)]
  \node[binnode, fill=gray!50] (00) {00};
  \node[binnode, right=of 00, fill=gray!50] (11) {11};
  \node[binnode, below=of 00, xshift=2mm, fill=white] (01) {01};
  \node[binnode, below=of 11, xshift=-2mm, fill=white] (10) {10};

  \begin{pgfonlayer}{background}
    \draw (00) -- (10);
    \draw (00) -- (01);
    \draw (10) -- (11);
    \draw (01) -- (11);
  \end{pgfonlayer}
\end{tikzpicture}
\label{dimerpicture}
\end{equation}
We call such subsets (e.g. the pair of shaded nodes in the top row above)
\emph{dimers}, and denote by $\Dim^\cE$ the collection of such subsets of the even side, and similarly for $\Dim^\cO$.
For a dimer $\{u,v\}$, we write $\N(\{u,v\})$ for the (random) number of vertices which neighbor either $u$ or $v$ in $\Qp$; note that $\N(\{u,v\}) \neq \N(u) + \N(v)$ in general.
With these definitions, we can write down the formula for $\Psi^\cE$, the formula for $\Psi^\cO$
being identical but with $\cE$ replaced by $\cO$ throughout.
\begin{equation}
\label{eq:psidefintro}
    \Psi^\cE \coloneqq \sum_{v \in \cE} \log\Rnd{1 + 2^{-\N(v)}}
    + \sum_{\{u,v\} \in \Dim^\cE} \Rnd{2^{-\N(\{u,v\})} - 2^{-\N(u)-\N(v)}}.
\end{equation}

Next, we turn to the deterministic quantity $\mu$.
Specifically, we first define
\begin{equation}
\label{eq:mu1defintro}
    \mu_1 \coloneqq \fr{1}{2} \Rnd{2 - p}^d - \fr{1}{4} \Rnd{2 - \fr{3}{2}p}^d + \fr{1}{6} \Rnd{2 - \fr{7}{4}p}^d.
\end{equation}
This is a proxy for the mean of the first sum in \eqref{eq:psidefintro}.
Additionally define
\begin{align}
\label{eq:mu2defintro}
    \mu_2 &\coloneqq \fr{d(d-1)}{4} \Fr{(2-p)^2}{2}^d \Fr{1+(1-p)^2}{2-p}^2, \\
\label{eq:mu2tildedefintro}
    \tilde{\mu}_2 &\coloneqq \fr{d(d-1)}{4} \Fr{(2-p)^2}{2}^d.
\end{align}
Then $\mu_2 - \tilde{\mu}_2$ is the mean of the second sum in \eqref{eq:psidefintro}.
So, we set
\begin{equation}
\label{eq:mudefintro}
    \mu \coloneqq \mu_1 + \mu_2 - \tilde{\mu}_2.
\end{equation}
A more complete understanding of these quantities is presented later in Section \ref{sec:clt}.

As mentioned already, our arguments also lead to a sampling scheme which is presented next.
For convenience, we allow the algorithm to fail at various points, but we will show that it fails with very low probability.
Additionally, here and throughout the paper we blur the lines between a collection of singletons and dimers considered as separate polymers,
and the union of such a collection (which is just a set of vertices), as this should not present a scope for confusion.

\begin{definition}[$\AS$]
\label{def:approxsampler_intro}
	Define the random independent set $\hat{I}$ in $\Qp$ via the following procedure.
	\begin{enumerate}
		\item Let $\cH$ be $\cE$ or $\cO$, with probabilities proportional to $e^{\Psi^\cE}$ and $e^{\Psi^\cO}$ respectively.
		\item Let $\tilde{S}$ denote a subset of $\cH$ which includes each $v \in \cH$ independently with probability
		$\fr{2^{-\N(v)}}{1 + 2^{-\N(v)}}$.
		\item If $\tilde{S}$ contains any triple of vertices which is \emph{2-linked},
        terminate the algorithm with a result of failure.
		Otherwise, remove all dimers from $\tilde{S}$ and let $S_1$ denote the resulting set of well-separated singletons (which share no neighbors).
		\item Let $S_2$ denote a collection of dimers constructed by including each $\{u,v\} \in \Dim^\cH$ independently with probability
		$\fr{2^{-\N(\{u,v\})}}{1 + 2^{-\N(\{u,v\})}}$.
		\item If any neighbor of a dimer in $S_2$ also neighbors a vertex in $S_1$, {or if any pair of dimers in $S_2$ overlap or share a neighbor,}
        terminate the algorithm with a result of failure.
		Otherwise, let $\hat{S} = S_1 \cup S_2$.
		\item Let $\hat{I}$ denote the independent set with $\hat{I} \cap \cH = \hat{S}$ and with every vertex in $\Q \setminus \cH$
		which does not neighbor $\hat{S}$ included independently with probability $\fr{1}{2}$ each.
	\end{enumerate}
\end{definition}

Note that there are two levels of randomness here as $\AS$ is a \emph{random algorithm} even when the disorder of the percolation configuration $\Qp$
is fixed, which is often termed as quenching the disorder.
For simplicity of notation, we will refer to both levels of randomness using $\P$, but which measure is being referred to will be clear from context. 
This convention is in full force in the following statement of our second main result, which states that $\AS$ gives a good approximation for a uniformly random
independent set in $\Qp$.

\begin{theorem}
\label{thm:sampler}
Let $p > \f12 - \gamma$, where $\gamma > 0$ is as in Theorem \ref{thm:main}.
Then
\begin{equation}
    \P[\AS \text{ fails}] \pto 0.
\end{equation}
Additionally, if $\hat{I}$ denotes the output of $\AS$ conditioned on not failing and $I$ denotes a uniformly random independent set in $\Qp$,
then $I$ and $\hat{I}$ may be coupled so that
\begin{equation}
    \P[I \neq \hat{I}] \pto 0.
\end{equation}
\end{theorem}

\subsection{Future directions}

As evidenced from the statement of our results in the previous subsection, the relevant ``defects'' of independent sets
in the parameter regime we consider are all either single vertices or dimers, consisting of two nearby vertices (which
share a neighbor in the non-defect side).
However, as $p$ decreases even further and typical independent sets get larger, the defect side will contain
more and more \emph{polymers} of higher order, {which are clusters of defect vertices that are \emph{2-linked}}.
Specifically, for $p \leq 2 - 2^{2/3} \approx 0.413$, size-three polymers (trimers) begin to appear, and in general order-$k$
polymers begin to appear for $p \leq 2 - 2^{(k-1)/k}$.

Broadly speaking, we introduce a framework for separating out the effect of singletons and dimers, and for analyzing their
interactions. This strategy of separating polymers of different sizes and estimating their interactions ought to yield dividends even while studying lower values of $p$ where higher-order polymers are relevant.
However, as the order of polymers increases, so does the number of possible shapes that the polymers form.
For instance, while every dimer looks essentially the same, geometrically speaking, as the pair of shaded vertices in top row of diagram \eqref{dimerpicture},
there are already two different nonisomorphic types of trimers, which look like the shaded vertices in the top row of either
\begin{center}
\begin{tikzpicture}[node distance=8mm and 8mm, every edge/.style={bitedge, shorten >=0pt, shorten <=0pt}, baseline=(current bounding box.center)]
  \node[binnode, fill=gray!50] (101) {101};
  \node[binnode, left=of 101, fill=gray!50] (110) {110};
  \node[binnode, right=of 101, fill=gray!50] (011) {011};
  \node[binnode, below=of 110, xshift=-5mm, fill=white] (100) {100};
  \node[binnode, below=of 101, xshift=-9mm, fill=white] (010) {010};
  \node[binnode, below=of 011, xshift=-13mm, fill=white] (001) {001};
  \node[binnode, below=of 011, xshift=5mm, fill=white] (111) {111};

  \begin{pgfonlayer}{background}
    \draw (101) -- (100);
    \draw (101) -- (001);
    \draw (110) -- (100);
    \draw (110) -- (010);
    \draw (011) -- (010);
    \draw (011) -- (001);
    \draw (110) -- (111);
    \draw (101) -- (111);
    \draw (011) -- (111);
  \end{pgfonlayer}
\end{tikzpicture}
\qquad
or
\qquad
\begin{tikzpicture}[node distance=8mm and 8mm, every edge/.style={bitedge, shorten >=0pt, shorten <=0pt}, baseline=(current bounding box.center)]
  \node[binnode, fill=gray!50] (0000) {0000};
  \node[binnode, left=of 0000, fill=gray!50] (1100) {1100};
  \node[binnode, right=of 0000, fill=gray!50] (0011) {0011};
  \node[binnode, below=of 1100, xshift=-2mm, fill=white] (1000) {1000};
  \node[binnode, below=of 0000, xshift=-7mm, fill=white] (0100) {0100};
  \node[binnode, below=of 0000, xshift=7mm, fill=white] (0010) {0010};
  \node[binnode, below=of 0011, xshift=2mm, fill=white] (0001) {0001};

  \begin{pgfonlayer}{background}
    \draw (0000) -- (1000);
    \draw (0000) -- (0100);
    \draw (0000) -- (0010);
    \draw (0000) -- (0001);
    \draw (1100) -- (1000);
    \draw (1100) -- (0100);
    \draw (0011) -- (0010);
    \draw (0011) -- (0001);
  \end{pgfonlayer}
\end{tikzpicture};
\end{center}
i.e.\ the induced adjacency structure of the ``sharing a neighbor'' relation is either a triangle or a path.
For larger polymers there are even more different nonisomorphic types with increasingly intricate adjacency structures between the vertices.
As may be predicted, then, as $p$ decreases to $0$ the number of different relevant terms in sums of the form \eqref{eq:psidefintro} will increase
rapidly.

In addition, even before trimers appear (which occurs for $p \leq 2 - 2^{2/3}$), the interactions between singletons and dimers become relevant.
Indeed, our proof in the present work aims to cover the complete range of $p$ for which the singletons and dimers behave roughly independently,
but not the full range of $p > 2 - 2^{2/3}$.
To be more specific, although the threshold in our theorem is only written as $\f12 - \gamma$ for some $\gamma > 0$,
one may check, numerically or otherwise, that $\f12 - \gamma \approx 0.465$,
which is strictly greater than $2 - 2^{2/3} \approx 0.413$.
{Although we do not prove it, we do expect that this range of $p$ is optimal, in the sense that for $p \in (2 - 2^{2/3}, \f12 - \gamma]$,  the singletons and dimers will fail to be independent. It remains an interesting problem to capture the precise nature of the interaction. }

\subsection{Acknowledgements}
SG was partially supported by NSF Career grant-1945172. 
VW was partially supported by the NSF Graduate Research Fellowship grant DGE 2146752.  SG
learnt about this problem from Gal Kronenberg’s talk on \cite{ks} while attending the workshop titled
‘Bootstrap Percolation and its Applications’ at Banff in April 2024. He thanks the speaker as well
as the organizers.

%% file: sections/iop.tex
%!TEX root =../main.tex

\section{Idea and proof of main result}
\label{sec:iop}
In this section we provide a brief overview of the main ideas.

As has already been emphasized a few times, a key observation going back to \cite{ksbinary}
is that each independent set can be separated into its intersections with the even side
and odd side of the hypercube, yielding a convenient representation for $\Cnt$.

Since there are no edges between even vertices, every subset $S \sse \cE$ is a valid choice for the even side of
an independent set.
For any $S \sse \cE$, the number of independent sets $I \sse \Qp$ with $I \cap \cE = S$ can be written explicitly as
\begin{equation}
\label{eq:numindS}
    2^{2^{d-1}-\N(S)},
\end{equation}
where $\N(S)$ denotes the number of vertices in $\Odd$ which are neighbors of $S$ in $\Qp$.
Indeed, any non-neighbor of $S$ can be included in such an independent set, and since there are no edges between odd
vertices, any subset $T$ of the odd non-neighbors of $S$ yields a valid independent set $I = S \cup T$.
There are $2^{d-1} - \N(S)$ such odd non-neighbors of $S$, yielding the formula \eqref{eq:numindS}.
This yields the exact representation
\begin{equation}
\label{eq:exactrep}
    \fr{\Cnt}{2^{2^{d-1}}} = \sum_{S \sse \cE} 2^{-\N(S)}
\end{equation}
for the total number of independent sets in $\Qp$.

Now, for well-separated sets $S \sse \cE$, where no vertices of $S$ share a neighbor in $\cO$ (even in the unpercolated hypercube $\Q$),
the expression $2^{-\N(S)}$ factorizes as a product of contributions from each individual vertex:
\begin{equation}
\label{eq:factorization}
    2^{-\N(S)} = \prod_{v \in S} 2^{-\N(v)},
\end{equation}
since the neighborhoods of $v \in S$ do not overlap.
Moreover, the variables $2^{-\N(v)}$ depend on disjoint subsets of the edges in the percolation configuration,
meaning that the factors of \eqref{eq:factorization} are independent (in fact, the entire family of random variables $\{ \N(v) : v \in \cE \}$ is mutually independent,
and similarly for $\cO$).

Of course, there are many sets $S \sse \cE$ which are \emph{not well-separated}, meaning that vertices of $S$
share neighbors in $\cO$ (before percolating).
For general sets $S$, one could instead form a factorization like \eqref{eq:factorization} but where the factors correspond
to \emph{subsets} of $S$ which are well-separated from each other.
However, the complexity of such factorizations quickly spirals out of control when $S$ is large, since the form
of each factor depends on the internal structure of the corresponding subset.

A key feature of this model which allows us to overcome this issue is that when $p$ is bounded away from $0$,
the model is effectively in a \emph{very low-temperature regime}.
More precisely, most independent sets $I$ are very close to one of the two ``ground states'' meaning that
\emph{exactly one of $I \cap \cE$ and $I \cap \cO$ is large}, and the other is very small in comparsion.

This allows us to work separately with two classes of independent sets, those with small even side and those with
small odd side.
If $S \sse \cE$ is small, we can decompose it as a union of \emph{small} subsets which are well-separated from
each other, which are easier to categorize.
Our approach then relies on analyzing the contributions of these small separated subsets.

In the next few sections, we briefly describe the ingredients and then combine them to prove Theorem \ref{thm:main}.

\subsection{Polymer models}
\label{sec:iop_polymer}

A convenient framework for reasoning about a set $S$ as a union of small subsets which are well-separated from each other
is provided by the language of \emph{polymer models}.
Broadly speaking, a polymer model is a probability distribution on \emph{polymer configurations}, which are collections of
small units called \emph{polymers} satisfying a compatibility condition.
The defining feature is that the probability of a polymer configuration is proportional to a product of given \emph{weights},
one for each polymer in the configuration.

In our case, polymers will be subsets $\fp \sse \cE$ (or $\cO$) {which are \emph{2-linked}, recalling that it means that they are
connected under the relation of sharing a common neighbor in $\Q$.}
We also require that polymers are \emph{small} in a sense to be defined later, in Section \ref{sec:polymerdecomp}.
Two polymers are said to be \emph{compatible} if their neighborhoods in $\Q$ do not intersect.
We denote by $\Sep^\cE$ the set of configurations of pairwise-compatible polymers $\fp \sse \cE$, and similarly for $\cO$
($\Sep$ stands for polymer-decomposable).
The weight of each polymer $\fp$ is given by 
\begin{equation}
\label{eq:phidef}
    \phi_\fp \coloneqq 2^{-\N(\fp)},
\end{equation}
{so that the weight of a configuration of compatible polymers $S \in \Sep^\cE$ is}
\begin{equation}
    \prod_{\fp \in S} \phi_\fp = \prod_{\fp \in S} 2^{-\N(\fp)},
\end{equation}
which is the same as the weight of $S$ in the sum \eqref{eq:exactrep}, when it is considered as a \emph{union} of its
polymers (the subtle distinction between a \emph{configuration (set) of polymers} and a set of vertices as a \emph{union of polymers}
will be ignored as it causes no confusion).
The partition function (normalization constant) for this polymer model is
\begin{equation}
\label{eq:polymerpfdef}
    \polymerpf^\cE \coloneqq \sum_{S \in \Sep^\cE} \prod_{\fp \in S} \phi_\fp,
\end{equation}
and similarly there is a partition function $\polymerpf^\cO$ on the odd side.
The first step of our proof, as in various previous works on independent sets \cite{brcgw,jp,ks},
is to reduce the study of $\Cnt$ to the study of these polymer model partition functions.

\begin{proposition}
\label{prop:polymerdecomp_iop}
For $p \in (0,1)$,
\begin{equation}
    \fr{\Cnt}{2^{2^{d-1}}} \psim \polymerpf^\cE + \polymerpf^\cO.
\end{equation}
\end{proposition}

This proposition is a straightforward consequence of one the main conclusions of \cite{ks} and was already presented in our previous work \cite[Lemma 4.8]{brcgw}, but in the present article, specifically as
Proposition \ref{prop:polymerdecomp} in Section \ref{sec:polymerdecomp} below, we present a somewhat more streamlined proof
than has previously appeared.
The key input is an \emph{annealed version} of the statement from \cite{ks}, where it (along with various extensions) was used to derive
arbitrarily accurate estimates of the \emph{moments} of $\Cnt$ via the \emph{cluster expansion} of polymer models,
which will be discussed briefly in Section \ref{sec:polymerdecomp}.
However, it is worth emphasizing that the present work does not make use of the cluster expansion beyond this annealed input;
instead we analyze the quenched polymer model partition functions directly.

Due to the structural similarity between Proposition \ref{prop:polymerdecomp_iop} and our goal, which is Theorem \ref{thm:main},
it may be tempting to simply define $\Psi^\cE = \log \polymerpf^\cE$ and similarly for $\cO$.
However, these quantities are not easy to understand a priori, letting alone their joint distribution, and we will instead obtain a good approximation
of $\log \polymerpf^\cE$ by a simpler and more explicit quantity (already spelled out in \eqref{eq:psidefintro})
which is more amenable to analysis.

\subsection{Main step: reducing to a simpler polymer model}

In this subsection we do not discuss the joint behavior of $\polymerpf^\cE$ and $\polymerpf^\cO$, deciding to
focus only on $\polymerpf^\cE$.
Everything in this section also goes through with $\cE$ replaced by $\cO$.

In order to analyze $\polymerpf^\cE$, a central observable will be  the partition function of
a simpler polymer model under which \emph{each vertex behaves independently}.
It is worth pointing out that this simpler partition function already serves as a strong proxy for  $\polymerpf^\cE$ if $p > 2-\sqrt 2 \approx 0.586$ but not beyond. Thus the simpler object does not by itself suffice to approximate $\polymerpf^\cE$. Rather, it serves as an important base model to compare to. This quantitative comparison for $p \leq 2 - \sqrt{2}$ is conceptually one of the major contributions of the present work.

To define the simpler model, we replace $\phi_\fp = 2^{-\N(\fp)}$ by
\begin{equation}
\label{eq:fakephidef}
    \fphi_\fp \coloneqq \prod_{v \in \fp} \phi_v,
\end{equation}
where again $\phi_v = 2^{-\N(v)}$.
This means the weight of a polymer configuration $S$ is simply
\begin{equation}
    \prod_{\fp \in S} \fphi_\fp = \prod_{\fp \in S} \prod_{v \in \fp} \phi_v = \prod_{v \in S} \phi_v
\end{equation}
(recall from the previous subsection our identification of a \emph{set} of compatible polymers and the \emph{union} of those same polymers).
We also remove the restriction that polymers be small in this polymer model, so that \emph{any} set $S \sse \cE$
is given the above weight.
The partition function of this model is thus
\begin{equation}
\label{eq:fakepolymerpfdef}
    \fakepolymerpf^\cE \coloneqq \sum_{S \sse \cE} \prod_{v \in S} \phi_v = \prod_{v \in \cE} (1 + \phi_v).
\end{equation}
In other words, to sample a set $S \sse \cE$ from this polymer model, one includes each vertex $v \in \cE$ independently
with probability $\fr{\phi_v}{1 + \phi_v}$.
The simplicity of the formula for $\fakepolymerpf^\cE$ makes it amenable to analysis. 

In light of the above, the key innovation in this paper is to develop a characterization of the ratio $\polymerpf^\cE / \fakepolymerpf^\cE$.

The first step in this analysis is to reduce both $\polymerpf^\cE$ and $\fakepolymerpf^\cE$ by eliminating terms 
with polymers which are too large.
To this end, for each $k \in \mathbb{N}$ define $\Sep^\cE_{\leq k}$ to be the set of configurations of pairwise-compatible
polymers $\fp \sse \cE$ with at most $k$ vertices, and define
\begin{equation}
\label{eq:smallpolymerpfdef}
    \polymerpf^\cE_{\leq k} \coloneqq \sum_{S \in \Sep^\cE_{\leq k}} \prod_{\fp \in S} \phi_\fp,
    \qquad \text{and} \qquad
    \fakepolymerpf^\cE_{\leq k} \coloneqq \sum_{S \in \Sep^\cE_{\leq k}} \prod_{\fp \in S} \fphi_\fp.
\end{equation}

As has already been indicated, in the regime of $p$ addressed in this article, polymers of sizes more than $2$ do not appear,
which is formalized by the following proposition.

\begin{proposition}
\label{prop:smallpolymers_iop}
For $p > 2 - 2^{2/3} \approx 0.413$,
\begin{equation}
    \polymerpf^\cE \psim \polymerpf^\cE_{\leq 2}
    \qquad \text{and} \qquad
    \fakepolymerpf^\cE \psim \fakepolymerpf^\cE_{\leq 2}.
\end{equation}
\end{proposition}

This will be proved in two parts, as Propositions \ref{prop:smallpolymers} and \ref{prop:smallpolymers_full}
in Section \ref{sec:smallpolymers} below.
With Proposition \ref{prop:smallpolymers_iop} in place, we can think of both polymer models as having two parts:
first, a collection $S_1$ of well-separated singletons, and second, a collection $S_2$ of well-separated dimers
\emph{which are also compatible with $S_1$}.

While in principle, the law of $S_2$ given $S_1$ depends on the instantiation of $S_1$, one of the key observations in this paper is that there is some threshold $\f12 - \gamma$ strictly less than $\f12$
such that
\begin{equation}
\textbf{when $p > \tfrac{1}{2} - \gamma$, the distributions of $S_1$ and $S_2$ are in fact \emph{approximately independent}}
\end{equation}
under both polymer models.
Numerically, this threshold is around $0.465$.

In order to state this approximate independence result, we define for each $k \in \mathbb{N}$
\begin{equation}
\label{eq:polymerpfsubkdef}
    \polymerpf^\cE_k \coloneqq \sum_{S \in \Sep^\cE_k} \prod_{\fp \in S} \phi_\fp
    \qquad \text{and} \qquad
    \fakepolymerpf^\cE_k \coloneqq \sum_{S \in \Sep^\cE_k} \prod_{\fp \in S} \fphi_\fp,
\end{equation}
where $\Sep^\cE_k$ denotes the set of configurations of pairwise-compatible $k$-polymers.
Note that we have $\fakepolymerpf^\cE_1 = \polymerpf^\cE_1$ since $\phi_v = \fphi_v$ for singletons, but
this does not extend to $k > 1$.
The following proposition exhibits the aforementioned asymptotic independence of the singletons and dimers.

\begin{proposition}
\label{prop:independence_iop}
There is some $\gamma > 0$ such that for all $p > \fr{1}{2} - \gamma$, we have
\begin{equation}
    \polymerpf^\cE_{\leq 2} \psim \polymerpf^\cE_1 \cdot \polymerpf^\cE_2
    \qquad \text{and} \qquad
    \fakepolymerpf^\cE_{\leq 2} \psim \polymerpf^\cE_1 \cdot \fakepolymerpf^\cE_2.
\end{equation}
\end{proposition}

To show this, we must prove that the restriction that the dimer part $S_2$ be well-separated from the singleton part $S_1$
is superfluous.
In other words, when one first samples the singleton set $S_1$ from the polymer model corresponding to $\polymerpf^\cE_1$,
then by simply sampling the collection of dimers $S_2$ from the polymer model with partition function
$\polymerpf^\cE_2$, one obtains $S_2$ which is well-separated from $S_1$ with high probability.
This follows by developing an understanding of the probabilistic behavior of the \emph{set of dimers adjacent to $S_1$} reliant on certain large deviation estimates. 

Finally, the dimer part $S_2$ can itself be approximated by independently including every possible dimer, with the
appropriate probability.
In other words, the requirement that $S_2$ be well-separated is also superfluous; in formulas, this means
\begin{equation}
\label{eq:dimerssepiop}
    \polymerpf^\cE_2 \psim \sum_{S_2 \sse \Dim^\cE} \prod_{\fd \in S_2} \phi_\fp = \prod_{\fd \in \Dim^\cE} (1 + \phi_\fd).
\end{equation}
Here we recall from just above \eqref{eq:psidefintro} the notation $\Dim^\cE$ for the collection of dimers (i.e.\ size-$2$ polymers)
in the even side (note that we typically use $\fd$ to denote a dimer, instead of $\fp$ which denotes an arbitrary polymer).
Now, we can go one step further than \eqref{eq:dimerssepiop} and approximate each factor of $(1 + \phi_\fd)$
by $e^{\phi_\fd}$.
For the parameter regime we consider, the error in this approximation is small enough that it does not accumulate
across the product over dimers.
Moreover, the same analysis goes through for $\fakepolymerpf^\cE_2$.
Thus, defining
\begin{equation}
\label{eq:deltadef}
    \Delta^\cE = \sum_{\fd \in \Dim^\cE} \phi_\fd
    \qquad \text{and} \qquad
    \tilde{\Delta}^\cE = \sum_{\fd \in \Dim^\cE} \fphi_\fd,
\end{equation}
we arrive at the following lemma.

\begin{lemma}
\label{lem:dimersdontcollide_iop}
For $p > 0.391$, 
\begin{equation}
    \polymerpf^\cE_2 \psim e^{\Delta^\cE}
    \qquad \text{and} \qquad
    \fakepolymerpf^\cE_2 \psim e^{\tilde{\Delta}^\cE}.
\end{equation}
\end{lemma}

To summarize at a very high level, Lemma \ref{lem:dimersdontcollide_iop} and Proposition \ref{prop:independence_iop}
capture the phenomenon that if there are not many dimers present, then with high probability none of them will ``collide'' with each
other or with any of the singletons present in the defect side of a uniformly-sampled independent set, allowing for a
sampling strategy which first produces only the singletons, then afterwards produces the dimers.
We expect this broad strategy of hierarchically sampling larger and larger polymers to be useful to analyze the model even for values of $p$
below those addressed in this paper.

\begin{remark}
\label{rmk:worstvsavg}
It is worth emphasizing that the strategy of heirarchically sampling the dimers after the singletons does not work for the \emph{worst-case}
singleton sets that could appear in the first round of sampling, at least not for all $p$ under consideration in the present article.
Indeed, as will be discussed further in Section \ref{sec:dimers_firststrategy}, the probability that the sampled dimer set avoids
the worst-case singleton set converges to $0$ only when $p$ is bigger than approximately $0.549$.
To cover the full regime of $p > \f12 - \gamma$, we instead present an \emph{average-case}
analysis in Sections \ref{sec:dimers_secondstrategy} and \ref{sec:dimers_adj}, averaging over the sampled singleton set.
\end{remark}

In any case, combining Propositions \ref{prop:smallpolymers_iop} and \ref{prop:independence_iop} as well as Lemma
\ref{lem:dimersdontcollide_iop}, we obtain our final in-probability approximation of each polymer partition function.

\begin{corollary}
\label{cor:approxinprob_iop} 
There is some $\gamma > 0$ such that for all $p > \fr{1}{2} - \gamma$, we have
\begin{equation}
    \polymerpf^\cE \psim \fakepolymerpf^\cE \cdot e^{\Delta^\cE - \tilde{\Delta}^\cE}.
\end{equation}
\end{corollary}

Of course, the same holds with $\cE$ replaced by $\cO$ throughout.

Before proceeding, we remark that our actual proof strategy slightly reorganizes Proposition \ref{prop:independence_iop}
and Lemma \ref{lem:dimersdontcollide_iop} and these particular statements do not appear in the body of the work.
However, the main important result is actually Corollary \ref{cor:approxinprob_iop} itself, which will be proved directly
via Proposition \ref{prop:dimers} below.
If the reader wishes to use the statements of Proposition \ref{prop:independence_iop} or Lemma \ref{lem:dimersdontcollide_iop}
directly, these statements may be proven by combining the statement of Proposition \ref{prop:dimers} with that of Lemma \ref{lem:dimer_sep}.

\subsection{Distributional limits}
With Corollary \ref{cor:approxinprob_iop} in hand, we define
\begin{equation}
\label{eq:psidefiop}
    \Psi^\cE = \log \fakepolymerpf^\cE + \Delta^\cE - \tilde{\Delta}^\cE,
\end{equation}
and similarly for $\Psi^\cO$; note that this definition agrees with \eqref{eq:psidefintro}.
Also, let $\mu$ and $\sigma$ be as defined previously in \eqref{eq:mudefintro} and \eqref{eq:sigmadefintro};
we remind the reader that the justification for these definitions can be found in Section \ref{sec:clt}.
With these definitions, we have a joint central limit theorem for $\Psi^\cE$ and $\Psi^\cO$, with the limit being
a pair of independent normal random variables.

\begin{proposition}
\label{prop:clt_iop}
For $p > 0.455$,
\begin{equation}
    \Rnd{\fr{\Psi^\cE - \mu}{\sigma}, \fr{\Psi^\cO - \mu}{\sigma}} \dto \Nor{0}{1} \otimes \Nor{0}{1}.
\end{equation}
\end{proposition}

This will be proved as Proposition \ref{prop:clt_withdimers} in Section \ref{sec:clt} below.
Note that the threshold $0.455$ appearing in this result is unrelated to the threshold $\f12 - \gamma$ of Corollary \ref{cor:approxinprob_iop}.
Instead, it is related to the specific form of the constant $\mu$ defined in \eqref{eq:mudefintro},
which is a good approximation for the true mean of $\Psi^\cE$ if $p > 0.455$.

We also remark that the fluctuations in this limit theorem are solely due to the $\log \fakepolymerpf^\cE$ and
$\log \fakepolymerpf^\cO$ terms in \eqref{eq:psidefiop}.
These two variables themselves exhibit a joint CLT, and the variables
$\Delta^\cE, \Delta^\cO, \tilde{\Delta}^\cE$, and $\tilde{\Delta}^\cO$ each concentrate around deterministic quantities.

The variables contributing to the fluctuations can be written as sums of independent variables, namely
\begin{equation}
    \log \fakepolymerpf^\cE = \sum_{v \in \cE} \log(1 + \phi_v)
    \qquad \text{and} \qquad
    \log \fakepolymerpf^\cO = \sum_{v \in \cO} \log(1 + \phi_v).
\end{equation}
Note however that \emph{the two sums are not independent} since each edge's state in the percolation configuration
contributes to one summand in each sum. However this dependence is rather weak and as in \cite[Lemma 3.1]{brcgw},
a dependency-neighborhood version of Stein's method suffices to conclude the joint CLT to a pair of normal random 
variables which are independent in the limit.

\subsection{Proof of first main theorem}

With all of the above statements in place, we can easily conclude the proof of our first main theorem, yielding
an in-probability approximation of $\Cnt$ as a simple function of two explicit random variables which exhibit a joint
central limit to a pair of independent standard normals. 

\begin{proof}[Proof of Theorem \ref{thm:main}]
By Proposition \ref{prop:polymerdecomp_iop} and Corollary \ref{cor:approxinprob_iop}, we have
\begin{align}
    \fr{\Cnt}{2^{2^{d-1}}} &\psim \fakepolymerpf^\cE \cdot e^{\Delta^\cE - \tilde{\Delta}^\cE}
    + \fakepolymerpf^\cO \cdot e^{\Delta^\cO - \tilde{\Delta}^\cO} \\
    &= e^{\Psi^\cE} + e^{\Psi^\cO}.
\end{align}
So Proposition \ref{prop:clt_iop} finishes the proof.
\end{proof}

\subsection{Sampling independent sets}

As is amply clear from the above few sections, there is a deep connection between approximating a partition function
and understanding the structure of samples from the relevant distribution.
Indeed, the latter can be leveraged to prove our second main theorem, Theorem \ref{thm:sampler},
which states that the approximate sampling algorithm $\AS$ given in Definition \ref{def:approxsampler_intro} gives a good approximation to the uniform
distribution on all independent sets. Recalling the definition of $\AS$, we give some brief intuition regarding this result.

First of all, the relative sizes of $\polymerpf^\cE$ and $\polymerpf^\cO$ are the right quantities to use in deciding whether the constructed independent
set will be mostly contained in the $\cO$ side or the $\cE$ side, respectively.
Next, supposing we've chosen $\cE$ to be the defect side, the statement of Corollary \ref{cor:approxinprob_iop} can be viewed as a mnemonic for the sampling algorithm $\AS$ as follows:
To sample from the distribution with partition function $\polymerpf^\cE$, first sample vertices independently (which has partition function $\fakepolymerpf^\cE$).
This gives the wrong weight to the dimers which will inevitably appear when $p \leq 2 - \sqrt{2}$, so to correct this we remove all such dimers (which corresponds
to dividing the partition function by $e^{\tilde{\Delta}}$) and then resample the dimers according to the correct distribution
(which corresponds to multiplying the partition function by $e^{\Delta}$).

Of course, to make this rigorous we must carefully extract the relevant parts of the various propositions we prove, and so we cannot yet provide a proof of
Theorem \ref{thm:sampler} as we did with Theorem \ref{thm:main} above.
Instead, the proof of Theorem \ref{thm:sampler} will be presented in Section \ref{sec:sampling}.

\subsection{Organization of the article}

We begin in Section \ref{sec:polymerdecomp} with a brief review of polymer models and a proof
of Proposition \ref{prop:polymerdecomp_iop}, which is restated as Proposition \ref{prop:polymerdecomp}.

Next, in Section \ref{sec:smallpolymers}, we prove that only singletons and dimers are relevant 
for the regime we are considering, i.e.\ we prove Proposition \ref{prop:smallpolymers_iop},
which is split into Propositions \ref{prop:smallpolymers} and \ref{prop:smallpolymers_full}.

In Section \ref{sec:nodimers}, we survey the case $p > 2 - \sqrt{2}$, where in fact the dimers are also
irrelevant and an analysis of the singletons suffices.
Much of this case was already covered in our previous work \cite{brcgw}, so this section is mostly an informal overview
which gives some context for the next section.

Section \ref{sec:dimers} contains our analysis of the joint behavior of the singletons and dimers,
and contains the bulk of the technical work done in this article, as well as the main novel conceptual contributions.
We prove Corollary \ref{cor:approxinprob_iop} via Proposition \ref{prop:dimers}, showing that,
for $p > \fr{1}{2} - \gamma$, the dimers do not interact with the singletons or each other when they are sampled independently.

Next, in Section \ref{sec:clt}, we prove the distributional limit statement, Proposition \ref{prop:clt_iop},
restated as Proposition \ref{prop:clt_withdimers} below, showing that the two quantities $\Psi^\cE$ and $\Psi^\cO$ in the exponents of
our representation for $\Cnt$ have independent Gaussian limits after proper rescaling.

Finally, in Section \ref{sec:sampling}, we prove Theorem \ref{thm:sampler}, showing that the algorithm $\AS$ of Definition \ref{def:approxsampler_intro}
yields a good approximation to a uniformly random independent set in $\Qp$.

This article contains two appendices: Appendix \ref{sec:moments} contains various moment calculations which are used
throughout, and Appendix \ref{sec:entropy} contains mostly standard lemmas relating binomial probabilities and
entropy functions, which are used primarily in the proof of Proposition \ref{prop:dimers}.

%% file: sections/polymerdecomp.tex
%!TEX root =../main.tex
\section{Polymer decomposition}
\label{sec:polymerdecomp}
In this section, we set up the polymer framework which makes precise the intuition that a typical independent set in $\Qp$ is only a small perturbation
away from one of the two \emph{ground states}, which consist of independent sets entirely contained within one
of the two halves of the hypercube.
Moreover, the perturbation can be represented by a configuration of \emph{polymers}, which are defects in the
opposite side of the hypercube from the ground state.
In general, this strategy is made rigorous via \emph{polymer models}, which can translate a variety of low-temperature
statistical mechanical models into an associated high-temperature hard-core model on a different graph.

One key appeal of polymer models is the \emph{cluster expansion} which represents their partition functions in a
convenient form which can often yield insights into the model in question.
However, as previously mentioned in Section \ref{sec:iop_polymer},
it should be noted that the present work \emph{does not} use the cluster expansion beyond one key annealed input to be discussed below.
Instead, we work with the partition functions of (random) polymer models directly as random variables.

We remark that many of the results of this section already appeared, implicitly or otherwise, in our previous work \cite{brcgw}:
see Section 4 of that article, and in particular Lemma 4.8 therein.
Nevertheless, we provide a complete exposition in the present article for the reader's convenience, and
the proofs we present here have been reworked for improved clarity.

\subsection{Review of polymer models and the cluster expansion}

To discuss general polymer models briefly, let us consider a general set $\polymers$ of possible polymers $\fp$
(while the precise definition will appear shortly, the reader can think of them as small connected sets for now).
Suppose that there is a pairwise compatibility relation on $\polymers$, and say that a collection $S \sse \polymers$
is compatible if every pair of distinct $\fp_1, \fp_2 \in S$ are compatible.
Additionally suppose that there is a weight $\omega(\fp)$ for each polymer $\fp$.
This data defines a polymer model with a partition function given by
\begin{equation}
    \sum_{\substack{S \sse \polymers \\ \mathrm{compatible}}} \prod_{\fp \in S} \omega(\fp).
\end{equation}

The first step in using polymer models is to show that the partition function of a model of interest is close to
the partition function of a particular polymer model, or perhaps a sum of multiple such partition functions.
In the present article, for instance, this is the content of Proposition \ref{prop:polymerdecomp_iop},
which will be proven below as Proposition \ref{prop:polymerdecomp} after we introduce the details of our polymer model.

Often, to extract useful information or efficiently compute approximations to these polymer model partition functions,
one can use the cluster expansion.
Briefly, a \emph{cluster} of polymers is a multiset $\Gamma$ of polymers which are mutually \emph{incompatible},
in the sense that the \emph{incompatibility graph} $H_\Gamma$, which has vertex set $\Gamma$ and an edge between
incompatible polymers, is connected.
The weight of a cluster $\Gamma$ is given by
\begin{equation}
    \omega(\Gamma) \coloneqq \varphi(H_\Gamma) \prod_{\fp \in \Gamma} \omega(\fp),
\end{equation}
where $\varphi(H)$ is the \emph{Ursell function} of a graph $H$, defined by
\begin{equation}
    \varphi(H) \coloneqq \fr{1}{|V(H)|!} \sum_{\substack{A \sse E(H) \\ \mathrm{spanning} \\ \mathrm{connected}}} (-1)^{|A|}.
\end{equation}
The precise details here are not so important for our purposes, but the key point is the following formal equality,
which is known as the \emph{cluster expansion}:
\begin{align}
\label{eq:clusterexpansion}
    \sum_{\substack{S \sse \polymers \\ \mathrm{compatible}}} \prod_{\fp \in S} \omega(\fp)
    = \Exp{\sum_{\Gamma \text{ cluster}} \omega(\Gamma)}.
\end{align}
Note that the above is a priori only a formal equality; since a cluster is a multiset, even if there are only finitely
many possible polymers, there are infinitely many possible clusters.
So an important step in using the cluster expansion is to prove that the sum inside the exponential actually converges.

\subsection{The cluster expansion of an annealed model}

In the present article, we consider a \emph{random} polymer model, i.e.\ one where the weights are all
random variables, determined by the percolation configuration.
This complicates matters and it is not a priori clear how to prove directly
that this polymer model gives a good approximation for $\Cnt$, as stated in Proposition \ref{prop:polymerdecomp_iop}.
However, \cite{ks} applied the technique outlined above to an \emph{annealed} version of the model, where a set of vertices in $\Q$
(not necessarily an independent set) is given a weight equal to the probability that it is independent in $\Qp$.
The partition function of this model is then exactly equal to $\E[\Cnt]$.
We will use this input and a careful application of Markov's inequality to obtain Proposition \ref{prop:polymerdecomp_iop}.

We remark that \cite{ks} actually used a similar method to compute \emph{all} moments of $\Cnt$,
by constructing a deterministic model for each $k \in \mathbb{N}$ with partition functions $\E[\Cnt^k]$.
The main contribution of \cite{ks} is to use the cluster expansion of associated polymer models in order to get
accurate estimates of the moments of $\Cnt$, which can be computed via an algorithm.
Unfortunately, the moments of $\Cnt$ themselves are not descriptive of the typical behavior of $\Cnt$ unless 
$p$ is very close to $1$, but it turns out that some of the results of \cite{ks} will be useful for our purposes.

We will henceforth only focus on the $k=1$ polymer model with partition function $\E[\Cnt]$,
and we now briefly describe this polymer model, introduced by \cite{ks}.

\begin{definition}[Language relating to polymer models from \cite{ks}]
\label{def:polymers}
In these definitions, we use the notation $N(A)$ to denote the \emph{set of neighbors} of $A$ in $\Q$, rather than
the number of neighbors.
Additionally, $\cH$ denotes either $\cE$ or $\cO$.
\begin{itemize}
    \item
    A subset $A \sse \cH$ is said to be \emph{2-linked} if there is no decomposition $A = A_1 \cup A_2$ into
    nonempty sets $A_1$ and $A_2$ which satisfy $N(A_1) \cap N(A_2) = \emptyset$.

    \item 
    For a subset $A \sse \cH$, we denote by $[A]$ the \emph{closure} of $A$,
    defined to be the largest subset of $\cH$ with the same neighborhood set as $A$.
    In other words, $[A] = \{ v \in \cH : N(v) \sse N(A) \}$.

    \item
    A 2-linked set $\fp \sse \cH$ with $|[\fp]| \leq \fr{3}{4} 2^{d-1}$ is called a \emph{polymer}.

    \item
    Polymers $\fp_1$ and $\fp_2$ are said to be \emph{compatible} if $\fp_1 \cup \fp_2$ is not 2-linked.

    \item
		We denote by $\Sep[\cH]$ the set of pairwise compatible collections of polymers.
    
    \item
    The \emph{weight} of a polymer $\fp$ is $\omega(\fp) = \E[\phi_\fp]$, where $\phi_\fp = 2^{-\N(\fp)}$.
    We remind the reader that $\N(\fp)$ denotes the (random) number of vertices neighboring $\fp$
    in the percolated graph $\Qp$.

    \item
    We define \emph{clusters} and their weights as in the previous section, which is a general construction that
    works for all polymer models.
    The set of clusters will be denoted by $\cC^\cH$.
\end{itemize}
\end{definition}

A few remarks are in order.
First, the number $\fr{3}{4}$ which appears in the definition of a polymer is a somewhat arbitrary choice, and,
as remarked in \cite{ks}, any constant between $\fr{1}{2}$ and $1$ can be used, for reasons which we will not
explore in this article.
Second, the definition given in \cite[Equation 14]{ks} for the weight $\omega(\fp)$ of a polymer differs from the expression
presented here, but it is easy to see that the two expressions are equal.
Finally, note that the two polymer models, one for each half $\cH \in \{\cE,\cO\}$, actually have identical
partition functions, but we will nonetheless keep these separate.

With these definitions in hand, we can state one of the main results of \cite{ks}, which expands $\E[\Cnt]$ in a
sum of two convergent cluster expansions for these two polymer models.

\begin{theorem}[Theorem 3.1 of \cite{ks}, with $k=1$]
\label{thm:clusterexpansion}
For $p \in (0,1)$, 
\begin{equation}
    \E[\Cnt] = 2^{2^{d-1}} \left(
        \Exp{\sum_{\Gamma \in \cC^\cE} \omega(\Gamma)} 
     + \Exp{\sum_{\Gamma \in \cC^\cO} \omega(\Gamma)} 
    \right) \cdot \Rnd{1 + \Exp{-\Omega(2^d/d^4)}},
\end{equation}
where the cluster expansion series inside the exponentials converge absolutely.
\end{theorem}

We remark here that the series in the above exponential expressions
may be evaluated via bounds (arising from the graph container method)
on the expectations of terms corresponding to higher order clusters of polymers,
and one can obtain
\begin{equation}
\label{expectation12}
    \E[\Cnt] = 2^{2^{d - 1}} \cdot 2 \cdot \Exp{\f{1}{2} (2 - p)^d \cdot (1 + o(1))},
\end{equation}
which is a simplified version of \cite[Theorem 1.1]{ks}.

Now, by the formal cluster expansion \eqref{eq:clusterexpansion}, which converges absolutely as stated in
Theorem \ref{thm:clusterexpansion}, for each $\cH \in \{\cE,\cO\}$, we have
\begin{align}
    \Exp{\sum_{\Gamma \in \cC^\cH} \omega(\Gamma)} 
	&= \sum_{S \in \Sep[\cH] } \prod_{\fp \in S} \omega(\fp) \\
	&= \sum_{S \in \Sep[\cH] } \prod_{\fp \in S} \E[\phi_\fp] \\
    &= \E \left[
		\sum_{S \in \Sep[\cH] } \prod_{\fp \in S} \phi_\fp
    \right] = \E[\polymerpf^\cH],
\end{align}
using at the second-to-last equality the fact that the $\phi_\fp$ are independent for $\fp \in S$ when
$S \in \Sep[\cH]$, as there are no shared vertices between any $\fp \in S$.
The last equality is simply the definition of $\polymerpf^\cH$ from \eqref{eq:polymerpfdef}.
So, as an immediate corollary of Theorem \ref{thm:clusterexpansion}, we obtain the following annealed
version of Proposition \ref{prop:polymerdecomp_iop} providing an approximation
of the expected number of independent sets in $\Qp$.

\begin{corollary}
\label{cor:polymermodel}
For $p \in (0,1)$,
\begin{equation}
    \E\left[ \fr{\Cnt}{2^{2^{d-1}}} \right]
	= \E[\polymerpf^\cE + \polymerpf^\cO] \cdot \Rnd{1 + \Exp{-\Omega(2^d/d^4)}}.
\end{equation}
\end{corollary}

\subsection{Polymer decomposition of the quenched partition function}
\label{sec:polymerdecomp_quenched}

Using Corollary \ref{cor:polymermodel} as well as another input from \cite{ks} to be mentioned below,
we can now obtain a quenched approximation for the number of independent sets in $\Qp$.
This takes the form of a sum of two polymer model partition functions, which we will analyze directly in the sequel.
The reader familiar with polymer models should note that we will \emph{not} apply the cluster expansion to these partition functions at this point.
The following proposition was previously stated as Proposition \ref{prop:polymerdecomp_iop}, and we
remind the reader that $A \psim B$ means that $\fr{A}{B} \pto 1$ as $d \to \infty$.

\begin{proposition}
\label{prop:polymerdecomp}
For $p \in (0,1)$,
\begin{equation}
    \fr{\Cnt}{2^{2^{d-1}}} \psim \polymerpf^\cE + \polymerpf^\cO.
\end{equation}
where the terms on the RHS were defined in \eqref{eq:polymerpfdef}.
\end{proposition}

The proof is an application of Markov's inequality; however, since Corollary \ref{cor:polymermodel} only says
that the expectations of the two sides in Proposition \ref{prop:polymerdecomp} are close, we will also need some
form of monotonicity as well as an a priori bound on $\Cnt$ to make the argument work.

As previously mentioned in Section \ref{sec:iop}, specifically in \eqref{eq:exactrep}, we have
\begin{equation}
    \fr{\Cnt}{2^{2^{d-1}}} = \sum_{T \sse \cE} 2^{-\N(T)}.
\end{equation}
Note that, if $T$ can be represented as a union of well-separated polymers $T = \bigcup_{\fp \in S} \fp$
for $S \in \Sep[\cE]$, then
\begin{equation}
    2^{-\N(T)} = \prod_{\fp \in S} \phi_\fp,
\end{equation}
and so the quantity $2^{-\N(T)}$ appears as a term inside the expression $\polymerpf^\cE$
on the right-hand side of Proposition \ref{prop:polymerdecomp}.
In other words, every independent set $I$ with an even side which is representable via a polymer configuration
is accounted for in $\polymerpf^\cE$, and the number of such sets is exactly $2^{2^{d-1}} \polymerpf^\cE$.
Similarly, every independent set $I$ with an odd side which is representable via a polymer configuration is accounted
for in $\polymerpf^\cO$ and the number of such sets is exactly $2^{2^{d-1}} \polymerpf^\cO$.

However, there may be independent sets for which neither side is representable via a polymer configuration, which
are not counted in the right-hand side of Proposition \ref{prop:polymerdecomp}.
Moreover, there are also independent sets which are counted in both $\polymerpf^\cE$ and $\polymerpf^\cO$,
since both sides can be decomposed into well-separated polymers.
This means that the two sides of Proposition \ref{prop:polymerdecomp} are not immediately comparable via an inequality.
To remedy this, we use the following lemma from \cite{ks} which gives another annealed bound on the total number
of independent sets for which both sides are not representable by a polymer configuration.

\begin{lemma}[Lemma 3.3 of \cite{ks}, with $k=1$]
\label{lem:toobig}
For $p \in (0,1)$,
\begin{equation}
    \sum_{\substack{I \sse \Q \\ |[I \cap \cE]|, |[I \cap \cO]| > \fr{3}{4} 2^{d-1}}}
    \P[I \text{ is independent in } \Qp]
    \leq \E[\Cnt] \cdot \Exp{-\Omega(2^d/d)}.
\end{equation}
\end{lemma}

Recall from Definition \ref{def:polymers} that $[A]$ is the closure of a set, and that $|[\fp]| \leq \fr{3}{4} 2^{d-1}$
is the only requirement, beyond being $2$-linked, for $\fp$ to be a polymer under Definition \ref{def:polymers}.
So if $|[I \cap \cE]| \leq \fr{3}{4} 2^{d-1}$, then $I \cap \cE$ can be represented by a polymer configuration, and
similarly for $I \cap \cO$.
Thus we immediately obtain the following corollary.

\begin{corollary}
\label{cor:toobig} 
Let $\TooBig$ denote the number of independent sets $I \sse \Q$ for which neither $I \cap \cE$ nor $I \cap \cO$
are representable via a polymer configuration.
Then, for all $p \in (0,1)$,
\begin{equation}
    \E[\TooBig] \leq \E[\Cnt] \cdot \Exp{-\Omega(2^d/d)}.
\end{equation}
\end{corollary}

With this in hand, we can now prove Proposition \ref{prop:polymerdecomp}.

\begin{proof}[Proof of Proposition \ref{prop:polymerdecomp}]
Define
\begin{equation}
    \cX \coloneqq
    \fr{\Cnt}{2^{2^{d-1}}} - \polymerpf^\cE - \polymerpf^\cO,
\end{equation}
so that our goal is to show that
\begin{align}
\label{eq:polydecompgoal}
    \fr{\cX}{\Cnt / 2^{2^{d-1}}} \pto 0.
\end{align}
An initial idea might be to use Markov's inequality, since the expectation of $\cX$ is small; indeed
Corollary \ref{cor:polymermodel} says that
\begin{align}
\label{eq:xsmall}
    \left| \E[\cX] \right| \leq \E \left[ \fr{\Cnt}{2^{2^{d-1}}} \right] \cdot \Exp{-\Omega(2^d/d^4)}.
\end{align}
However, $\cX$ is not positive (note that the absolute value is on the outside of the expectation above),
and moreover the denominator of \eqref{eq:polydecompgoal} is random.
Thus we need to be a bit more careful.

First we will show that
\begin{align}
\label{eq:toobigsmall}
    \fr{\TooBig}{\Cnt} \pto 0,
\end{align}
where $\TooBig$ is defined as in Corollary \ref{cor:toobig}.
In order to apply Markov's inequality with that corollary, we will need an a priori lower bound on $\Cnt$.
Since $\Cnt \geq 2 \cdot 2^{2^{d-1}}$ deterministically, by \eqref{expectation12} we have
\begin{align}
\label{eq:aprioriLB}
    \Cnt \geq \Exp{- \fr{1}{2}(2-p)^d (1 + o(1))} \cdot \E[\Cnt].
\end{align}
So for any $\eps > 0$, we have
\begin{align}
    \P\left[\TooBig \geq \eps \cdot \Cnt\right] &\leq
    \P\left[\TooBig \geq \eps \cdot \Exp{-\fr{1}{2}(2-p)^d (1 + o(1))} \E[\Cnt]\right] \\
    &\leq \fr{\Exp{-\Omega(2^d/d)}}{\eps \cdot \Exp{-\fr{1}{2}(2-p)^d (1 + o(1))}},
\end{align}
which tends to zero as $d \to \infty$ since $p > 0$, proving \eqref{eq:toobigsmall}.

Next we turn to \eqref{eq:polydecompgoal}.
By the definitions of $\cX$ and $\TooBig$, as well as the discussion (between the statements of Proposition \ref{prop:polymerdecomp} and Lemma \ref{lem:toobig} above)
of how the polymer partition functions relate to the number of polymer-decomposable independent sets,
\begin{equation}
    \TooBig - 2^{2^{d-1}} \cX =
    \underbrace{2^{2^{d-1}} \cdot \polymerpf^\cE}_{\substack{\text{number of independent sets} \\ \text{polymer-decomposable} \\ \text{on even side}}} +
    \underbrace{2^{2^{d-1}} \cdot \polymerpf^\cO}_{\substack{\text{number of independent sets} \\ \text{polymer-decomposable} \\ \text{on odd side}}} - \quad
    \underbrace{(\Cnt - \TooBig)}_{\substack{\text{number of independent sets} \\ \text{polymer-decomposable} \\ \text{on either side}}} \geq 0,
\end{equation}
since the above quantity is exactly the number of independent sets for which \emph{both sides} have representations
as polymer configurations.
So, using this nonnegativity, we now have
\begin{align}
    \P&\left[ \TooBig - 2^{2^{d-1}} \cX \geq \eps \cdot \Cnt \right] \\
    &\leq \P\left[ \TooBig - 2^{2^{d-1}} \cX \geq \eps \cdot \Exp{- \f12 (2-p)^d (1 + o(1))} \cdot \E[\Cnt] \right] \tag*{(by \eqref{eq:aprioriLB})} \\
    &\leq \fr{\E[\TooBig - 2^{2^{d-1}} \cX]}{\eps \cdot \Exp{- \f12 (2-p)^d (1+o(1))} \cdot \E[\Cnt]} \tag*{(by Markov's inequality)} \\
    &\leq \fr{\E[\TooBig] + 2^{2^{d-1}} |\E[\cX]|}{\eps \cdot \Exp{-\f12 (2-p)^d (1+o(1))} \cdot \E[\Cnt]} \\
    &\leq \fr{\Exp{-\Omega(2^d/d)} + \Exp{-\Omega(2^d/d^4)}}{\eps \cdot \Exp{-\fr{1}{2}(2-p)^d (1 + o(1))}}. \tag*{(by \eqref{eq:xsmall} and Corollary \ref{cor:toobig})}
\end{align}
The above tends to zero since $p > 0$, proving that
\begin{equation}
    \fr{\TooBig}{\Cnt} - \fr{\cX}{\Cnt / 2^{2^{d-1}}} \pto 0.
\end{equation}
Combining this with \eqref{eq:toobigsmall} completes the proof.
\end{proof}

%% file: sections/smallpolymers.tex
%!TEX root =../main.tex
\section{Bounding the size of polymers}
\label{sec:smallpolymers}

In principle, a polymer configuration $S$ can contain polymers of any size, and there are
increasingly many possibilities for a polymer of size $k$ as $k$ grows, which would be somewhat cumbersome
to deal with.
However, it turns out that for $p$ bounded away from zero, only finitely many types of polymers are relevant
in the decomposition.
This is formalized in Proposition \ref{prop:smallpolymers} below, which is a more general version of
Proposition \ref{prop:smallpolymers_iop} above.

In the sequel, we only consider $p$ large enough that in fact only polymers of size $\leq 2$ (i.e.\ singletons
and dimers) are relevant.
Nevertheless, we state the results in this section in greater generality in the hopes that they may be useful
for future work investigating smaller values of $p$.
In addition, we remark that a version of these results (covering the case where only singletons are relevant) already appeared in
our previous work \cite[Lemma 4.9]{brcgw}.
Thus, as with Section \ref{sec:polymerdecomp} above, this section is provided more for completeness, in addition to providing streamlined proofs,
rather than presenting new ideas.

Note that in this section we drop the superscript $\cH$ denoting the side of the hypercube;
for instance, $\polymerpf^\cH$ will simply become $\polymerpf$.
This is because all results in this section focus on only one side at a time, and hold equally for the
even and the odd sides.

Recall from \eqref{eq:polymerpfdef} and \eqref{eq:smallpolymerpfdef} the notation
\begin{equation}
    \polymerpf = \sum_{S \in \Sep} \prod_{\fp \in S} \phi_\fp
    \qquad \text{and} \qquad
    \polymerpf_{\leq k} = \sum_{S \in \Sep_{\leq k}} \prod_{\fp \in S} \phi_\fp.
\end{equation}
Here, and in the sequel, $\Sep = \Sep[\cH]$ denotes the collection of separated polymer configurations
on side $\cH$ of the hypercube, and $\Sep_{\leq k}$ denotes the subcollection where all the polymers
are constrained to have size at most $k$.

The following proposition gives the threshold in $p$ for the irrelevance of polymers of size larger than $k$,
for each $k \in \mathbb{N}$.
This proposition will be proved in Section \ref{sec:smallpolymers_true} below.

\begin{proposition}
\label{prop:smallpolymers}
For each $k \geq 1$, if $p > 2 - 2^{k/(k+1)}$, then we have $\polymerpf \psim \polymerpf_{\leq k}$.
\end{proposition}

\subsection{The independent singletons model}

In this section, as a warm-up which will also be useful in the overall proof of our main theorem, we prove that a random set drawn according
to a simpler product measure (which was alluded to in the proof ideas presented in Section \ref{sec:iop}, and which will appear repeatedly in our arguments)
also satisfies a polymer decomposition with small polymers when $p$ is bounded away from zero.
Specifically, we consider the model which samples $S \sse \cH$ by choosing each vertex $v$ to be included in $S$ independently
with probability $\fr{\phi_v}{1 + \phi_v}$, where we recall from \eqref{eq:phidef} that $\phi_v = 2^{-\N(v)}$.
If the weight of a set $S$ is written simply as $\prod_{v \in S} \phi_v$, then the partition function of this model is
\begin{equation}
	\prod_{v \in \cH} \Rnd{1 + \phi_v}.
\end{equation}
Recall from \eqref{eq:fakephidef} the notation $\fphi_\fp = \prod_{v \in \fp} \phi_v$.
Then in expanding the above product, each term can be written as a product of $\fphi_\fp$ for some
polymers $\fp$ in a valid polymer configuration, as follows
\begin{equation}
	\prod_{v \in \cH} \Rnd{1 + \phi_v} = \sum_{S \in \widetilde{\Sep}} \prod_{\fp \in S} \fphi_\fp
    = \fakepolymerpf,
\end{equation}
recalling the definition of $\fakepolymerpf$ from \eqref{eq:fakepolymerpfdef}.
Note that in defining $\fakepolymerpf$ we use a slightly relaxed notion of polymers than that which appears in \cite{ks} and 
in Definition \ref{def:polymers};
specifically, we do not impose the restriction that $|[\fp]| \leq \fr{3}{4} 2^{d-1}$.
We denote the collection of these polymers by $\fakepolymers$, and define $\widetilde{\Sep}$ and
$\widetilde{\Sep}_{\leq k}$ analogously to $\Sep$ and $\Sep_{\leq k}$.
However we remark that for large enough $d$ and for any fixed $k$, we have
$\widetilde{\Sep}_{\leq k} = \Sep_{\leq k}$, and so we will simply use the latter notation in the
sequel, where we will only work with polymers of bounded size.

The following proposition is the analogue of Proposition \ref{prop:smallpolymers} for this model.
It says that for large enough $p$, only small enough polymers are relevant.
Recall from \eqref{eq:smallpolymerpfdef} the notation
\begin{equation}
    \fakepolymerpf_{\leq k} = \sum_{S \in \Sep_{\leq k}} \prod_{\fp \in S} \fphi_\fp.
\end{equation}

\begin{proposition}
\label{prop:smallpolymers_full}
For each $k \geq 1$, if $p > 2 - 2^{k/(k+1)}$, then we have $\fakepolymerpf \psim \fakepolymerpf_{\leq k}$.
\end{proposition}

The key lemma is the following, which bounds the expected total weight of polymers of size at least $k$.
We denote the collection of relaxed polymers of size exactly $j$ by $\fakepolymers_j$.

\begin{lemma}
\label{lem:weightbound_full}
For every $p \in (0,1)$ and every fixed $k \geq 0$,
\begin{equation}
    \sum_{j > k} \sum_{\fp \in \fakepolymers_j} \E[\fphi_\fp]    
    \leq \poly(d) \cdot \Fr{(2-p)^{k+1}}{2^{k}}^d.
\end{equation}
\end{lemma}

\begin{proof}[Proof of Lemma \ref{lem:weightbound_full}]
The number of polymers of size $j$ is at most $2^d (e d^2)^j$; to see why, consider the following method
of constructing a polymer. First, choose one of the $j^{j-2}$ labeled trees of size $j$.
Next, let choose any of the $2^{d-1}$ vertices of $\cH$ to be the first vertex, with label $1$.
Perform breadth-first or depth-first search in the tree, always choosing the vertex with minimum label
among possible choices; when a vertex $i$ is selected, attach it to its parent $k$ (with the root having label $1$)
by selecting one of the $\leq d^2$ vertices sharing a neighbor with the vertex labeled $k$, and giving it the
label $i$.

Every polymer of size $j$ is formed by an instance of this construction, since the $2$-linked graph structure of any
polymer has at least one spanning tree.
Moreover, each polymer is counted at least $j!$ times since we can permute the labels of the tree arbitrarily,
and in any permutation of the labels it is still possible to generate the same polymer.
So the number of polymers of size $j$ is at most
\begin{equation}
    j^{j-2} \cdot 2^{d-1} (d^2)^{j-1} \cdot \fr{1}{j!} \leq 2^d (d^2)^j \fr{j^j}{j!} \leq 2^d (ed^2)^j
\end{equation}
as required.

Now for each polymer $\fp$ with size $j$, we have
$\fphi_\fp \stackrel{d}{=} 2^{-\Bin(j d, p)}$.
So by a simple calculation like the one in presented in Lemma \ref{lem:moments} (a), we have
\begin{equation}
    \E[\fphi_\fp] = \Fr{2-p}{2}^{jd},
\end{equation}
and so
\begin{align}
    \sum_{j > k} \sum_{\fp \in \fakepolymers_j} \E[\fphi_\fp]
    &\leq 2^d \sum_{j > k} \Rnd{ e d^2 \Fr{2-p}{2}^d }^j.
\end{align}
For large enough $d$, we have $e d^2 \Fr{2-p}{2}^d < \fr{1}{2}$, and so the geometric sum is dominated
by $2$ times its first term, which is the term corresponding to $j=k+1$, i.e.\ we have
\begin{align}
    \sum_{j > k} \sum_{\fp \in \fakepolymers_j} \E[\fphi_\fp]
    &\leq 2^d \cdot 2 \cdot \Rnd{ e d^2 \Fr{2-p}{2}^d }^{k+1},
\end{align}
which is the desired bound.
\end{proof}

With this lemma in hand, we can now prove Proposition \ref{prop:smallpolymers_full}.

\begin{proof}[Proof of Proposition \ref{prop:smallpolymers_full}]
Consider the quantity
\begin{align}
\label{eq:bad}
    \Exp{\sum_{j > k} \sum_{\fp \in \fakepolymers_j} \fphi_\fp} \cdot
    \sum_{S \in \Sep_{\leq k}} \prod_{\fp \in S} \fphi_\fp.
\end{align}
By Taylor expanding the exponential of the (finite) sum over polymers of size larger than $k$, one obtains various products
of terms in the sum, and those corresponding to collections of $m$ \emph{distinct} polymers appear with a coefficient of $1$,
since they appear $m!$ times, each with a coefficient of $\fr{1}{m!}$ from the Taylor series for the exponential function.
Moreover, since all terms in the sum are positive, we arrive at
\begin{align}
    \Exp{\sum_{j > k} \sum_{\fp \in \fakepolymers_j} \fphi_\fp} \cdot
    \sum_{S \in \Sep_{\leq k}} \prod_{\fp \in S} \fphi_\fp
    \geq \sum_{T \in \tilde{\Sep}} \prod_{\fp \in T} \fphi_\fp
    =  \fakepolymerpf,
\end{align}
as soon as $d$ is large enough so that $\tilde{\Sep}_{\leq k} = \Sep_{\leq k}$.
On the other hand, by definition we have
\begin{equation}
    \sum_{S \in \Sep_{\leq k}} \prod_{\fp \in S} \fphi_\fp = \fakepolymerpf_{\leq k},
\end{equation}
and we thus obtain
\begin{equation}
    \left( \Exp{\sum_{j > k} \sum_{\fp \in \fakepolymers_j} \fphi_\fp} - 1 \right)
    \cdot \sum_{S \in \Sep_{\leq k}} \prod_{\fp \in S} \fphi_\fp
    \geq \fakepolymerpf - \fakepolymerpf_{\leq k} \geq 0,
\end{equation}
which, upon rearranging, yields
\begin{equation}
    \Exp{\sum_{j > k} \sum_{\fp \in \fakepolymers_j} \fphi_\fp}
    \geq \fr{\fakepolymerpf}{\fakepolymerpf_{\leq k}} \geq 1.
\end{equation}
So the proof will be finished if we can show that
\begin{equation}
    \sum_{j > k} \sum_{\fp \in \fakepolymers_j} \fphi_\fp
    \pto 0.
\end{equation}
But this follow from Lemma \ref{lem:weightbound_full} and Markov's inequality, since when
$p > 2 - 2^{k/(k+1)}$ the upper bound in that lemma tends to $0$ as $d \to \infty$.
\end{proof}

\subsection{The true independent set model}
\label{sec:smallpolymers_true}

We now turn to the proof of Proposition \ref{prop:smallpolymers}, which is the same statement as Proposition \ref{prop:smallpolymers_full}
which was just proved, but with
\begin{equation}
    \fphi_\fp = \prod_{v \in \fp} 2^{-\N(v)} \qquad \text{replaced by} \qquad
    \phi_\fp = 2^{-\N(\fp)}.
\end{equation}
The key lemma in the previous proof, Lemma \ref{lem:weightbound_full}, crucially leveraged the independence of $\phi_v = 2^{-\N(v)}$ for distinct vertices
to give an upper bound on the sum of $\E[\fphi_\fp]$ over \emph{all} large polymers $\fp$.
While it may be expected that a similar bound holds for $\E[\phi_\fp]$ for \emph{small enough} polymers, it is far from obvious that the sum 
over all large polymers may be bounded in the same way, since $\phi_\fp$ may be much larger than $\fphi_\fp$ for large polymers owing to the
large number of overlapping neighborhoods of vertices in $\fp$.

The \emph{graph container method} is a powerful tool that allows one to bound the number of such large sets based on their neighborhood sizes,
and it yields the estimates we need.
This calculation was carried out in \cite{ks}, using graph container results presented in \cite{galvinthresh} that follow the work of \cite{sapozhenko} closely,
and the result is the following lemma, which will replace Lemma \ref{lem:weightbound_full} in the proof of Proposition \ref{prop:smallpolymers}.
Note that we are now restricting to polymers as defined by \cite{ks}, which are the same polymers
as those appearing in the decomposition given by Proposition \ref{prop:polymerdecomp}.
Specifically, these include the restriction that $|[\fp]| \leq \fr{3}{4} 2^{d-1}$, and we denote the collection
of such polymers with size exactly $j$ by $\polymers_j$.

\begin{lemma}[Adapted from Lemma 4.1 of \cite{ks}]
\label{lem:weightbound}
    For every $p \in (0,1)$ and every fixed $k \geq 1$,
    \begin{equation}
        \sum_{j > k}
        \sum_{\fp \in \polymers_j} \E[\phi_\fp]
        \leq \poly(d) \cdot \Fr{(2-p)^{k+1}}{2^k}^d
    \end{equation}
\end{lemma}

We remark that this statement is weaker than the original version of Lemma 4.1 of \cite{ks}, but will suffice for
our purposes.
For the reader who wishes to compare the original statement to the above statement, we remark that the original
statement is about the sum over all \emph{clusters} of polymers, and each term in the sum includes an extra exponential
factor not written above.
The above statement follows from the original statement by restricting the sum to only clusters of size $1$ (which are the
same as individual polymers), using the fact that every term in the sum is nonnegative.
Moreover, the aforementioned exponential factors are all $\geq 1$, so we can simply erase them and still retain the
same bound.
With Lemma \ref{lem:weightbound} in hand, Proposition \ref{prop:smallpolymers} follows by the same proof as for
Proposition \ref{prop:smallpolymers_full}, but with $\fphi_\fp$ replaced by $\phi_\fp$.

%% file: sections/nodimers.tex
%!TEX root =../main.tex

\section{The regime with no dimers}
\label{sec:nodimers}

The main new technical and conceptual contribution of the present article is a sharp understanding of $\Cnt$ in the regime of parameters where
dimers are relevant.
But, before pursuing that in Section \ref{sec:dimers} below, in the present section we highlight how our preparation up to this point yields a quick analysis of the regime
$p > 2-\sqrt 2 \approx 0.586$, where no dimers appear in the defect side of typical independent sets.
In particular, with a bit of extra care in tracking error terms through our proofs of Propositions \ref{prop:polymerdecomp},
\ref{prop:smallpolymers}, and \ref{prop:smallpolymers_full},
one can recover the results of our previous work \cite{brcgw}, which were proven therein for $p \ge \f23$.
We do not spell out all of the details presently as these results have already appeared in our previous article. 

\subsection{In-probability approximation}
\label{sec:nodimers_approx}

By Propositions \ref{prop:polymerdecomp} and \ref{prop:smallpolymers}, when $p > 2 - \sqrt{2}$, we have
\begin{equation}
    \fr{\Cnt}{2^{2^{d-1}}} \psim \polymerpf^\cE_1 + \polymerpf^\cO_1.
\end{equation}
We remind the reader that
\begin{equation}
    \polymerpf^\cH_1 = \sum_{S \in \Sep^\cH_1} \prod_{v \in S} \phi_v,
\end{equation}
where $\Sep^\cH_1$ denotes the collection of well-separated subsets of $\cH$.
Now, by Proposition \ref{prop:smallpolymers_full}, since $p > 2 - \sqrt{2}$, we have
\begin{equation}
    \polymerpf^\cH_1 \psim \fakepolymerpf^\cH = \prod_{\fs \in \cH} (1 + \phi_\fs)
\end{equation}
for each $\cH$, using the fact that $\polymerpf^\cH_1 = \fakepolymerpf^\cH_1$ since $\phi_\fs = \fphi_\fs$ for singletons.
Thus we have
\begin{equation}
\label{eq:papprox2rt2}
    \fr{\Cnt}{2^{2^{d-1}}} \psim \Exp{\Phi_{\log}^\cE} + \Exp{\Phi_{\log}^\cO},
\end{equation}
where
\begin{equation}
\label{eq:philogdef}
    \Phi_{\log}^\cH \coloneqq \log \fakepolymerpf^\cH = \sum_{\fs \in \cH} \log(1 + \phi_\fs).
\end{equation}
This completes the first part of our program, namely exhibiting an explicit in-probability approximation
for $\Cnt$.
In the regime $p > 2 - \sqrt{2}$, note that the quantities $\Delta^\cH$ and $\tilde{\Delta}^\cH$ for $\cH \in \{\cE,\cO\}$,
as defined in \eqref{eq:deltadef} via sums of contributions coming from dimers, are not relevant for this approximation.

\subsection{Distributional limits}

Next, as will be proved in Section \ref{sec:clt} below (specifically as Corollary \ref{cor:clt_withapprox}),
$\Phi_{\log}^\cE$ and $\Phi_{\log}^\cO$ satisfy a joint central limit theorem with explicit recentering and
rescaling, i.e.\ we have
\begin{equation}
    \Rnd{
        \fr{\Phi_{\log}^\cE - \mu_1}{\sigma}, \fr{\Phi_{\log}^\cO - \mu_1}{\sigma}
    } \dto \Nor{0}{1} \otimes \Nor{0}{1},
\end{equation}
where, as previously defined in \eqref{eq:mu1defintro} and \eqref{eq:sigmadefintro},
\begin{align}
    \mu_1 &= \fr{1}{2} (2-p)^d - \fr{1}{4} \Rnd{2 - \fr{3}{2} p}^d + \fr{1}{6} \Rnd{2 - \fr{7}{4} p}^d, \\
    \sigma^2 &= \fr{1}{2} \Rnd{2 - \fr{3}{2}p}^d. \label{eq:sigmarecall}
\end{align}
In fact, we may simplify the centering quantity in the regime $p > 2 - \sqrt{2}$.
Specifically, in this regime the third term of $\mu_1$ is $\ll \sigma$, which follows from the fact that
\begin{equation}
    \Rnd{2 - \fr{7}{4} p}^2 < 2 - \fr{3}{2} p
\end{equation}
for such $p$.
To verify the above inequality, note that the difference between the left and right-hand sides is
\begin{equation}
    \fr{49}{16} p^2 - \fr{11}{2} p + 2,
\end{equation}
which has roots
\begin{equation}
    \fr{11 \pm \sqrt{11^2 - 2 \cdot 49}}{49/8} \approx 0.506, 1.289.
\end{equation}
Therefore, since $2 - \sqrt{2} \approx 0.586$, the difference is negative on the interval $(2 - \sqrt{2}, 1)$.
All this is to say that if we define
\begin{equation}
\label{eq:muprimedef}
    \mu' = \fr{1}{2} (2-p)^d - \fr{\sigma^2}{2},
\end{equation}
then we in fact have
\begin{equation}
\label{eq:clt2rt2}
    \Rnd{
        \fr{\Phi_{\log}^\cE - \mu'}{\sigma},
        \fr{\Phi_{\log}^\cO - \mu'}{\sigma}
    } \dto \Nor{0}{1} \otimes \Nor{0}{1}
\end{equation}
when $p > 2 - \sqrt{2}$.

\subsection{Recovery of previous results}

With a bit of extra care in tracking the errors through our proofs, our present results may be used to
witness a previously-discovered phase transition at $p = \fr{2}{3}$, whereby the quantity $\Cnt$ itself
exhibits a Gaussian scaling limit for $p > \fr{2}{3}$, a lognormal limit for $p = \fr{2}{3}$,
and no scaling limit for $p < \fr{2}{3}$.
As these distributional limit results were already proved in \cite{brcgw}, we do not provide full details presently,
instead being satisfied with an informal discussion.

In the statements of Propositions \ref{prop:polymerdecomp}, \ref{prop:smallpolymers}, and \ref{prop:smallpolymers_full},
various quantities are related via $\psim$, meaning there is a multiplicative $(1+o(1))$-in-probability error.
By carefully following the proofs this may be improved to a multiplicative error of $(1+o(\sigma))$-in-probability
when $p > \fr{2}{3}$, which is the regime for which $\sigma \to 0$ (recall the formula for $\sigma^2$ in \eqref{eq:sigmarecall}).
Additionally, $\sigma^2 \ll \sigma$ in this regime, so if we define
\begin{equation}
    \hat{\mu} = \fr{1}{2}(2-p)^d,
\end{equation}
then the improved version of \eqref{eq:papprox2rt2} combined with \eqref{eq:clt2rt2} yields
\begin{align}
    \fr{\Cnt}{2^{2^{d-1}} e^{\hat{\mu}}} &= \Rnd{
        \Exp{\Phi_{\log}^\cE - \hat{\mu}} + 
        \Exp{\Phi_{\log}^\cO - \hat{\mu}}
    } (1 + o(\sigma)) \\
    &= \Rnd{
        2 + \Phi_{\log}^\cE + \Phi_{\log}^\cO - 2 \hat{\mu}
    } (1 + o(\sigma)).
\end{align}
Here the $o(\sigma)$ error term satisfies $\fr{o(\sigma)}{\sigma} \pto 0$ and we have used the fact that the
\eqref{eq:clt2rt2} implies that the error in approximating the exponentials by linear functions is $o(\sigma)$,
since $(\Phi_{\log}^\cE - \hat{\mu})^2 = O(\sigma^2) = o(\sigma)$ because $\sigma \to 0$
(here again the $O(\sigma^2)$ means in-probability).
The central limit theorem \eqref{eq:clt2rt2} also shows that
\begin{equation}
    \fr{\Phi_{\log}^\cE + \Phi_{\log}^\cO - 2 \hat{\mu}}{\sigma} \dto \Nor{0}{2},
\end{equation}
and so, since the approximation is $(1 + o(\sigma))$-in-probability, we obtain a central limit theorem
for $\Cnt$.
Spelling it out with all quantities written explicitly in terms of $d$ and $p$, we get
\begin{equation}
    \fr{\Cnt - 2 \cdot 2^{2^{d-1}} \Exp{\fr{1}{2} (2-p)^d}}
    {2^{2^{d-1}} \Exp{\fr{1}{2}(2-p)^d} \Rnd{2 - \fr{3}{2} p}^{d/2} } \dto \Nor{0}{1}.
\end{equation}
This recovers the first part of \cite[Theorem 1.1]{brcgw}, which extends \cite[Theorem 1.3]{ks}
from a regime of nonconstant $p = p_n \to 1$ as $d \to \infty$ to all constant $p \in (\fr{2}{3},1)$.

We also recover the second part of \cite[Theorem 1.1]{brcgw}, which covers the case $p = \fr{2}{3}$.
In this case $\sigma^2 = \fr{1}{2}$, and \eqref{eq:papprox2rt2} thus gives
\begin{equation}
    \fr{\Cntcrit}{2^{2^{d-1}} e^{\mu'}} \psim
    \Rnd{\Exp{\fr{1}{\sqrt{2}} \fr{\Phi_{\log}^\cE - \mu'}{\sigma}}
    +\Exp{\fr{1}{\sqrt{2}} \fr{\Phi_{\log}^\cE - \mu'}{\sigma}}}.
\end{equation}
So, after plugging in the value for $\mu'$ in terms of $d$ with $p = \fr{2}{3}$, \eqref{eq:clt2rt2} gives
\begin{equation}
    \fr{\Cntcrit}{2^{2^{d-1}} \Exp{\fr{1}{2}(2-p)^d - \fr{1}{4}}}
    \dto e^{W^\cE/\sqrt{2}} + e^{W^\cO/\sqrt{2}},
\end{equation}
where $W^\cE, W^\cO \sim \Nor{0}{1}$ are independent.

Finally, for $2 - \sqrt{2} < p < \fr{2}{3}$, \eqref{eq:papprox2rt2} yields
\begin{equation}
    \fr{\Cnt}{2^{2^{d-1}} e^{\mu'}} \psim
    \Exp{\sigma \cdot \fr{\Phi_{\log}^\cE - \mu'}{\sigma}} + \Exp{\sigma \cdot \fr{\Phi_{\log}^\cO - \mu'}{\sigma}}.
\end{equation}
Since, by \eqref{eq:clt2rt2} the quantities inside the exponentials fluctuate at scale $\sigma \to \infty$,
it may easily be checked that this has no scaling limit under any deterministic recentering and rescaling.

\subsection{Sampling with replacement}
\label{sec:nodimers_replacement}

We would like to briefly remark that the phase transition at $p = \fr{2}{3}$ corresponds to the following qualitative behavior
of typical independent sets in $\Qp$, which was already mentioned in \cite{brcgw} as well as in Sections \ref{sec:intro} and \ref{sec:iop} above.

When $p > \fr{2}{3}$, one way to sample the defect side of an independent set is to first determine the (random) number of
singletons one must include, and then sample them \emph{with replacement} from a particular probability distribution which
depends on the percolation configuration and weights each singleton $v$ proportionally to $\phi_v$.
This works because when $p > \fr{2}{3}$, the typical number of singletons is small enough that there will be no repeats in the
sample with high probability, and in effect we get a sample of a \emph{subset} with the correct distribution.
For reasons which we do not delve into presently but which were thoroughly explored in \cite{brcgw},
the exponential form of the solution arises from this scheme of sampling with replacement.

However, when $p \leq \fr{2}{3}$, the number of singletons that must be sampled is large enough that when sampling with replacement
one begins to see repeated singletons.
This situation was analyzed in \cite{brcgw} as a non-uniform random birthday problem, and the outcome of this analysis is that
a correction of $- \fr{\sigma^2}{2}$ in $\mu'$ is necessary when using the exponential formula arising from sampling with replacement.

In sampling with replacement, there are still no dimers for $p \in (2 - \sqrt{2}, \fr{2}{3})$, meaning that it is very unlikely
that two \emph{neighboring vertices} will both be sampled.
This leads to the interpretation that the transition happens when certain \emph{individual singletons} gain an outsized probability
in the aforementioned random probability distribution.
In effect, even when there are no dimers, this transition marks the point where the influence of singletons becomes lopsided and
a small number of high-relevance singletons begin to make a meaningful contribution beyond the aggregate behavior.

%% file: sections/dimers.tex
%!TEX root =../main.tex
\section{When dimers begin to appear}
\label{sec:dimers}

We now arrive at our main contribution in this paper, which is an analysis of model in the regime of parameters
where typical independent sets will contain dimers in their defect sides.
This begins to happen when $p \leq 2 - \sqrt{2} \approx 0.586$, and our goal in this section is to prove that the set of singletons and the set of dimers
in the defect side of a uniformly random independent set behave roughly independently from each other for a large range of $p$ below this value,
extending all the way below $p = \f12$.
This will be formalized in Proposition \ref{prop:dimers} below.
As discussed in Section \ref{sec:intro_results}, this allows us to resolve the question posed by \cite{ks} about the existence of a phase
transition at $p=\f12$, which is a natural point of phase transition for many other features of the percolated hypercube.
We reiterate that, interestingly, our results show that there is actually \emph{no} phase transition at $p=\f12$ for the behavior of independent sets in $\Qp$.

However, our analysis does uncover a \emph{different} point of potential phase transition below $p=\f12$.
To explain this, first note that in this section we always restrict to the regime of $p > 2 - 2^{2/3} \approx 0.413$, so that there are no trimers in the defect side
(by Proposition \ref{prop:smallpolymers}).
But even for certain $p$ above this value where more complex clustering behavior arises as \emph{adjacent} singletons and dimers merge to form trimers,
it seems that the singletons and dimers will start to interact in a more subtle way.
Specifically, if one tried to sample the singleton set and the dimer set independently from one another, we expect that there will be some singleton and some dimer
which contain \emph{the same} vertex.
This is analogous to the phase transition at $p=\f23$ discussed in Section \ref{sec:nodimers_replacement} above, where in sampling singletons with replacement,
one vertex will be sampled multiple times if $p \leq \f23$, even though there are no \emph{adjacent} vertices sampled until $p \leq 2 - \sqrt{2} \approx 0.586$.

Our proof aims to cover the full regime of parameters $p$ for which the above ``overlapping samples'' behavior does not occur, and it is valid for all $p$
above a certain threshold which we denote by $\f12 - \gamma$ for some $\gamma > 0$.
Although we do not prove this, one might guess after observing the structure of our argument that the threshold we obtain is indeed sharp,
in the sense that the singleton set and dimer set will not be independent from one another below it.
One should also note that while we will simply be satisfied herein with proving that this threshold is strictly less than $\f12$, one may easily numerically approximate
the threshold as about $0.465$; in particular, it is strictly larger than $2 - 2^{2/3} \approx 0.413$.

Another interesting point that was mentioned briefly in Remark \ref{rmk:worstvsavg} is that a \emph{worst-case} analysis does not suffice for the full regime of independence.
To be specific, if we sampled the singleton set $S_1$ and then the dimer set $S_2$ independently, then there are certain worst-case choices of $S_1$ for which the
independently sampled $S_2$ will actually overlap with $S_1$, if $p$ is low enough (below roughly $0.548$).
To overcome this, we must perform an \emph{average-case} analysis (in Sections \ref{sec:dimers_secondstrategy} and \ref{sec:dimers_adj}),
which will allow us to prove that the singleton set and dimer set do not interact for $p > \f12 - \gamma$ by averaging over the singleton set $S_1$.
However, we do also provide a short non-rigorous overview of the worst-case analysis in Section \ref{sec:dimers_firststrategy} as a warm-up to the more
technical later sections, as many of the relevant tools can be introduced in this easier setting.

Before stating the main result of this section, Proposition \ref{prop:dimers} below, we summarize the above few paragraphs with the following diagram.
\begin{equation}
\begin{tikzpicture}[scale=60]
  \draw[<->, ultra thick] (0.365,0) -- (0.625,0);

  \foreach \x/\lab/\h in {0.413/{$2-2^{\f23}$}/0.006, 0.465/{$0.464...$}/0.004, 0.5/{$\f12$}/0.002, 0.548/{$0.548...$}/0.004, 0.586/{$2-2^{\f12}$}/0.006}
    {
      \draw (\x,\h) -- (\x,-\h);
      \node[above] at (\x,\h+0.001) {\lab};
    }

  \node[below] at (0.39,0) {trimers appear};
  \node[below] at (0.604,0) {no dimers};

  \draw[decorate,decoration={brace,mirror,amplitude=6pt}]
        (0.465,-0.005) -- (0.548,-0.005)
        node[midway,below=6pt, xshift=-5] {$\substack{\text{only average-case analysis can} \\ \text{show that singletons and dimers are} \\ \text{independent (Sections \ref{sec:dimers_secondstrategy} and \ref{sec:dimers_adj})}}$};

  \draw[decorate,decoration={brace,mirror,amplitude=6pt}]
        (0.548,-0.007) -- (0.586,-0.007)
        node[midway,below=6pt, xshift=5] {$\substack{\text{worst-case analysis shows} \\ \text{singletons and dimers are} \\ \text{independent (Section \ref{sec:dimers_firststrategy})}}$};

  \draw[decorate,decoration={brace,mirror,amplitude=6pt}]
        (0.413,-0.007) -- (0.465,-0.007)
        node[midway,below=6pt, xshift=-5] {$\substack{\text{singletons and dimers are not} \\ \text{independent (conjecturally)}}$};
\end{tikzpicture}
\end{equation}
We also emphasize that, although we do not pursue this presently, we expect that the techniques developed in this section will
prove fruitful for analyzing the regime of non-independence with values of $p$ down to $2 - 2^{2/3}$ and even beyond, when
trimers and other higher-order polymers begin to appear.

\subsection{Statement of main proposition}
\label{sec:dimers_statement}

In order to state our main proposition, we remind the reader of our running notation, introduced first in
Section \ref{sec:iop}.
As in Section \ref{sec:smallpolymers} above, to avoid notational overload
we drop the superscript $\cH$ as every statement in this section relates
two quantities which correspond to the \emph{same side}.
So, for instance, $\Sep_{\leq 2}$ means $\Sep_{\leq 2}^\cH$ here.
Recall that
\begin{align}\label{keynot234}
    \polymerpf &= \sum_{S \in \Sep} \prod_{\fp \in S} \phi_\fp, \\
    \label{keynot235}
    \polymerpf_{\leq 2} &= \sum_{S \in \Sep_{\leq 2}} \prod_{\fp \in S} \phi_\fp, \\
    \label{keynot236}
    \polymerpf_1 &= \sum_{S \in \Sep_1} \prod_{v \in S} \phi_v, \\
    \label{keynot237}
    \Delta &= \sum_{\fd \in \Dim} \phi_\fd,
\end{align}
as well as the corresponding notions of $\fakepolymerpf, \fakepolymerpf_{\leq 2}, \fakepolymerpf_1$, and $\tilde{\Delta}$
which are obtained by replacing each $\phi_\fp$ by $\fphi_\fp$, et cetera.
We remind the reader that $\fakepolymerpf_1 = \polymerpf_1$ since $\phi_v = \fphi_v$ for singletons $v$,
but that $\fakepolymerpf_{\leq 2} \neq \polymerpf_{\leq 2}$ and $\Delta \neq \tilde{\Delta}$ since $\fphi_\fd \neq \phi_\fd$
in general for dimers $\fd$.

The following proposition is the main result of this section, and holds for both $\cH = \cE$ and $\cH = \cO$
(although the $\cH$ superscript is only present implicitly in the statement).

\begin{proposition}
\label{prop:dimers}
    There exists a universal constant $\gamma > 0$ such that for $p > \fr{1}{2} - \gamma$,
    \begin{equation}
        \polymerpf_{\leq 2} \psim \polymerpf_1 \cdot e^\Delta
        \qquad \text{and} \qquad
        \fakepolymerpf_{\leq 2} \psim \polymerpf_1 \cdot e^{\tilde{\Delta}}.
    \end{equation}
\end{proposition}

Together with Propositions \ref{prop:smallpolymers} and \ref{prop:smallpolymers_full} which state that
for $p > 2 - 2^{2/3}$ we have
\begin{equation}
    \polymerpf \psim \polymerpf_{\leq 2}
    \qquad \text{and} \qquad 
    \fakepolymerpf \psim \fakepolymerpf_{\leq 2},
\end{equation}
Proposition \ref{prop:dimers} implies Corollary \ref{cor:approxinprob_iop}, which states that
\begin{equation}
    \polymerpf \psim \fakepolymerpf \cdot e^{\Delta - \tilde{\Delta}}
\end{equation}
for all $p$ larger than some threshold strictly below $\fr{1}{2}$.
Combining this with Proposition \ref{prop:polymerdecomp}, namely that
\begin{equation}
    \Cnt \psim \polymerpf^\cE + \polymerpf^\cO,
\end{equation}
completes our in-probability approximation of $\Cnt$, as discussed in Section \ref{sec:iop}.

Note that we will focus on proving the first of the two statements in Proposition \ref{prop:dimers}, with no tildes.
As will be expanded upon after the statement of Lemma \ref{lem:adjupperbound} below,
our proof of the tilde-free statement will also easily yield a proof of the tilde statement,
simply because $\fphi_\fp \leq \phi_\fp$.

\subsection{Proof overview and section outline}
\label{sec:dimers_outline}

We will reinterpret the statement of Proposition \ref{prop:dimers} probabilistically as follows.
We would like to understand the distribution of $S_{\leq 2} \in \Sep_{\leq 2}$ sampled from the model with partition
function $\polymerpf_{\leq 2}$, which contains both singletons and dimers, and which is sampled with probability 
proportional to
\begin{equation}
    \prod_{\text{singleton } v \in S_{\leq 2}} \phi_v \cdot \prod_{\text{dimer } \fd \in S_{\leq 2}} \phi_\fd.
\end{equation}
We will compare this to the alternate sampling scheme whereby we sample the singleton part and the dimer part independently.
First sample a well-separated singleton set $S_1 \in \Sep_1$ from the model with partition function $\polymerpf_1$,
i.e. with probability proportional to $\prod_{v \in S_1} \phi_v$ (but note that we are restricting to only $S_1 \in \Sep_1$).
Then we sample a dimer set $S_2$ by simply including each dimer independently with probability $\fr{\phi_\fd}{1+\phi_\fd}$,
so that the probability of a particular dimer set is proportional to $\prod_{\fd \in S_2} \phi_\fd$.
For the dimer set $S_2$, unlike the singleton set $S_1$, we have no separation
constraint, but the set $S_2$ we sample will turn out to naturally be well-separated with high probability.
Proposition \ref{prop:dimers} amounts to showing that the distribution of $S_{\leq 2}$ is close to that of $S_1 \cup S_2$.

The key difference, a priori, between $S_{\leq 2}$ and $S_1 \cup S_2$ is that we force the singleton and dimer parts of $S_{\leq 2}$
to not interact, meaning that no singletons are adjacent to any dimers in $S_{\leq 2}$.
On the other hand, there is no such restriction in $S_1 \cup S_2$ as the two parts are sampled independently.
In order to prove Proposition \ref{prop:dimers}, we will introduce and analyze an auxiliary random variable associated to
the singleton set $S_1$, denoted by $\Adj(S_1)$, which controls the probability that the independently sampled dimer set $S_2$
will interact nontrivially with $S_1$.

\begin{definition}
\label{def:adj}
For a set $S \sse \cH$ and a dimer $\fd$, we write $\fd \sim S$ if $\fd$ shares a neighbor with a vertex in $S$.
Note that this includes the case where $S$ contains a vertex of $\fd$ itself.
Now for $S \sse \cH$, define
\begin{equation}
    \Adj(S) \coloneqq \sum_{\fd \sim S} \phi_\fd.
\end{equation}
For future reference, we also define $\tAdj(S)$ in the same way but with $\phi_\fd$ replaced by $\fphi_\fd$.
\end{definition}

In Section \ref{sec:dimers_sep} below, we will prove a convenient representation for the dimer part
of the partition function $\polymerpf_{\leq 2}$.
This will motivate the definition of $\Adj(S)$ and it also shows that the dimers do not interact with themselves.
As will be explained further in that section, showing that singletons and dimers behave independently
amounts to showing that $\Adj(S_1)$ is small, for typical samples of the well-separated singleton set $S_1$.

An initial idea for proving this is to show that in fact, for \emph{any} well-separated singleton set $S_1$ that could be sampled,
$\Adj(S_1)$ is uniformly small.
This argument is outlined briefly in Section \ref{sec:dimers_firststrategy} below.
However, since this \emph{worst-case} strategy does not work for $p < 0.548$, that section is presented somewhat non-rigorously
and should be considered as a warm-up for later sections which cover the full range of $p$ in Proposition \ref{prop:dimers}.

In Section \ref{sec:dimers_secondstrategy}, we present the setup for the \emph{average-case} argument, where we prove that
$\Adj(S_1)$ is small for \emph{most} well-separated singleton sets $S_1 \in \Sep_1$ that are sampled from the model with partition
function $\polymerpf_1$.
For this, we will introduce the auxiliary (random) measure $\pi$ on subsets of $\cH$, which samples $S$ with probability
\begin{equation}
\label{eq:pidef}
    \pi[S] = \fr{1}{\fakepolymerpf} \prod_{v \in S} \phi_v,
\end{equation}
or in other words includes each vertex $v \in \cH$ independently with probability $\fr{\phi_v}{1 + \phi_v}$.
With this definition in hand, our result will turn out to reduce to the following statement:
\begin{equation}
\label{eq:adjsmall_initial}
    \pi[\Adj(S) > \eps | S \in \Sep_1] \pto 0
\end{equation}
for some choice of $\eps = \eps_d \to 0$.

To prove \eqref{eq:adjsmall_initial}, in Section \ref{sec:dimers_adj}
we will decompose the sum defining $\Adj(S)$ into polynomially many pieces, such that each piece
consists of of variables which are independent under $\pi$.
This will allow us to apply various standard estimates of binomial probabilities via the relative entropy function,
although much care will be needed as all of the binomial random variables in question have
\emph{random} specifications in terms of the percolation configuration.
As such, Section \ref{sec:dimers_adj} is the most technical part of this entire paper.

\subsection{A convenient representation of the partition function}
\label{sec:dimers_sep}

In this section we show that the dimer part of the polymer configuration $S_{\leq 2}$ is quite well-behaved for the regime
of parameters we consider.
Specifically, it behaves similarly to sampling dimers independently, and
we may rewrite the dimer part of the partition function $\polymerpf_{\leq 2}$ in a convenient exponential form
involving the quantity $\Adj(S)$ introduced in Definition \ref{def:adj}.

This should be compared to the exponential form obtained in our previous work \cite{brcgw} for the singleton partition function
when $p > \f23$ (which was re-proven in Section \ref{sec:nodimers_approx} above).
Our result here for dimers has an interpretation analogous to the one discussed in Section \ref{sec:nodimers_replacement},
where one can consider instead sampling the dimers \emph{with replacement} after choosing the correct number of dimers to sample,
and no individual dimer will have a weight large enough to be sampled twice.

Let us begin by recording a definition.

\begin{definition}\label{def:compat}
For any $S_1 \in \Sep_1$, we denote by $\Cpt_2(S_1)$ the collection of dimers $\fd$ with $\fd \not \sim S_1$, meaning that
$S_1 \cup \{\fd\}$ remains a valid polymer configuration.
\end{definition}

With this notation, we have the following representation for the singleton-and-dimer partition function.
\begin{align}
    \polymerpf_{\leq 2} &= \sum_{S_{\leq 2} \in \Sep_{\leq 2}} \prod_{\fp \in S_{\leq 2}} \phi_\fp \\
    &= \sum_{S_1 \in \Sep_1} \left( \prod_{v \in S_1} \phi_v \cdot
    \sum_{\substack{S_2 \in \Sep_2 \\ S_2 \sse \Cpt_2(S_1)}} \prod_{\fd \in S_2} \phi_\fd \right).
\label{eq:zleq2decomp}
\end{align}
We will first show that the constraint $S_2 \in \Sep_2$ is superfluous, i.e.\ that
\begin{equation}
\label{eq:dimersepconstraint}
    \sum_{\substack{S_2 \in \Sep_2 \\ S_2 \sse \Cpt_2(S_1)}} \prod_{\fd \in S_2} \phi_\fd
    \psim 
    \sum_{S_2 \sse \Cpt_2(S_1)} \prod_{\fd \in S_2} \phi_\fd,
\end{equation}
with the approximation being uniform in $S_1 \in \Sep_1$.
Since
\begin{equation}
\label{eq:dimerbinomial}
    \sum_{S_2 \sse \Cpt_2(S_1)} \prod_{\fd \in S_2} \phi_\fd = \prod_{\fd \not \sim S_1} (1 + \phi_\fd),
\end{equation}
we can interpret the approximation \eqref{eq:dimersepconstraint} as saying that, in order to sample a
collection $S_2$ of dimers compatible with any singleton set $S_1 \in \Sep_1$, with the correct distribution as given
by \eqref{eq:zleq2decomp}, we may independently include each dimer $\fd \in \Cpt_2(S_1)$ in the set $S_2$ with probability $\fr{\phi_\fd}{1 + \phi_\fd}$.
Thus, this is a statement about non-collision (with one another) of independently drawn dimers.

We will also apply the approximation $1 + \phi_\fd \approx e^{\phi_\fd}$ to the right-hand side of \eqref{eq:dimerbinomial},
and the smallness of $\phi_\fd$ for a certain range of $p$ which suffices for our purposes will ensure that the error does not accumulate as the product over all dimers $\fd \not \sim S_1$ is taken.
We combine this approximation with \eqref{eq:dimersepconstraint} in the following lemma.

\begin{lemma}
\label{lem:dimer_sep}
For $p > 0.391$, we have
\begin{equation}
    \sum_{\substack{S_2 \in \Sep_2 \\ S_2 \sse \Cpt_2(S_1)}} \prod_{\fd \in S_2} \phi_\fd
    \psim \Exp{\sum_{\fd \not \sim S_1} \phi_\fd} = \Exp{\Delta - \Adj(S_1)},
\end{equation}
where the approximation is uniform among all $S_1 \sse \cH$ (recall what this means from above the statement of
Theorem \ref{thm:main}).
The same statement also holds with $\phi_\fd$ replaced by $\fphi_\fd$, $\Delta$ replaced by $\tilde{\Delta}$,
and $\Adj$ replaced by $\tAdj$.
\end{lemma}

We remark that the threshold $p > 0.391$ arises from a simple moment calculation and is not related to the threshold
$\fr{1}{2} - \gamma$ of our main result which arises from other considerations.
Additionally, the equality in the statement follows directly from the definition of $\Delta$ in \eqref{eq:deltadef}
and of $\Adj$ in Definition \ref{def:adj}, so it only remains to prove the uniform in-probability approximation.

\begin{proof}[Proof of Lemma \ref{lem:dimer_sep}]
We first derive an upper bound on the difference between the two sides of \eqref{eq:dimersepconstraint}.
Note that if $S_2 \sse \Cpt_2(S_1)$ but $S_2 \notin \Sep_2$, then there are two dimers in $S_2$ which are adjacent or overlapping,
i.e.\ there are distinct $\fd_1, \fd_2 \in S_2$ with $\fd_1 \sim \fd_2$.
This implies that
\begin{equation}
    0 \leq \sum_{\substack{S_2 \notin \Sep_2 \\ S_2 \sse \Cpt_2(S_1)}} \prod_{\fd \in S_2} \phi_\fd
    \leq \left(\sum_{\substack{\fd_1 \neq \fd_2 \\ \fd_1 \sim \fd_2}} \phi_{\fd_1} \phi_{\fd_2} \right) \cdot \Exp{\sum_{\fd \not \sim S_1} \phi_\fd}.
\end{equation}
To see this, note that in the Taylor expansion of the exponential, every collection of distinct dimers $\fd$ will lead to
a term which is the product of the relevant $\phi_\fd$, with a coefficient of $1$ (this is similar to the strategy used in
the proof of Proposition \ref{prop:smallpolymers_full} above).
Note that the prefactor above does not depend on $S_1$.

Now consider the right-hand side of \eqref{eq:dimersepconstraint} which is equal to the right-hand side of \eqref{eq:dimerbinomial}.
Since
\begin{equation}
    e^{x - x^2/2} \leq 1 + x \leq e^x
\end{equation}
for all $x \geq 0$, we have
\begin{equation}
    \Exp{\sum_{\fd \not \sim S_1} \phi_\fd - \fr{1}{2} \sum_{\fd \not \sim S_1} \phi_\fd^2}
    \leq \sum_{S_2 \sse \Cpt_2(S_1)} \prod_{\fd \in S_2} \phi_\fd
    \leq \Exp{\sum_{\fd \not \sim S_1} \phi_\fd}.
\end{equation}
Moreover, since every $\phi_\fd$ is positive, we can drop the restriction that $\fd \not \sim S_1$ in the
second sum in the first exponential above, obtaining
\begin{equation}
    \Exp{- \fr{1}{2} \sum_{\fd \in \Dim} \phi_\fd^2} \cdot
    \Exp{\sum_{\fd \not \sim S_1} \phi_\fd} 
    \leq \sum_{S_2 \sse \Cpt_2(S_1)} \prod_{\fd \in S_2} \phi_\fd
    \leq \Exp{\sum_{\fd \not \sim S_1} \phi_\fd},
\end{equation}
wherein again the pre-factor in the left-hand side does not depend on $S$.
Thus the non-tilde case of the lemma will be proved if we can show that both
\begin{equation}
    \sum_{\substack{\fd_1 \neq \fd_2 \\ \fd_1 \sim \fd_2}} \phi_{\fd_1} \phi_{\fd_2} \pto 0
    \qquad \text{and} \qquad
    \sum_{\fd \in \Dim} \phi_\fd^2 \pto 0,
\end{equation}
which, by positivity of $\phi_\fd$ and the fact that $\fd \sim \fd$ for any $\fd \in \Dim$, is implied by
\begin{equation}
\label{eq:togotozero}
    \sum_{\fd_1 \sim \fd_2} \phi_{\fd_1} \phi_{\fd_2} \pto 0.
\end{equation}
This follows from a simple moment calculation which is relegated to the appendix.
Namely, Lemma \ref{lem:dimermoments} (b) states that the expectation of the above quantity
converges to $0$ whenever $p > 0.391$.
Moreover, since the same algebra goes through with $\phi_\fd$ replaced by $\fphi_\fd$,
and since $\fphi_\fd \leq \phi_\fd$, \eqref{eq:togotozero} also implies the tilde case of the lemma.
\end{proof}

With Lemma \ref{lem:dimer_sep} in hand, we have effectively relegated the interaction between the singleton part $S_1$
and the dimer part $S_2$ to the study of the quantity $\Adj(S_1)$, and as such
\begin{equation}
    \textbf{for the rest of this section, we will use } S \textbf{ instead of } S_1 \textbf{ to denote the singleton set,}
\end{equation}
to reduce notational clutter.

For instance, recalling the representation \eqref{eq:zleq2decomp} for $\polymerpf_{\leq 2}$,
Lemma \ref{lem:dimer_sep} implies that
\begin{equation}
    \polymerpf_{\leq 2} \psim \sum_{S \in \Sep_1} \left( \prod_{v \in S} \phi_v \cdot \Exp{\Delta - \Adj(S)} \right),
\end{equation}
crucially using the fact that the approximation in Lemma \ref{lem:dimer_sep} is uniform over all singleton sets.
Recalling the definition of $\pi$ from \eqref{eq:pidef}, we obtain
\begin{equation}
    \polymerpf_{\leq 2} \psim e^\Delta \cdot \fakepolymerpf \cdot \E_\pi \left[\ind{S \in \Sep_1} \cdot e^{-\Adj(S)}\right].
\end{equation}
Now recalling the definitions of $\fakepolymerpf$ and $\polymerpf_1$ from \eqref{eq:fakepolymerpfdef} and
\eqref{eq:polymerpfsubkdef}, we obtain
\begin{equation}
\label{eq:pf_mgf}
    \polymerpf_{\leq 2} \psim e^\Delta \cdot \polymerpf_1 \cdot \E_\pi \left[ e^{-\Adj(S)} \middle| S \in \Sep_1 \right].
\end{equation}
So, to prove the first approximation in Proposition \ref{prop:dimers}, which states that $\polymerpf_{\leq 2} \psim \polymerpf_1 \cdot e^\Delta$,
it is equivalent to show that
\begin{equation}
\label{eq:adjsmall_goal}
    \E_\pi \left[ e^{-\Adj(S)} \middle| S \in \Sep_1 \right] \pto 1.
\end{equation}
First of all, $\Adj(S) \geq 0$, and so the above expectation is always at most $1$.
To show that it converges to $1$, we need to prove some form of a corresponding upper bound on $\Adj(S)$.
This task will occupy us for the remainder of this section of the paper.

Note also that all of the above manipulations can be done with $\phi_\fd$ replaced by $\fphi_\fd$ everywhere, and we will obtain
\begin{equation}
    \fakepolymerpf_{\leq 2} \psim e^{\tilde{\Delta}} \cdot \polymerpf_1 \cdot \E_\pi \left[ e^{-\tAdj(S)} \middle| S \in \Sep_1 \right].
\end{equation}
Once we have a workable upper bound for $\Adj(S)$, the same bound will also follow for $\tAdj(S)$, since $\fphi_\fd \leq \phi_\fd$ deterministically.
So the tilde case of Proposition \ref{prop:dimers} follows from the same analysis as the non-tilde case, which we focus on henceforth.

\subsection{A worst-case argument to bound the contributions of adjacent dimers}
\label{sec:dimers_firststrategy}

In this section we present a sketch of a ``worst-case'' type argument to prove \eqref{eq:adjsmall_goal}
by showing that $\Adj(S)$ is small for all sets $S \sse \cH$, assuming a simple \emph{a priori}
bound on the size of $S$.
We will see what this size constraint should be shortly, and then find the worst-case $S$, which maximizes $\Adj(S)$
while satisfying this constraint.
If this worst-case $S$ has $\Adj(S) \pto 0$, then the proof of \eqref{eq:adjsmall_goal} will be finished.

Unfortunately, this strategy does not work for the full range of $p$ which we wish to consider, and in later sections
we will present an ``average-case'' argument which proves \eqref{eq:adjsmall_goal} in what we believe to be the full
regime of parameters for which it holds.
Nevertheless, some ideas used later will already surface in this easier worst-case analysis, so this 
somewhat non-rigorous section should be viewed as a warm-up.

\subsubsection{Size constraint}

First we remark that, as can be proven via simple moment calculations such as those done in Lemma \ref{lem:moments},
the typical size of a set $S$ sampled from $\pi$ is approximately $\f12 (2-p)^d$.
Since vertices are included in $S$ independently from one another, it is easy to show (via a Chernoff bound for instance)
that
\begin{equation}
\label{eq:setsmall_worstcase}
    \pi\left[|S| > C (2-p)^d\right] \leq e^{-c (2-p)^d}(1+o(1)),
\end{equation}
for some constants $C, c > 0$, where the $o(1)$ term converges to $0$ in probability.

On the other hand, by the definitions of $\fakepolymerpf$ from \eqref{eq:fakepolymerpfdef} and $\polymerpf_1$ from \eqref{eq:polymerpfsubkdef},
as well as recalling the definition of $\pi$ from \eqref{eq:pidef}, we have
\begin{equation}
    \pi\left[ S \in \Sep_1 \right] = \fr{\polymerpf_1}{\fakepolymerpf} \psim \fr{\polymerpf_1}{\fakepolymerpf_{\leq 2}},
\end{equation}
using Proposition \ref{prop:smallpolymers_full} for the approximation of $\fakepolymerpf$ by $\fakepolymerpf_{\leq 2}$.
Now, by similar manipulations as just done in the previous section to derive \eqref{eq:pf_mgf}, we may obtain
\begin{align}
    \fakepolymerpf_{\leq 2} &\psim e^{\tilde{\Delta}} \cdot \polymerpf_1 \cdot \E_\pi \left[ e^{-\tAdj(S)} \middle| S \in \Sep_1 \right] \\
    &\leq e^{\tilde{\Delta}} \cdot \polymerpf_1,
\end{align}
using simply the fact that $\tAdj(S) \geq 0$ in the last inequality.
Thus we obtain
\begin{equation}
\label{eq:seplarge_worstcase}
    \pi[S \in \Sep_1] \geq e^{-\tilde{\Delta}} (1 - o(1)),
\end{equation}
where again the $o(1)$ term converges to $0$ in probability.

Now, combining \eqref{eq:setsmall_worstcase} and \eqref{eq:seplarge_worstcase}, we find that
\begin{equation}
\label{eq:sizeconstraint}
    \pi\left[ |S| > C (2-p)^d \middle| S \in \Sep_1 \right] \leq \Exp{-c (2-p)^d + \tilde{\Delta}} \cdot (1+o(1)),
\end{equation}
where again $o(1) \pto 0$.
Now, moment calculations which can be found in Lemma \ref{lem:dimermoments} show that $\tilde{\Delta}$
concentrates at constant scale around its mean, which is
\begin{equation}
    \tilde{\mu}_2 = \fr{d(d-1)}{4} \Fr{(2-p)^2}{2}^d
\end{equation}
as defined first in \eqref{eq:mu2tildedefintro}.
Since $2-p < 2$ we have $(2-p)^d \gg \Fr{(2-p)^2}{2}^d$, which means that the bound \eqref{eq:sizeconstraint}
tends to $0$.
Thus, to prove \eqref{eq:adjsmall_goal}, it suffices to show that
\begin{equation}
    \max \left\{ \Adj(S) : |S| \leq C (2-p)^d \right\} \pto 0,
\end{equation}
where $C$ is as in \eqref{eq:setsmall_worstcase}.

\subsubsection{Worst-case sum}

Given that $|S| \leq C (2-p)^d$, since each vertex is only adjacent to at most $d^4$ distinct dimers,
the number of dimers $\fd \sim S$ is at most $\poly(d) \cdot (2-p)^d$.
Thus, to bound
\begin{equation}
    \Adj(S) = \sum_{\fd \sim S} \phi_\fd
\end{equation}
for $S$ with this size constraint, it suffices to bound the sum of the \emph{largest}
$\poly(d) \cdot (2-p)^d$ variables $\phi_\fd$.
To do this, we need to get a better understanding of how many dimers $\fd$ have each possible value of $\phi_\fd$.

Since we are just looking for an upper bound, let us first use the fact that if $\fd = \{u,v\}$ then
\begin{equation}
    \phi_\fd \leq 4 \cdot 2^{-\N(u) - \N(v)},
\end{equation}
because there are exactly two vertices which neighbor both $u$ and $v$ in the unpercolated hypercube $\Q$,
so $\N(\{u,v\}) \geq \N(u) + \N(v) - 2$.
Now we will use the fact that $\N(u)$ are independent $\Bin(d,p)$ variables (with respect to the randomness of the
percolation $\Qp$) to bound the number of dimers $\fd = \{u,v\}$ with various (low) values of $\N(u) + \N(v)$.

Using standard entropy estimates on binomial probabilities, many of which are included in Appendix \ref{sec:entropy},
we have
\begin{equation}
    \P[\N(u) = a] = \poly(d) \cdot 2^{-d H_p(r)},
\end{equation}
where $r = \fr{a}{d}$ and $H_p$ is the binary relative entropy function
\begin{equation}
\label{eq:hpdef}
    H_p(q) \coloneqq q \log_2 \fr{q}{p} + (1-q) \log_2 \fr{1-q}{1-p}.
\end{equation}
Now since there are $\poly(d) \cdot 2^d$ dimers in total and $\N(u)$ are independent for $u \in \cH$,
the number of dimers $\fd = \{u,v\}$ which have $\N(u) = a$ and $\N(v) = b$
is thus within a (fixed) polynomial factor of
\begin{equation}
    2^{d\left(1 - H_p\Fr{a}{d} - H_p\Fr{b}{d} \right)}
\end{equation}
with high probability.
Now by convexity of relative entropy, we have
\begin{equation}
    1 - H_p\Fr{a}{d} - H_p\Fr{b}{d} \leq 1 - 2 \cdot H_p\Rnd{\fr{a+b}{2d}},
\end{equation}
and so the number of dimers $\fd = \{u,v\}$ for which $\N(u) + \N(v) = c$ is, with high probability, at most
\begin{equation}
\label{eq:ensum}
    \poly(d) \sum_{a+b=c} 2^{d \left( 1 - H_p \Fr{a}{d} - H_p \Fr{b}{d} \right)}
    \leq \poly(d) \cdot 2^{d \left(1 - 2 \cdot H_p \Fr{c}{2d} \right)}.
\end{equation}
In the above bound we have used the fact that the sum in the left-hand side of \eqref{eq:ensum}
has only polynomially many terms.

For notational convenience, let us define
\begin{equation}
    h_c^* \coloneqq 1 - 2 \cdot H_p\Fr{c}{2d}.
\end{equation}
With this definition, we see that the sum of the $\poly(d) \cdot (2-p)^d$ largest values of $\phi_\fd$ is at most
\begin{equation}
\label{eq:bigsum}
    \poly(d) \sum_{c=0}^{c_{\max}} 2^{d \cdot h^*_c} \cdot 2^{-c},
\end{equation}
with high probability, where $c_{\max}$ is the maximal value of $c$ for which
$2^{d \cdot h^*_c} \leq \poly(d) \cdot (2-p)^d$, which ensures that the number of dimers $\fd$ satisfying
$\fphi_\fd = 2^{-c}$ is at most $\poly(d) \cdot (2-p)^d$.

Finally, since there are only polynomially many terms in the sum \eqref{eq:bigsum}, we will have $\Adj(S) \pto 0$
uniformly for $|S| \leq C (2-p)^d$ if we can show that
\begin{equation}
    \poly(d) \cdot 2^{d h^*_c - c} \ll 1
\end{equation}
whenever $c$ is small enough that $2^{d h^*_c} \leq \poly(d) \cdot (2-p)^d$.
In other words, setting $t = \fr{c}{2d}$, we need to have
\begin{equation}
\label{eq:optimization1}
    1 - 2 \cdot H_p(t) - 2 t < 0 \quad
    \text{for all} \quad t \in [0,1] \quad \text{satisfying} \quad
    1 - 2 \cdot H_p(t) < \log_2(2-p).
\end{equation}
This can be numerically checked, and it holds for $p > 0.549$, but does not hold for $p < 0.548$.
Thus, in order to achieve our goal of investigating the behavior near $p = \fr{1}{2}$, we will need a more
sophisticated argument which leverages the fact that \eqref{eq:adjsmall_goal} is an average
and so we can ignore worst-case sets $S$ if they are rare.

\subsection{The average-case argument}
\label{sec:dimers_secondstrategy}

In this section we lay the groundwork for the average-case argument to prove \eqref{eq:adjsmall_goal}, i.e.\ that
\begin{equation}
\label{eq:adjsmall_goal_again}
    \E_\pi \left[ e^{-\Adj(S)} \middle| S \in \Sep_1 \right] \pto 1,
\end{equation}
for all $p > \f12 - \gamma$, where $\Adj(S)$ is as in Definition \ref{def:adj} and $\pi$ is defined in \eqref{eq:pidef}.
This will finish the proof of Proposition \ref{prop:dimers}, as discussed in Section \ref{sec:dimers_sep}.
As mentioned at the beginning of Section \ref{sec:dimers}, we expect that the strategy we employ in this section and the next actually works for all $p$ 
such that the result does indeed hold, although we do not prove this.
The heuristic for this belief was outlined at the beginning of Section \ref{sec:dimers}, and is related to a phase
transition similar in spirit to the one which occurs at $p=\fr{2}{3}$, which was discussed in Section \ref{sec:nodimers_replacement}.

Note that since $\Adj(S) \geq 0$, to prove \eqref{eq:adjsmall_goal_again} it suffices to show that
\begin{equation}
\label{eq:adjconditioned}
    \pi\left[\Adj(S) > \eps \middle| S \in \Sep_1 \right] \pto 0
\end{equation}
for some choice of $\eps = \eps_d \to 0$ as $d \to \infty$.
In other words, instead of attempting to fully rule out sets with large $\Adj(S)$ as in the worst-case argument of the previous section,
we will show that such sets are rare under $\pi$ conditioned on $S \in \Sep_1$.
In fact, we will show a stronger statement, that
\begin{equation}
\label{eq:LB}
    \pi[\Adj(S) > \eps] \ll \pi[S \in \Sep_1] \cdot (1 + o(1))
\end{equation}
for some choice of $\eps = \eps_d$ decreasing to zero polynomially as a function of $d$, where the $o(1)$ term converges to $0$ in probability.
This will be achieved by the following two lemmas, which give a lower bound on the right-hand side and an upper bound on the left-hand side.

\begin{lemma}
\label{lem:seplowerbound}
For all $p \in (0,1)$,
\begin{equation}
    \pi[S \in \Sep_1] \geq \Exp{-\poly(d) \cdot \Fr{(2-p)^2}{2}^d} \cdot (1 - o(1)),
\end{equation}
where the $o(1)$ term tends to $0$ in probability.
\end{lemma}

\begin{lemma}
\label{lem:adjupperbound}
There exists a universal constant $\gamma > 0$ such that for $p > \fr{1}{2} - \gamma$,
there is some constant $\alpha_p > \fr{(2-p)^2}{2}$ such that
for any $\epsilon \geq \fr{1}{\poly(d)}$,
\begin{equation}
    \pi[\Adj(S) > \epsilon] \leq \Exp{- \fr{1}{\poly(d)} \cdot \alpha_p^d} \cdot (1 + o(1)),
\end{equation}
where the $o(1)$ term tends to $0$ in probability.
\end{lemma}

These two lemmas combine to prove \eqref{eq:LB} which proves \eqref{eq:adjsmall_goal_again} and subsequently,
by the discussion in Section \ref{sec:dimers_sep}, finishes the proof of Proposition \ref{prop:dimers} in the first (non-tilde) case.
As mentioned in Section \ref{sec:dimers_statement}, the tilde case follows from the non-tilde case simply because $\fphi_\fd \leq \phi_\fd$,
and so $\tAdj(S) \leq \Adj(S)$, meaning that the bound \eqref{eq:LB} with $\Adj(S)$ replaced by $\tAdj(S)$ follows immediately from the original bound.
Thus, Lemmas \ref{lem:seplowerbound} and \ref{lem:adjupperbound} complete the full proof of Proposition \ref{prop:dimers}.

We now prove Lemma \ref{lem:seplowerbound}, with the proof of Lemma \ref{lem:adjupperbound},
which is one of the key novelties of this article, deferred to the next subsection.

\begin{proof}[Proof of Lemma \ref{lem:seplowerbound}]
First note that
\begin{align}
    \fakepolymerpf &\psim \fakepolymerpf_{\leq 2} \tag*{(by Proposition \ref{prop:smallpolymers_full})} \\
    &= \sum_{S \in \Sep_1} \left( \prod_{v \in S} \phi_v \cdot \sum_{\substack{S_2 \in \Sep_2 \\ S_2 \sse \Cpt_2(S)}} \prod_{\fd \in S_2} \fphi_\fd \right) \label{eq:zleq2decomp_fake} \\
    &\psim \sum_{S \in \Sep_1} \left( \prod_{v \in S} \phi_v \cdot \Exp{\tilde{\Delta} - \tAdj(S)} \right) \tag*{(by Lemma \ref{lem:dimer_sep})} \\
    &\leq e^{\tilde{\Delta}} \cdot \sum_{S \in \Sep_1} \prod_{v \in S} \phi_v. \tag*{(since $\tAdj(S) \geq 0$)}
\end{align}
Note that at step \eqref{eq:zleq2decomp_fake}, we used the manipulation done in \eqref{eq:zleq2decomp} but with $\phi_\fd$ replaced by $\fphi_\fd$.
Additionally, in the application of Lemma \ref{lem:dimer_sep} above, it was crucial that the approximation in that lemma is uniform over $S$.

Now, recalling the definition of $\pi$ from \eqref{eq:pidef}, rearranging the above display shows that
\begin{align}
    \pi[S \in \Sep_1] &= \fr{1}{\fakepolymerpf} \sum_{S \in \Sep_1} \prod_{v \in S} \phi_v \\
    &\geq \Exp{-\tilde{\Delta}} \cdot (1-o(1)),
\end{align}
where $o(1) \pto 0$.
Now, as it turns out, $\tilde{\Delta}$ concentrates at around its mean
\begin{equation}
    \tilde{\mu}_2 = \fr{d(d-1)}{4} \Fr{(2-p)^2}{2}^d,
\end{equation}
this notation being initially defined in \eqref{eq:mu2tildedefintro}.
This can be seen by moment calculations in the appendix, specifically in Lemma \ref{lem:dimermoments}.
The concentration result is also written explicitly as Lemma \ref{lem:dimerconcentration} below,
which states that
\begin{equation}
    \tilde{\Delta} - \tilde{\mu}_2 \pto 0.
\end{equation}
This finishes the proof.
\end{proof}

\subsection{Bounding the contribution of singleton sets with large dimer adjacency}
\label{sec:dimers_adj}

In this section we prove Lemma \ref{lem:adjupperbound}.
For convenience, we will replace the distribution $\pi$ with one which chooses each singleton $\fs$
independently with probability $\phi_\fs$ instead of the unwieldy ratio $\fr{\phi_\fs}{1 + \phi_\fs}$
(note that this makes sense since $\phi_\fs\le 1$).
This new measure stochastically dominates the old one (via inclusion of sets) since $\fr{\phi_\fs}{1 + \phi_\fs}\le \phi_\fs $.
Since $\Adj(S)$ is an increasing function of $S$ (recall from Definition \ref{def:adj} that it is the sum of $\phi_\fd$ over all $\fd \sim S$),
if we can show that the event $\{ \Adj(S) > \epsilon \}$ has small probability under this new measure, then it will
also have small probability under the old measure.

\begin{multline}
\textbf{Henceforth we denote by $\pi$ this new measure, which should not cause confusion} \\
\textbf{as we will not return to the old definition of $\pi$ for the rest of this section.}
\end{multline}

Now notice that, by rewriting the definition of $\Adj(S)$ from Definition \ref{def:adj}, we have
\begin{equation}
    \Adj(S) = \sum_{\fd \in \Dim} \phi_\fd \ind{\fd \sim S},
\end{equation}
and so broadly speaking our goal is to bound the tail of a sum of bounded random variables. Note that if the variables were i.i.d., one may have applied a version of Hoeffding's inequality
or a standard large deviations bound for sums of i.i.d. variables to get a bound on this tail.
However, the variables $\left\{ \phi_\fd \ind{\fd \sim S} \right\}_{\fd \in \Dim}$ are not independent under $\pi$
since many different dimers $\fd$ can be adjacent to one single vertex $v$, meaning that many different indicators $\ind{\fd \sim S}$ may be set to $1$ as a result of a single vertex being added to $S$.
Furthermore they are not identically distributed since the distribution of $\phi_\fd \ind{\fd \sim S}$ under $\pi$ (which samples vertices $v$ to include in $S$ based on their degree in $\Qp$)
depends on the percolation configuration near $\fd$.

Nonetheless, the dependence is local in nature and if the dimers are separated enough from each other, then indeed they correspond to independent variables $\phi_\fd \ind{\fd \sim S}$.
Specifically, the variables $\fd_1$ and $\fd_2$ are independent under $\pi$ when there is no vertex $v$ whose inclusion in $S$ can cause both $\fd_1 \sim S$ \emph{and} $\fd_2 \sim S$,
since the vertices are included in $S$ independently.
This is the case when the dimers have distance at least $3$ in the shared-neighbor graph, meaning that there is no vertex $v$ which
has a shared neighbor with $\fd_1$ and a shared neighbor with $\fd_2$, in the unpercolated hypercube $\Q$.

Hence to remedy the non-independence problem, we will partition the set of dimers into $\poly(d)$ \emph{layers} $\{\frL_\ell\}_\ell$
such that for each $\ell$, the dimers in $\frL_\ell$ are at shared-neighbor-distance at least $3$ from each other. One particularly quick way to see that this is possible is to recall that the chromatic number of a graph is at most the maximum degree plus one,
and each dimer is within distance $3$ of at most $\poly(d)$ other dimers.
Note that the classes $\frL_\ell$ are deterministic, i.e.\ they do not depend on the percolation configuration.

 It would be useful to have the following definition.

\begin{definition}
The \emph{layer} of a dimer $\fd$ is the unique integer $\ell$ such that $\fd \in \frL_\ell$.
\end{definition}

We will use the independence between dimers within each layer to get a good upper bound on the tail for the sum over each layer,
and then use a union bound over the only $\poly(d)$ many layers.
However, different dimers in the same layer may still have different distributions with different typical magnitudes,
and it is a priori unclear which types of dimers yield the dominant contribution to the sum. This necessitates another large deviation analysis. Towards this, we further decompose each layer in terms of the value of $\phi_\fd$, or, more precisely, by the number of neighbors
in $\Qp$ each of the constituent vertices have.
To make this precise, we fix an arbitrary ordering $\prec$ of the vertices and equate a dimer $\fd$ with its pair $(u,v)$
of vertices, where $u \prec v$.

\begin{definition}
The \emph{type} of a dimer $\fd = (u,v)$ is the ordered pair $(a,b)$ where $\N(u) = a$ and $\N(v) = b$.
We will also say that a singleton $\fs$ has \emph{type} $a$ if $\N(\fs) = a$.
\end{definition}

Note that there are only $(d+1)^2$ different possible types for a dimer.
For a dimer $\fd = (u,v)$ of type $(a,b)$, we have
\begin{equation}
\label{eq:typebound}
    \phi_\fd = 2^{-\N(\{u,v\})} \leq 2^{-\N(u) - \N(v) + 2} = 4 \cdot 2^{-a-b},
\end{equation}
because two vertices in the same side sharing a neighbor share exactly two neighbors in $\Q$, before percolation
(recall for instance the diagram \eqref{dimerpicture}).

Finally, due to the differences in the behaviors of dimers \emph{adjacent} to a vertex in $S$ and dimers
\emph{containing} a vertex in $S$, we separate the dimers by this distinction as well.

\begin{definition}
The \emph{status} of a dimer $\fd = (u,v)$ with respect to $S$ is an element
$\sigma \in \{\sfO_1, \sfO_2, \sfN\}$
with $\sigma = \sfO_1$ if $u \in S$, $\sigma = \sfO_2$ if $u \notin S, v \in S$ ($\sfO$ stands for ``overlap''),
and $\sigma = \sfN$ if $u,v \not \in S$ but one of them is adjacent to a singleton in $S$
($\sfN$ stands for ``neighbor'').
\end{definition}

We remark that \emph{dimers with the same type and status may still not have the same distribution} under $\pi$
due to the different possible values of $\N(\{u,v\})$ arising from the different ways in which the neighborhoods of
$u$ and $v$ may overlap.
However, the relationship \eqref{eq:typebound} will be enough for our purposes.
Broadly speaking, the rest of the proof is to bound the tail of each sum of variables corresponding to dimers
of the same layer, type, and status; we will then conclude with a union bound.

To start making this precise, let us denote by $\scrC(a,b;\ell;\sigma;S)$ the number of dimers of type $(a,b)$,
layer $\ell$, and status $\sigma$ with respect to $S$.
So, by \eqref{eq:typebound}, we have
\begin{equation}
    \Adj(S) \lesssim \sum_{a,b,\ell,\sigma} \scrC(a,b;\ell;\sigma;S) \cdot 2^{-a-b}.
\end{equation}
There are at most $\poly(d)$ terms in this sum, so $\Adj(S) > \epsilon$ implies that
for some choice of $a,b,\ell,\sigma$ we have
\begin{equation}
\label{eq:cbig}
    \scrC(a,b;\ell;\sigma;S) \cdot 2^{-a-b} > \fr{\epsilon}{\poly(d)},
\end{equation}
which is still at least $\fr{1}{\poly(d)}$ if $\epsilon > \fr{1}{\poly(d)}$.
So Lemma \ref{lem:adjupperbound} will follow from the following lemma, which states that,
with high probability in terms of the percolation configuration,
\eqref{eq:cbig} has low probability when $S$ is sampled from $\pi$.

\begin{lemma}
\label{lem:separated_adj}
There exists a universal constant $\gamma > 0$ such that for $p > \fr{1}{2}-\gamma$,
there is a constant $\alpha_p > \fr{(2-p)^2}{2}$ such that, for all $\eps = \eps_d > \poly(d)^{-1}$,
there is a function $f(d) > \poly(d)^{-1}$ (which may depend on $\eps_d$) such that  the following event
(determined by the percolation configuration) holds with probability tending to $1$.
\begin{equation}\label{keytailbound}
    \left\{ \text{For all choices of } a,b,\ell,\sigma, \text{ we have }
    \pi\Box{\scrC(a,b;\ell;\sigma;S) > \eps \cdot 2^{a+b}} \leq \Exp{- f(d) \cdot \alpha_p^d}
    \right\}.
\end{equation}
\end{lemma}

\begin{proof}[Proof of Lemma \ref{lem:separated_adj}]
The cases $\sigma \in \{\sfO_1,\sfO_2\}$ and $\sigma = \sfN$ are similar but with some minor differences.
We tackle the former case first.

\noindent \textbf{Case 1:} $\sigma = \sfO_1$. We only show the $\sfO_1$ case as the $\sfO_2$ case is
identical.

For each $a,b,\ell$, let $\scrN(a,b;\ell)$ denote the number of dimers of type $(a,b)$ in layer $\ell$;
note that this only depends on the percolation configuration and so is deterministic with respect to
the measure $\pi$.
So when $S$ is sampled from $\pi$ the count $\scrC(a,b;\ell;\sfO_1;S)$ is a binomial random variable with parameters
$\scrN(a,b;\ell)$ and $2^{-a}$, recalling that we are using a simplified version of $\pi$ which chooses
a singleton of type $a$ with probability $2^{-a}$, and also that dimers in a particular layer
are adjacent to $S$ independently due to their separation.
This implies that 
\begin{equation}
\label{eq:pibnd1}
    \pi\Box{\scrC(a,b;\ell;\sfO_1;S) > \epsilon \cdot 2^{a+b}} =
    \P\Box{\Bin(\scrN(a,b;\ell),2^{-a}) > \epsilon \cdot 2^{a+b}}.
\end{equation}
In order to estimate the right-hand side we will make use of various standard estimates for binomial
probabilities, which are stated and proved in the appendix.

First we will estimate $\scrN(a, b; \ell)$.
Due to the separation imposed by the layer constraint, when considered as a random variable which depends on
the random percolation configuration, $\scrN(a, b; \ell)$ itself is a sum of independent indicators,
one for each $\fd \in \frL_\ell$, of the event that $\fd$ has type $(a,b)$.
In other words, $\scrN(a,b;\ell)$ is itself a binomial random variable, measurable with respect to the percolation configuration.
A fixed dimer $\fd = (u,v)$ in layer $\ell$ 
is included in $\scrN(a, b; \ell)$ with probability
\begin{equation}
    \label{eq:pdefo}
    \P[\N(u) = a, \N(v) = b] \leq 2^{-d \Rnd{H_p\F{a}{d} + H_p\F{b}{d}}} \eqcolon p_{a, b},
\end{equation}
via the standard Lemma \ref{lem:bin}, which bounds binomial probabilities using the entropy function $H_p$ defined in \eqref{eq:hpdef}.
Thus $\scrN(a, b; \ell)$ is stochastically dominated by a $\Bin(2^d d^2, p_{a,b})$ random
variable. Therefore, for each polynomial $f(d)$ and for each choice of $a, b, \ell$, we have
\[
    \P[\scrN(a, b; \ell) \geq f(d) \cdot 2^d d^2 \cdot p_{a,b}] \leq \f{1}{f(d)}
\]
by Markov's inequality. Since the number of labels $(a, b, \ell)$ is at most $\poly(d)$,
by a union bound, choosing the correct $f(d)$, there is a polynomial $g(d)$ such that with probability $\geq 1 - \fr{1}{d}$,
\begin{equation}
    \label{eq:nboundo}
    \scrN(a, b; \ell) \leq 2^d \cdot g(d) \cdot p_{a,b}, \quad \text{for all $a, b, \ell$}.
\end{equation}

Now if $\scrN(a,b; \ell) \leq \epsilon \cdot 2^{a+b}$ then the probability on the right-hand side
of \eqref{eq:pibnd1} is zero.
Thus, by considering the values of $p_{a,b}$, \eqref{eq:nboundo} will allow us to rule out certain values of $a$ and $b$ a priori,
by showing that those values have $\scrN(a,b;\ell) \leq \eps \cdot 2^{a+b}$ with high probability.
The other values of $a$ and $b$ not ruled out here will be handled afterwards.

We will derive an entropy inequality for the values of $a$ and $b$ \emph{not} ruled out by this idea.
First, by convexity, we have
\begin{equation}
    H_p\Rnd{\fr{a}{d}} + H_p\Rnd{\fr{b}{d}} \geq 2 H_p\Rnd{\fr{a+b}{2d}},
\end{equation}
and so, defining
\begin{equation}
    p_{a+b} \coloneqq 2^{-2d H_p\Fr{a+b}{2d}}.
\end{equation}
we obtain $p_{a,b} \leq p_{a+b}$.
Thus $\scrN(a,b;\ell)$ is stochastically dominated by $\Bin(2^d d^2, p_{a+b})$, and so
\begin{align}
    \P\Box{\scrN(a, b; \ell) > \epsilon \cdot 2^{a + b}} &\leq \P\Box{\Bin(2^d d^2, p_{a + b}) > \epsilon \cdot 2^{a + b}} \\
                                                        &\leq 2^{d (1 - 2 H_p(t) - 2t)} \cdot \poly(d),
\end{align}
for $t = \fr{a+b}{2d}$, via Markov's inequality.
The bound vanishes exponentially fast in $d$ whenever
\begin{equation}
\label{eq:negligible1}
    1 - 2 H_p(t) - 2t < - \delta
\end{equation}
for some fixed absolute constant $\delta > 0$.
In other words, when we consider values of $a,b$ for which \eqref{eq:negligible1} holds with $t = \fr{a+b}{2d}$,
the right-hand side of \eqref{eq:pibnd1} is zero with high probability.

This means we only need to analyze the values of $a$ and $b$ for which \eqref{eq:negligible1} fails, i.e.\ for which
\begin{equation}
    \label{eq:cond1}
    2 H_p(t) + 2t \leq 1 + \delta,
\end{equation}
for some fixed absolute constant $\delta > 0$, where $t = \fr{a+b}{2d}$.
So for the remainder of the proof we assume this condition holds and also that 
$\scrN(a,b;\ell) > \epsilon \cdot 2^{a+b}$, so that we may effectively use
the binomial entropy bounds collected in the appendix
(again, if this condition fails then the probability \eqref{eq:pibnd1} is simply zero).

Let us return to \eqref{eq:pibnd1} with these assumptions.
We wish to apply Lemma \ref{lem:bin}, which states that
\begin{equation}
    \P[\Bin(n,q_1) > n q_2] \leq 2^{- n \cdot H_{q_1}(q_2)}
\end{equation}
when $q_1 \leq q_2$.
We will then apply Lemma \ref{lem:entropy}, which states that
\begin{equation}
    H_{q_1}(q_2) \geq \fr{1}{2} q_2 \cdot \log_2 \Fr{q_2}{q_1}
\end{equation}
whenever $q_2 \geq 10 q_1$.
To use these two lemmas for \eqref{eq:pibnd1}, where $n = \scrN(a,b;\ell)$, $q_1 = 2^{-a}$,
and $n q_2 = \eps \cdot 2^{a+b}$, we need
\begin{equation}
\label{eq:nbound1o}
    10 \cdot \scrN(a,b;\ell) \cdot 2^{-a} \leq \eps \cdot 2^{a+b},
    \qquad \text{i.e.} \qquad
    \scrN(a,b;\ell) \leq \fr{\eps}{10} 2^{2a+b}.
\end{equation}
Under \eqref{eq:nboundo}, which holds with probability $\geq 1 - \fr{1}{d}$, \eqref{eq:nbound1o} holds
as long as
\begin{equation}
    2^d \cdot g(d) \cdot p_{a,b} \leq \fr{\eps}{10} 2^{2a+b}.
\end{equation}
Recalling the definition of $p_{a,b}$ from \eqref{eq:pdefo}, this is equivalent to
\begin{equation}
    2^{-d \left(2s + H_p(s) + r + H_p(r)\right)} \leq 2^{-d} \fr{\eps}{10 g(d)},
\end{equation}
where we have defined $s = \fr{a}{d}$ and $r = \fr{b}{d}$.
Since $\fr{\eps}{10 g(d)}$ is only polynomially small, the above is implied for $d$ large enough whenever
\begin{equation}
\label{eq:delta1}
    2 s + H_p(s) + r + H_p(r) > 1 + \delta_1
\end{equation}
for some universal $\delta_1 > 0$.
To understand the above inequality, we turn to Lemma \ref{lem:minentropy} which states that
for any $m \geq 1$, the function
\begin{equation}
    f_m(p) \coloneqq \inf\{ m s + H_p(s) : s \in [0,1] \}
\end{equation}
is continuous, strictly increasing in $p \in (0,1)$, and satisfies
\begin{equation}
    f_m(\tfrac{1}{2}) = m + 1 - \log_2(2^m+1).
\end{equation}
With $m = 1$ and $m = 2$ respectively, this yields
\begin{align}
    r + H_{\fr{1}{2}}(r) &\geq 2 - \log_2 3 \approx 0.415, \label{eq:mlem1} \\
    2s + H_{\fr{1}{2}}(s) &\geq 3 - \log_2 5 \approx 0.678
\end{align}
for any $r, s \in [0,1]$.
The sum of these is strictly greater than $1$, which means, by the continuity and increasing nature of $f_m$
as a function of $p$, that there are some universal $\delta_1, \gamma_1$ such that \eqref{eq:delta1} holds for all
$r, s \in [0,1]$ and all $p > \fr{1}{2} - \gamma_1$.

\def\ts{t^\star}
\def\ss{s^\star}
All of this is to say that we may in fact apply Lemmas \ref{lem:bin} (which bounds binomial probabilities in terms of entropy $H_p(q)$)
and \ref{lem:entropy} (which gives a helpful estimate on $H_p(q)$) to \eqref{eq:pibnd1},
and we obtain
\begin{equation}
\label{eq:pibnd1result}
    \pi\left[\scrC(a, b; \ell; \sfO_1; S) \geq \e \cdot 2^{a + b}\right]
    = \P\left[\Bin(\scrN(a, b; \ell), 2^{-a}) \geq \e \cdot 2^{a + b}\right]
    \leq 2^{-\text{exponent}}
\end{equation}
where
\begin{align}
    \text{exponent} & \coloneqq \scrN(a, b; \ell) \cdot
    H_{2^{-a}}\F{\e \cdot 2^{a + b}}{\scrN(a, b; \ell)} \tag*{(Lemma \ref{lem:bin})} \\
    %%%%%%%%%%%%%%%%%%%%%%%%%%%%%%%%%%%%%%%%%%%%%%%%%%%%%%%%%%%%%%%%%%%%%%%%%%%%%%%%%%%%%%%%%%%%%%%%%%
        &\geq \scrN(a, b; \ell) \cdot \f12 \cdot \F{\e \cdot 2^{a + b}}{\scrN(a, b; \ell)} \log_2
        \F{\e \cdot 2^{a + b}}{\scrN(a, b; \ell) \cdot 2^{-a}} \tag*{(Lemma \ref{lem:entropy})} \\
        %%%%%%%%%%%%%%%%%%%%%%%%%%%%%%%%%%%%%%%%%%%%%%%%%%%%%%%%%%%%%%%%%%%%%%%%%%%%%%%%%%%%%%%%%%%%%%%%%%
        &\geq \fr{\eps}{2} \log_2(10) \cdot 2^{a + b} \tag*{(Equation \eqref{eq:nbound1o})} \\
        &= \fr{\eps}{2} \log_2(10) \cdot 2^{2dt}. \tag*{(Definition of $t$)}
\end{align}
Note that since \eqref{eq:nbound1o} holds for all $a,b,\ell$ simultaneously with probability at least $1 - \fr{1}{d}$
(since it is a consequence of \eqref{eq:nboundo}),
the above inequality for the exponent also holds simultanously for all $a,b,\ell$ with the same probability.
So we find that
\begin{equation}\label{exprate}
    \pi\left[\scrC(a,b;\ell,\sfO_1;S) > \eps \cdot 2^{a+b}\right] \leq \Exp{- \fr{\eps}{2} \log_2(10) \cdot 2^{2d\ts}}
\end{equation}
simultaneously for all $a,b,\ell$ with probability at least $1 - \fr{1}{d}$, where, recalling condition \eqref{eq:cond1}, we define
\begin{equation}\label{deft}
    \ts \coloneqq \inf \{ t : 2H_p(t) + 2t \leq 1 + \delta \}.
\end{equation}
Eventually, we will take $\alpha_p = 2^{2\ts}$, and we will analyze the above optimization problem
to show that $\alpha_p > \fr{(2-p)^2}{2}$, as stated in the lemma.

As mentioned earlier, the same upper bound holds for the $\sfO_2$ case with the same proof.
At this point we will derive the same upper bound for the $\sfN$ case, after which we will analyze the 
optimization problem just obtained to finish the proof.

\noindent \textbf{Case 2:} $\sigma = \sfN$.

In this case, the dimer $\fd$ does not overlap (share a vertex) with $S$, but there is a singleton $\fs
\in S$ such that $\fs \sim \fd$ (choose arbitrarily if there are multiple such singletons).
The arguments will be very similar to the overlap case, with some extra work to keep track of the properties of $\fs$.

Let us further subdivide $\scrC(a, b; \ell; \sfN; S)$ depending on the type of the singleton $\fs$ (i.e.\ the value of $\N(\fs)$).
In other words, let $\scrC(a, b, c; \ell; \sfN; S)$ be the count of pairs $(\fd, \fs)$, such that $\fd$ is in layer $\ell$, has type
$(a, b)$, status $\sfN$ with respect to $S$, and $\fd \sim \fs \in S$ with type $\N(\fs) = c$. Then,
\[
    \scrC(a, b; \ell; \sfN; S) \leq \sum_{c=0}^d \scrC(a, b, c; \ell; \sfN; S)
\]
({\it this is not an equality only because some dimers will be counted in multiple summands
if they have multiple neighbors $v$ in $S$}).
If this is at least $\e \cdot 2^{a + b}$, we must then have some choice of $c$ such that
\[
    \scrC(a, b, c; \ell; \sfN; S) \geq \f{\e}{d+1} \cdot 2^{a + b} \eqcolon \e' \cdot 2^{a + b},
\]
where $\eps' = \eps/\poly(d)$.

Let $\scrN(a, b, c; \ell)$ be the count of pairs $(\fd, \fs)$ of the kind counted in $\scrC(a, b, c; \ell; \sfN; S)$
but excluding the status-$\sfN$ condition and the
$\fs \in S$ condition (note that this count only depends on the percolation configuration). 

Now for every pair $(\fd, \fs)$ counted in $\scrN(a, b, c; \ell)$,
it is counted in $\scrC(a, b, c; \ell; \sfN; S)$ if and only if $\fs \in S$ and $\fd \cap S = \emptyset$, which happens with
probability at most $2^{-c}$ (under $\pi$).
Further, no two pairs $(\fd, \fs), (\fd', \fs')$ counted in $\scrN(a, b, c; \ell)$ satisfy $\fs = \fs'$ (unless $\fd = \fd'$), 
since otherwise this would imply $\fd \sim \fs = \fs' \sim \fd'$ and thus $\fd$ and $\fd'$ are within a distance of 2 in the
shared-neighbor graph, which is ruled out by construction of the layers $\frL_\ell$.
Consequently, under $\pi$, $\scrC(a, b, c; \ell; \sfN; S)$ is stochastically dominated by a
$\Bin(\scrN(a, b, c; \ell), 2^{-c})$ random variable and thus
\begin{equation}
    \label{eq:pibnd3}
    \pi\Box{\scrC(a, b, c; \ell; \sfN; S) > \e' \cdot 2^{a + b}}
    \leq \P\Box{\Bin(\scrN(a, b, c; \ell), 2^{-c}) > \e' \cdot 2^{a + b}}.
\end{equation}

As in the overlap case, we will now estimate $\scrN(a, b, c; \ell)$.
By definition we have
\begin{align}
    \scrN(a, b, c; \ell) &= \sum_{\substack{\fd \in \frL_\ell, \\ \fs \sim \fd}} \ind{\N(\fs) = c} \cdot
    \ind{\fd \text{ has type } (a,b)} \\
                        &\leq 2 d^2 \sum_{\fd \in \frL_\ell} \1_{E_\fd},
\end{align}
since there are at most $2 d^2$ neighbors $\fs$ of each $\fd$, where we set
\begin{equation}
    E_\fd \coloneq \{\exists \fs \sim \fd : \N(\fs) = c\} \cap \{\fd \text{ has type } (a,b)\}.
\end{equation}
By a union bound over all possible neighbors $v$ of $\fd = (u,w)$, we have
\begin{align}
    \P[E_\fd] &\leq 2 d^2 \cdot \P[\N(v) = c, \N(u) = a, \N(w) = b] \\
    &\leq 2 d^2 \cdot \P[\N(v) + \N(u) + \N(w) = a + b + c] \\
    &\leq 2 d^2 \cdot 2^{-3 d H_p\Rnd{\fr{a+b+c}{3d}}} \eqcolon p_{a+b+c},
\end{align}
using the fact that $\N(v) + \N(u) + \N(w)$ is a $\Bin(3d,p)$ random variable and applying Lemma \ref{lem:bin}.
Note also that the indicators $\1_{E_\fd}$ for $\fd \in \frL_\ell$ are independent (with respect to the percolation measure) by the separation within a layer.
Thus $\scrN(a,b,c;\ell)$ is stochastically dominated by $2 d^2 \cdot \Bin(2^d d^2, p_{a+b+c})$.
By the same argument as was used for \eqref{eq:nboundo}, there is some polynomial $g(d)$ such that
with probability $\geq 1 - \fr{1}{d}$, we have
\begin{equation}
    \label{eq:nbound}
    \scrN(a, b, c; \ell) \leq 2^d \cdot g(d) \cdot p_{a+ b+ c}, \quad \text{for all $a, b, c, \ell$}.
\end{equation}
Once again, we know that the right-hand side of \eqref{eq:pibnd3} is zero if
\begin{equation}
\label{eq:irrelevant2}
    \scrN(a, b, c; \ell) \leq \e' \cdot 2^{a + b},
\end{equation}
and we can use \eqref{eq:nbound} to determine which values of $a, b, c$ satisfy \eqref{eq:irrelevant2} with high probability.
This time, since
\begin{equation}
    \scrN(a,b,c;\ell) \leq 2 d^2 \scrN(a,b;\ell),
\end{equation}
the condition \eqref{eq:irrelevant2} is implied by 
\begin{equation}
    \scrN(a,b,\ell) \leq \e'' \cdot 2^{a+b}
\end{equation}
for some $\e'' \geq \fr{\e'}{\poly(d)}$.
So by the same argument as in the previous case, with high probability we have $\scrN(a,b,c;\ell) \leq \e' \cdot 2^{a+b}$ 
simultaneously for all $a,b,c,\ell$ which do not satisfy
\begin{equation}
\label{eq:cond2}
    2 H_p(t) + 2 t \leq 1 + \delta
\end{equation}
for some fixed $\delta > 0$, where $t = \fr{a+b}{2d}$ as before.
For the rest of the proof we assume this condition holds and that $\scrN(a,b,c;\ell) > \e' \cdot 2^{a+b}$, otherwise
the upper bound is just zero.

With these assumptions we return to \eqref{eq:pibnd3}.
Recall that in the previous case, we bounded the analogous probability \eqref{eq:pibnd1} by applying Lemmas \ref{lem:bin} and \ref{lem:entropy}
(which bound binomial probabilities in terms of entropy, and give a convenient estimate on the entropy, respectively) to obtain \eqref{eq:pibnd1result}.
We will apply the same strategy here, and, analogous to the conditions \eqref{eq:nbound1o} required in the previous case, we will require
the following conditions to apply Lemmas \ref{lem:bin} and \ref{lem:entropy}.
\begin{equation}
\label{eq:nbound3n}
    10 \cdot \scrN(a,b,c;\ell) \cdot 2^{-c} \leq \eps' \cdot 2^{a+b},
    \qquad \text{i.e.} \qquad
    \scrN(a,b,c;\ell) \leq \fr{\eps'}{10} 2^{a+b+c}.
\end{equation}
As before, if \eqref{eq:nbound} holds (which happens with proabability at least $1 - \fr{1}{d}$),
then \eqref{eq:nbound3n} holds as long as 
\begin{equation}
    2^d \cdot g(d) \cdot p_{a+b+c} \leq \fr{\eps'}{10} 2^{a+b+c}.
\end{equation}
This is equivalent to
\begin{equation}
    2^{3d(s+H_p(s))} \geq 2^d \fr{10 g(d)}{\eps'},
\end{equation}
where we have defined $s = \fr{a+b+c}{3d}$.
Since $\fr{10 g(d)}{\eps'}$ is polynomial, this holds for high enough $d$ whenever
\begin{equation}
\label{eq:delta2}
    s + H_p(s) > \fr{1}{3} + \delta_2
\end{equation}
for some universal $\delta_2 > 0$.
As in the previous case, specifically in \eqref{eq:mlem1}, we now apply Lemma \ref{lem:minentropy} with $m=1$ and we find that
\begin{equation}
    s + H_{\fr{1}{2}}(s) \geq 2 - \log_2 3 \approx 0.415 > \fr{1}{3}
\end{equation}
for all $s \in [0,1]$.
Again by the continuity and increasing nature of the optimization problem, afforded by Lemma \ref{lem:minentropy},
there are some universal $\delta_2, \gamma_2$ such that \eqref{eq:delta2} holds for all $s \in [0,1]$ 
and all $p > \fr{1}{2}-\gamma_2$.
In other words, we may indeed apply Lemmas \ref{lem:bin} and \ref{lem:entropy} to \eqref{eq:pibnd3}, obtaining
\begin{equation}
    \pi \Box{\scrC(a,b,c;\ell;\sfN;S) \geq \eps' 2^{a+b}}
    \leq \P \Box{\Bin(\scrN(a,b,c;\ell),2^{-c}) \geq \eps' \cdot 2^{a+b}}
    \leq 2^{-\text{exponent}},
\end{equation}
where
\begin{align}
    \text{exponent} &\coloneqq \scrN(a,b;\ell) \cdot H_{2^{-c}} \Fr{\eps' \cdot 2^{a+b}}{\scrN(a,b,c;\ell)} \tag*{(Lemma \ref{lem:bin})} \\
    &\geq \scrN(a,b,c;\ell) \cdot \fr{1}{2} \cdot \Fr{\eps' \cdot 2^{a+b}}{\scrN(a,b,c;\ell)} \log_2 \Fr{\eps' \cdot 2^{a+b}}{\scrN(a,b,c;\ell) \cdot 2^{-c}} \tag*{(Lemma \ref{lem:entropy})} \\
    &\geq \fr{\eps'}{2} \log_2(10) \cdot 2^{a+b} \tag*{(equation \eqref{eq:nbound3n})} \\
    &= \fr{\eps'}{2} \log_2(10) \cdot 2^{2dt}. \tag*{(Definition of $t$)}
\end{align}
Thus we obtain the same upper bound as in \eqref{exprate} with an exponential rate given
in terms of the same optimization problem \eqref{deft} over $t$,
with the same condition \eqref{eq:cond2} on $t$ as in the $\sfO_1$ case (recall \eqref{eq:cond1}).
Note that the polynomial prefactor inside the exponential in this case is potentially different due to the union bound
over $c$, but it is still uniform over $a,b,\ell$.

\noindent \textbf{Final analysis.}
In all cases $\sigma = \{ \sfO_1, \sfO_2, \sfN \}$ above, we arrive at the same upper bound \eqref{exprate} with the same exponential rate given in terms of the same optimization problem \eqref{deft}.
Thus, by a union bound, there is some function $f(d) > \poly(d)^{-1}$ depending on the choice of $\eps > \poly(d)^{-1}$ such that, with probability tending to $1$
we have the following inequality simultaneously for all $a,b,\ell,\sigma$:
\begin{equation}
    \pi\Box{\scrC(a,b;\ell;\sigma;S) > \epsilon \cdot 2^{a+b}} \leq
    \Exp{- f(d) \cdot 2^{2 d \ts}},
\end{equation}
where $\ts$ solves the optimization problem \eqref{deft}, restated as follows:
\begin{equation}
    \ts = \inf \{ t : 2 H_p(t) + 2 t \leq 1 + \delta \}
\end{equation}
for some small constant $\delta > 0$ which we may choose.
As mentioned before, we will take $\alpha_p = 2^{2\ts}$.
So, to finish the proof of the lemma, we must simply show that $\alpha_p > \fr{(2-p)^2}{2}$.

For this, a calculation relegated to Lemma \ref{lem:dimersumoptimization} shows that
\begin{equation}
    \fr{(2-p)^2}{2} = 2^{1 - 2 \ss - 2 H_p(\ss)},
\end{equation}
where $\ss$ maximizes $1 - 2 s - 2 H_p(s)$ over $s \in [0,1]$.
So to prove that $\alpha_p > \fr{(2-p)^2}{2}$, it suffices to show that
\begin{equation}
    \sup\{ 1 - 2s - 2H_p(s) : s \in [0,1]\} < 2 \cdot \inf\{ t : 2 H_p(t) + 2t \leq 1 + \delta \},
\end{equation}
or in other words that
\begin{equation}
\label{eq:stineq}
    1 - 2s - 2H_p(s) < 2 t
\end{equation}
for all $s \in [0,1]$ and all $t$ satisfying $2 H_p(t) + 2t \leq 1 + \delta$.
By this condition on $t$, \eqref{eq:stineq} would hold if we had
\begin{equation}
\label{eq:bigentropy}
    1 - 2s - 2H_p(s) + (1+\delta) < 2t + (2t + 2H_p(t))
\end{equation}
for all $t, s \in [0,1]$.
Indeed, \eqref{eq:stineq} is obtained by adding the inequality
\begin{equation}
    - (1+\delta) \leq - (2H_p(t) + 2t)
\end{equation}
to \eqref{eq:bigentropy}.
Rearranging \eqref{eq:bigentropy} yields
\begin{equation}
    2 (s + H_p(s)) + 2 (2t + H_p(t)) > 2 + \delta,
\end{equation}
and so we may apply Lemma \ref{lem:minentropy} again with $m = 1$ and $m = 2$.
This yields
\begin{align}
    2( 2 t + H_{\fr{1}{2}}(t) ) &\geq 2 (3 - \log_2(5)) > 1.35 \\
    2 (s + H_{\fr{1}{2}}(s) ) &\geq 2 (2 - \log_2(3)) > 0.83,
\end{align}
for all $s, t \in [0,1]$, and $1.35 + 0.38 = 2.18$ is strictly larger than $2$.
By the continuity and monotonicity from Lemma \ref{lem:minentropy},
there is some $\gamma_3 > 0$ such that for $p > \fr{1}{2} - \gamma_3$, the sum of the relevant
quantities is still at least $2 + \delta$, for our fixed universal choice of $\delta$, as long as it is small enough.
Thus, taking $\gamma = \min\{\gamma_1, \gamma_2, \gamma_3\}$, the proof is concluded.
\end{proof}

%% file: sections/clt.tex
%!TEX root =../main.tex
\section{Distributional limits}
\label{sec:clt}

In this section we prove Proposition \ref{prop:clt_iop}, which states that
\begin{equation}
    \left(
        \fr{\Psi^\cE - \mu}{\sigma}, \fr{\Psi^\cO - \mu}{\sigma}
    \right) \dto \Nor{0}{1} \otimes \Nor{0}{1}
\end{equation}
for some explicit deterministic quantities $\mu = \mu_{d,p}$ and $\sigma = \sigma_{d,p}$
(initially defined in \eqref{eq:mudefintro} and \eqref{eq:sigmadefintro}),
and for all $p$ larger than some specific threshold $0.455$.
It should be noted that the threshold $0.455$ is not a threshold for any qualitative change, but rather the point above which 
the \emph{specific approximate mean} $\mu_{d,p}$ we use is a good enough approximation.
This threshold is unrelated to the qualitatively meaningful threshold $\f12 - \gamma$ appearing in Corollary \ref{cor:approxinprob_iop}
(which is a consequence of Proposition \ref{prop:dimers}), aside from the fact that we have chosen the specific approximate mean in \eqref{eq:mudefintro}
to be valid for all possible values of $p > \f12 - \gamma \approx 0.465$.

Recall that in the previous few sections, we have finished the proof of Corollary \ref{cor:approxinprob_iop},
deriving an in-probability approximation for the number of independent sets in $\Qp$, namely that
\begin{equation}
    \Cnt \psim 2^{2^{d-1}} \cdot \left( e^{\Psi^\cE} + e^{\Psi^\cO} \right),
\end{equation}
where the random variables $\Psi^\cE$ and $\Psi^\cO$ can be written explicitly in terms of the percolation configuration.
We remind the reader that, as initially defined in \eqref{eq:psidefintro}, we have
\begin{equation}
    \Psi^\cH = \log \fakepolymerpf^\cH + \Delta^\cH - \tilde{\Delta}^\cH,
\end{equation}
for $\cH \in \{\cE,\cO\}$, where, as defined in \eqref{eq:fakepolymerpfdef} and \eqref{eq:deltadef},
\begin{equation}
    \fakepolymerpf^\cH = \prod_{v \in \cH} (1 + \phi_v), \qquad
    \Delta^\cH = \sum_{\fd \in \Dim^\cH} \phi_\fd, \qquad \text{and} \qquad
    \tilde{\Delta}^\cH = \sum_{\fd \in \Dim^\cH} \fphi_\fd.
\end{equation}

It is also worth remarking that the requirement that $p$ is bigger than some threshold in the statement is not necessary for the
central limit theorem itself to hold, when $\Psi^\cE$ and $\Psi^\cO$ are re-centered by their exact mean
and rescaled by their exact standard deviation.
However, the exact mean and standard deviation do not have explicit closed-form expressions in general.
So we introduce the simpler closed-form quantities $\mu$ and $\sigma$ which work as proxies for the mean and standard
deviation, but only for $p$ large enough.
The threshold for the validity of these approximations is below the threshold $\fr{1}{2} - \gamma$ in
Corollary \ref{cor:approxinprob_iop}/Proposition \ref{prop:dimers}, so these approximations do not restrict the range of $p$ for
which our main theorem holds.

The first step towards proving the CLT will show that the quantities $\log \fakepolymerpf^\cE$ and $\log \fakepolymerpf^\cO$
satisfy a joint central limit theorem. We then show that the quantities $\Delta^\cH$ and $\tilde{\Delta}^\cH$
both concentrate around deterministic values for each $\cH \in \{\cE,\cO\}$.
In other words, the fluctuations of $(\Psi^\cE, \Psi^\cO)$ are driven purely by those of
$(\log \fakepolymerpf^\cE, \log \fakepolymerpf^\cO)$.

Before moving on, we remark that the proof of the central limit theorem for $(\log \fakepolymerpf^\cE, \log \fakepolymerpf^\cO)$
uses the same techniques as our previous proof of a similar central limit theorem for a simplification of these variables in \cite[Lemma 3.1]{brcgw}.
Only a minor modification of that proof is required for the present setting, and the concentration for $\Delta^\cH$ and $\tilde{\Delta}^\cH$
only requires a few novel but relatively routine moment calculations.
Nevertheless, we include complete proofs for the reader's convenience.

\subsection{Joint central limit theorem for sums of logarithmic expressions}
For notational convenience, as in \eqref{eq:philogdef}, we set
\begin{equation}
    \Phi^\cH_{\log} = \log \fakepolymerpf^\cH = \sum_{\fs \in \cH} \log(1 + \phi_\fs),
\end{equation}
for $\cH \in \{ \cE, \cO \}$.
When appropriate, we will drop the superscript $\cH$; for instance, the mean and variance of these
sums are identical for both sides, and we will denote these simply by $\E[\Phi_{\log}]$ and
$\Var(\Phi_{\log})$.
Note that, while both $\Phi_{\log}^\cE$ and $\Phi_{\log}^\cO$ are sums of independent random variables, the two sums themselves
are not independent, because deleting one edge affects the neighborhood of one odd vertex and one even vertex.
However, their dependence is sufficiently weak that we will be able to prove the following joint CLT to a pair of
i.i.d.\ standard normal random variables.

\begin{lemma}
\label{lem:clt}
For $p \in (0,1)$,
\begin{equation}
    \left( \fr{\Phi^\cE_{\log} - \E[\Phi_{\log}]}{\sqrt{\Var(\Phi_{\log})}},
    \fr{\Phi^\cO_{\log} - \E[\Phi_{\log}]}{\sqrt{\Var(\Phi_{\log})}} \right)
    \dto \Nor{0}{1} \otimes \Nor{0}{1}.
\end{equation}
\end{lemma}

We will invoke Stein's method to prove CLT for random
variables with a sparse dependency graph (a graph with the random variables as the vertex set and an edge set such that
any vertex is independent of the set of all variables outside its neighborhood).
The precise formulation we will rely on is 
\cite[Theorem 3.6]{ross} which was adapted from the main result of \cite{stein}.

\begin{theorem}{\cite[Theorem 3.6]{ross}}
\label{thm:dependentstein}
Let $X_1, \dotsc, X_n$ be random variables with $\E[X_i^4] < \infty$,
$\E[X_i] = 0$, $\sigma^2 = \Var(\sum_{i=1}^n X_i)$, and define $\widetilde{W} = \sum_{i=1}^n X_i / \sigma$. 
Let the collection $(X_i, \dotsc, X_i)$ have dependency neighborhoods $N_i$,
$i=1,\dotsc,n$, with $D \coloneqq \max_{1 \leq i \leq n} |N_i|$.
Then
\begin{align}
    \label{eq:wassersteinbound}
    \mathrm{d_{Was}}(\tilde{W}, W) \lesssim
     \fr{D^2}{\sigma^3} \sum_{i=1}^n \E|X_i|^3
        + \fr{D^{3/2}}{\sigma^2} 
        \sqrt{\sum_{i=1}^n \E[X_i^4]},
\end{align}
where $W$ is a standard normal random variable and $\mathrm{d_{Was}}$ is the $1$-Wasserstein distance.
\end{theorem}

In the above statement, the \emph{dependency neighborhood} of a variable $X_i$ in the collection $\{X_1, \dotsc, X_n\}$
is the set of indices $j$ for which $X_i$ and $X_j$ are not independent.

Presently we will not be concerned with the precise rate of convergence under the $1$-Wasserstein distance, and so
we forgo a precise definition of this distance.
Suffice it to say, $\tilde{W}$ converges in distribution to $W$ if $\mathrm{d_{Was}}(\tilde{W},W) \to 0$, so
we will be satisfied with showing that the right-hand side of \eqref{eq:wassersteinbound} converges to $0$.

\begin{proof}[Proof of Lemma \ref{lem:clt}]
Fix $c = (c_\cE, c_\cO)$ in the unit circle in $\R^2$, and set
\begin{equation}
    \Phi^c = c_\cE \left(\Phi^\cE_{\log} - \E[\Phi_{\log}]\right)
        + c_\cO \left(\Phi^\cO_{\log} - \E[\Phi_{\log}]\right).
\end{equation}
We can write $\Phi^\c$ as a sum of the following $2^d$ variables indexed by vertices $v$ of $\Q$:
\begin{equation}
    X_v \coloneqq c_\cH \left( \log (1 + \phi_v) - \E[\log (1 + \phi_v)] \right),
\end{equation}
where we take $\cH = \cE$ if $v \in \cE$ and $\cH = \cO$ if $v \in \cO$.

We will apply Theorem \ref{thm:dependentstein} to conclude that
\begin{equation}
    \fr{\Phi^c}{\sqrt{\Var(\Phi^c)}} \dto \Nor{0}{1}.
\end{equation}
Since this holds for an arbitrary $c$ in the unit circle,
the Cramer-Wold theorem then implies that
\begin{align}
\label{eq:cltintermediate}
    \left(
    \fr{\Phi_{\log}^\cE - \E[\Phi_{\log}]}{\sqrt{\Var(\Phi^c)}},
    \fr{\Phi_{\log}^\cO - \E[\Phi_{\log}]}{\sqrt{\Var(\Phi^c)}}
    \right)
    \dto \Nor{0}{1} \otimes \Nor{0}{1}.
\end{align}
Now, by Lemma \ref{lem:cov} we have
\begin{align}
    \Var(\Phi^c) &= c_\cE^2 \Var(\Phi_{\log}^\cE) + c_\cO^2 \Var(\Phi_{\log}^\cO)
        + c_\cE c_\cO \Cov(\Phi_{\log}^\cE, \Phi_{\log}^\cO) \\
        &= \Var(\Phi_{\log}) (1 + o(1)),
\end{align}
and so \eqref{eq:cltintermediate} implies the conclusion of the lemma.

So it just remains to verify the conditions of Theorem \ref{thm:dependentstein} and show that
the right-hand side of \eqref{eq:wassersteinbound} converges to $0$.
Of course by definition we have $\E[X_v] = 0$.
To bound the moments of the $X_v$, notice that
\begin{align}
    \E[X_v^k] &\leq \E\left[(\log(1 + \phi_v) - \E[\log(1 + \phi_v)])^k\right] \\
    &\leq \sum_{j=0}^k \binom{k}{j} \E\left[(\log(1 + \phi_v))^j\right] \E[\log(1 + \phi_v)]^{k-j} \\
    &\leq \sum_{j=0}^k \binom{k}{j} \E[\phi_v^j] \E[\phi_v]^{k-j} \\
    &\leq \sum_{j=0}^k \binom{k}{j} \E[\phi_v^k]^{j/k} \E[\phi_v^k]^{(k-j)/k} \\
    &= 2^k \cdot \E[\phi_v^k],
\end{align}
where we used the fact that $\log(1+\phi_v) \leq \phi_v$ (recall that $\phi_v > 0$) and then applied Jensen's inequality.
So we have
\begin{equation}
	\E[X_v^3] \lesssim \E[\phi_v^3] = \Fr{8-7p}{8}^d ,
\end{equation}
and
\begin{equation}
	\E[X_v^4] \lesssim \E[\phi_v^4] = \Fr{16-15p}{16}^d,
\end{equation}
where these moments are computed in the appendix, in Lemma \ref{lem:moments} (a).

Now, for the size of the dependency neighborhoods, each $X_v$ is independent of all but $d$ other variables,
which are exactly the ones corresponding to neighbors of $v$.
So by \eqref{eq:wassersteinbound}, the $1$-Wasserstein distance between $\Var(\Phi^c)^{-1/2} \Phi^c$ and a
standard normal random variable is, up to constants, at most
\begin{equation}
    \fr{d^2 2^d \Fr{8-7p}{8}^d}
        {\Fr{4-3p}{2}^{3d/2}} 
        + \fr{d^{3/2}\sqrt{2^d \Fr{16-15p}{16}^d}}
        {\Fr{4-3p}{2}^d}
    =
        d^2 \Fr{(8 - 7p)^2}{2(4-3p)^3}^{d/2}
        + d^{3/2} \Fr{16-15p}{2 (4-3p)^2}^{d/2}.
\end{equation}
The proof will thus be complete if we can show that both
\begin{equation}
    \fr{(8-7p)^2}{2(4-3p)^3} < 1
    \qquad \text{and} \qquad
    \fr{16-15p}{2(4-3p)^2} < 1.
\end{equation}
First notice that the second bound implies the first, since
\begin{equation}
    (8-7p)^2 \leq (16-15p)(4-3p).
\end{equation}
To verify the above inequality, note that the difference
\begin{equation}
    (8-7p)^2 - (16-15p)(4-3p)
\end{equation}
is a quadratic with a positive coefficient of $p^2$ (namely $49 - 45$) which vanishes
at $p=0$ and $p=1$, meaning that it must be negative for $0 < p < 1$.
Now, to actually prove the second inequality, note that
\begin{equation}
    2(4-3p)^2 - (16-15p) = 18p^2 - 33p + 16
\end{equation}
is another quadratic with a positive coefficient of $p^2$, and with discriminant
\begin{equation}
    33^2 - 4 \cdot 16 \cdot 18 = - 63,
\end{equation}
so it is positive for all $p$, finishing the proof.
\end{proof}

\subsection{Approximations of the expectation and variance}
\label{sec:clt_approx}

Lemma \ref{lem:clt} exhibits a joint central limit theorem for $\Phi_{\log}^\cE$ and $\Phi_{\log}^\cO$ upon
centering by $\E[\Phi_{\log}]$ and rescaling by $\sqrt{\Var(\Phi_{\log})}$.
However, there is no simple closed form for these quantities in general, and this is particularly an issue
when $p \leq \fr{1}{2}$, as will be discussed in Remark \ref{rmk:onehalf} below.
In this section, we construct quantities $\mu = \mu_{d,p}$ and $\sigma = \sigma_{d,p}$ with explicit closed-form
expressions that can be used as a replacement for the mean and standard deviation of $\Phi_{\log}$ in the central limit theorem.
Specifically, these quantities will satisfy
\begin{equation}
    \sigma^2 = \Var(\Phi_{\log}) (1 + o(1))
    \qquad \text{and} \qquad
    \mu = \E[\Phi_{\log}] + o(\sigma),
\end{equation}
which suffice for our purposes.

For $\sigma$, a simple moment calculation which has been relegated to the appendix as Lemma \ref{lem:logmoments} (a)
shows that
\begin{equation}
\label{eq:varapprox}
    \Var(\Phi_{\log}) = \fr{1}{2} \Rnd{2-\fr{3}{2}p}^d (1 + o(1))
\end{equation}
for all $p \in (0,1)$.
So, as previously defined in \eqref{eq:sigmadefintro}, we set
\begin{equation}
    \sigma^2 = \fr{1}{2} \Rnd{2 - \fr{3}{2}p}^d.
\end{equation}

Next we turn to $\mu$.
An initial strategy for approximating $\E[\Phi_{\log}]$ might be to expand $\log(1 + \phi_v)$ in a
Taylor series and then interchange the expectation and summation; since the moments of $\phi_v$ have explicit
expressions given in Lemma \ref{lem:moments}, this would in principle lead to a series expansion for $\E[\Phi_{\log}]$
which can be truncated to obtain a closed-form expression.
This does indeed work for $p > \fr{1}{2}$, and leads to an approximation of $\E[\Phi_{\log}]$ with additive $o(1)$ error.
However, for $p \leq \fr{1}{2}$, every term in the series grows exponentially with $d$, so there is no such approximation of
$\E[\Phi_{\log}]$ with only $o(1)$ additive error.

\begin{remark}
\label{rmk:onehalf}
The ability or inability to accurately approximate $\E[\Phi_{\log}]$ by a closed-form expression, or equivalently to approximate $\Phi_{\log}$ itself
by a polynomial in the variables $\phi_v$, appears to be the extent of the changes in this model at $p=\f12$.
So, unlike the many other features of $\Qp$ which exhibit a qualitative phase transition at $p=\f12$, the behavior of independent
sets does not change sharply at this value.
\end{remark}

In any case, for such values of $p \leq \f12$, the standard deviation proxy $\sigma$ itself also grows exponentially with $d$, and
only finitely many of the terms in the series $\Phi_{\log}$ grow faster, allowing us to obtain an approximation with $o(\sigma)$ additive
error, meaning this error is dominated in the CLT by the normalization by $\sigma$.
Specifically, for the range of $p$ that we consider, we only need to include three terms in the series.

Namely, for $k \geq 0$ we define
\begin{equation}
\label{eq:mu1kdef}
    \mu_1^k \coloneqq 2^{d-1} \E[\phi_v^k] = \fr{1}{2} \Rnd{2 - \fr{2^k-1}{2^{k-1}}p}^d,
\end{equation}
where the moment calculation here follows from Lemma \ref{lem:moments} (b).
Then, Lemma \ref{lem:logmoments} (b) states that
\begin{equation}
    \E[\Phi_{\log}] = \mu_1^1 - \fr{\mu_1^2}{2} + \fr{\mu_1^3}{3} + o(\sigma)
\end{equation}
for all $p > 0.455$.
It's worth noting that this range of $p$ is simply where $\mu_1^4 \ll \sigma$,
and is not related to the threshold $\fr{1}{2} - \gamma$ of our main theorem.
We also remark again that numerical calculations show that the threshold $\fr{1}{2} - \gamma$
can be taken as low as approximately $0.465$, although we do not prove this rigorously.
So in principle this three-term approximation holds for the full possible range of our theorem,
and regardless it holds for values of $p$ strictly below $\fr{1}{2}$.
Thus we set
\begin{equation}
    \mu_1 = \mu_1^1 - \fr{\mu_1^2}{2} + \fr{\mu_1^3}{3}
\end{equation}
as in \eqref{eq:mu1defintro},
and the above discussion leads to the following corollary of Lemma \ref{lem:clt}.

\begin{corollary}
\label{cor:clt_withapprox}
For $p > 0.455$,
\begin{equation}
    \left( \fr{\Phi^\cE_{\log} - \mu_1}{\sigma},
    \fr{\Phi^\cO_{\log} - \mu_1}{\sigma} \right)
    \dto \Nor{0}{1} \otimes \Nor{0}{1}.
\end{equation}
\end{corollary}

\subsection{Concentration of the dimer sums}

We now show that the quantities $\Delta^\cH$ and $\tilde{\Delta}^\cH$ concentrate at $o(1)$ scale around deterministic values.
As the joint behavior of these quantities across both sides is not relevant, we drop the superscript $\cH$ in
this section for readability.

In Lemma \ref{lem:dimermoments} (a), a short calculation leads to the following exact expressions for the means
of $\Delta$ and $\tilde{\Delta}$:
\begin{equation}
    \E[\Delta] = \mu_2, \qquad
    \E[\tilde{\Delta}] = \tilde{\mu}_2,
\end{equation}
where
\begin{align}
    \mu_2 &= \fr{d(d-1)}{4} \Fr{(2-p)^2}{2}^d \cdot \Fr{1 + (1-p)^2}{2 - p}^2, \\
    \tilde{\mu}_2 &= \fr{d(d-1)}{4} \Fr{(2-p)^2}{2}^d,
\end{align}
as previously defined in \eqref{eq:mu2defintro} and \eqref{eq:mu2tildedefintro}.
Additionally, Lemma \ref{lem:dimermoments} (b) calculates the second moments of $\phi_\fd$ and $\fphi_\fd$
and shows that for $p > 0.391$, we have $\Var(\Delta), \Var(\tilde{\Delta}) \ll 1$.
Again, this number arises simply as the range of $p$ for which $2^d \E[\phi_\fd^2] \ll 1$
and has nothing to do with the threshold $\fr{1}{2} - \gamma$ of our main theorem.
This variance calculation yields the following lemma, giving the $o(1)$ concentration of $\Delta$ and
$\tilde{\Delta}$ around their means.

\begin{lemma}
\label{lem:dimerconcentration}
For $p > 0.391$ we have
\begin{equation}
    \Delta - \mu_2 \pto 0
    \qquad \text{and} \qquad
    \tilde{\Delta} - \tilde{\mu}_2 \pto 0.
\end{equation}
\end{lemma}

Now recall from \eqref{eq:psidefiop} that
\begin{equation}
    \Psi^\cH = \Phi_{\log}^\cH + \Delta^\cH - \tilde{\Delta}^\cH
\end{equation}
for $\cH \in \{\cE, \cO\}$,
and set
\begin{equation}
    \mu = \mu_1 + \mu_2 - \tilde{\mu}_2,
\end{equation}
as originally defined in \eqref{eq:mudefintro}.
Then Corollary \ref{cor:clt_withapprox} and Lemma \ref{lem:dimerconcentration} yield
the following, which was previously stated as Proposition \ref{prop:clt_iop}.

\begin{proposition}
\label{prop:clt_withdimers}
For $p > 0.455$, we have
\begin{equation}
    \left(
        \fr{\Psi^\cE - \mu}{\sigma},
        \fr{\Psi^\cO - \mu}{\sigma}
    \right)
    \dto \Nor{0}{1} \otimes \Nor{0}{1}.
\end{equation}
\end{proposition}

We finally remark that, as can be checked numerically, the quantities $\mu_2$ and $\tilde{\mu}_2$ are indeed
of larger order than $\sigma$ when $p$ is less than $0.508$.
This means that, to ensure the validity of our main theorem which holds for values of $p$ below $\fr{1}{2}$,
these terms must be included in the formula for $\mu$, and cannot be absorbed into an $o(\sigma)$ error term
like the tail of the infinite series discussed in Section \ref{sec:clt_approx}.

%% file: sections/sampling.tex
%!TEX root =../main.tex
\section{Sampling}
\label{sec:sampling}

\def\ASA{\AS^{2,3,4}}
\def\AB{\A_{2,3}}
\def\STw{\mathsf{StepTwo}}
\def\STh{\mathsf{StepThree}}
\def\SF{\mathsf{StepFour}}
\def\Crit{\mathsf{Crit}}

In this section we prove Theorem \ref{thm:sampler}, exhibiting an approximate sampling algorithm which, given a quenched realization
of $\Qp$, yields an independent set with distribution very close to that of a uniformly random independent set in $\Qp$.

We remind the reader of the structure of this algorithm (initially presented in Definition \ref{def:approxsampler_intro}), now using notation developed in the body
of the work.
In order to present the algorithm, let us first introduce the probabilities $q^\cE$ and $q^\cO$ defined by
\begin{equation}
\label{eq:qhdef}
	q^\cH \propto e^{\Psi^\cH} = \fakepolymerpf^\cH \cdot \Exp{\Delta^\cH - \tilde{\Delta}^\cH}, \qquad q^\cE + q^\cO = 1,
\end{equation}
recalling the definition of $\fakepolymerpf^\cH$ from \eqref{eq:fakepolymerpfdef} and the definitions of $\Delta^\cH$ and $\tilde{\Delta}^\cH$
from \eqref{eq:deltadef}.
As we will see shortly in Lemma \ref{lem:proper} below, $q^\cE$ and $q^\cO$ approximate the probabilities that a uniformly random independent
set $I$ in $\Qp$ has the even side or the odd side as its defect side.

Additionally we remind the reader that the the algorithm may fail (in particular at steps \ref{step:fail1_s1construct} and
\ref{step:fail2_shatconstruct} below), and we enforce our usual convention of not distinguishing between a set of vertices and the
corresponding set of polymers, as mentioned in Section \ref{sec:iop_polymer}.

\begin{definition}[$\AS$]
\label{def:approxsampler}
	Define the random independent set $\hat{I}$ in $\Qp$ via the following procedure.
	\begin{enumerate}
		\item Let $\cH$ be $\cE$ or $\cO$, with probabilities $q^\cE$ and $q^\cO$ respectively.
		\label{step:hpick}
		\item Let $\tilde{S}$ denote a subset of $\cH$ constructed by including each $v \in \cH$ independently with probability
		$\fr{\phi_v}{1 + \phi_v}$. In other words, sample $\tilde{S} \sim \pi^\cH$, where $\pi^\cH$ was defined in \eqref{eq:pidef}
		(with the superscript $\cH$ being implicit in that definition).
		\label{step:stildeconstruct}
		\item If $\tilde{S} \notin \Sep_{\leq 2}$ (recall this notation from just above \eqref{eq:smallpolymerpfdef}), terminate the algorithm with a result of failure.
		Otherwise, remove all dimers from $\tilde{S}$ and let $S_1$ denote the resulting set of well-separated singletons.
		\label{step:fail1_s1construct}
		\item Let $S_2$ denote a collection of dimers constructed by including each $\{u,v\} \in \Dim^\cH$ independently with probability
		$\fr{\phi_\fd}{1 + \phi_\fd}$.
		\label{step:s2construct}
		\item If $S_1 \cup S_2 \notin \Sep_{\leq 2}$, terminate the algorithm with a result of failure.
		Otherwise, let $\hat{S} = S_1 \cup S_2$.
		\label{step:fail2_shatconstruct}
		\item Let $\hat{I}$ denote the independent set with $\hat{I} \cap \cH = \hat{S}$ and with every vertex in $\Q \setminus \cH$
		which does not neighbor $\hat{S}$ included independently with probability $\fr{1}{2}$ each.
		\label{step:final}
	\end{enumerate}
\end{definition}

There are three parts to the proof of Theorem \ref{thm:sampler}.
First, we must show that the defect side of a uniformly random independent set $I$ in $\Qp$ is $\cH$ with probability approximately $q^\cH$.
Next,  we must show that $\AS$ fails (at steps \ref{step:fail1_s1construct} or \ref{step:fail2_shatconstruct}) with low probability.
Finally, we must show that if $\AS$ does not fail, then the set $\hat{S}$ constructed in step \ref{step:fail2_shatconstruct} is close in
distribution to the defect side of a uniformly random independent set with defects on side $\cH$, and that this implies that the total variation
distance between a uniformly random independent set $I$ and the independent set $\hat{I}$ constructed via $\AS$ is small.

\subsection{The defect side of a uniformly random independent set}

Let us begin with a lemma detailing some properties of the defect side of a uniformly random independent set; these statements are simply
reinterpretations of various previous statements throughout the article.

\begin{lemma}
\label{lem:proper}
Let $I$ denote a uniformly random independent set in $\Qp$.
\begin{enumerate}
	\item  For $p > 0$, 
	\begin{equation}
		\P\left[I \cap \cE \in \Sep^\cE \text{ or } I \cap \cO \in \Sep^\cO \text{, but not both}\right] \pto 1.
	\end{equation}

	\item For $p > 0$ and any $S \in \Sep^\cE$ (uniformly),
	\begin{equation}
		\P\left[I \cap \cE = S \middle| I \cap \cE \in \Sep^\cE\right] \psim \fr{\prod_{\fp \in S} \phi_\fp}{\polymerpf^\cE},
	\end{equation}
	and similarly for $\cE$ replaced by $\cO$.

	\item For $p > \fr{1}{2} - \gamma$ (where $\gamma > 0$ is as in Proposition \ref{prop:dimers}),
	\begin{equation}
		\P\left[I \cap \cE \in \Sep^\cE \right] \psim q^\cE,
	\end{equation}
	and similarly for $\cE$ replaced by $\cO$.
\end{enumerate}
\end{lemma}

\begin{proof}[Proof of Lemma \ref{lem:proper}]
Recall from the discussion in Section \ref{sec:polymerdecomp_quenched} that $2^{2^{d-1}} \polymerpf^\cE$ is the number of
independent sets in $\Qp$ whose even side is polymer-decomposable, and similarly for $2^{2^{d-1}} \polymerpf^\cO$.
Now Corollary \ref{cor:toobig} states that the number of independent sets that are \emph{not accounted for} in
\begin{equation}
	\label{eq:sumpolypfs}
	2^{2^{d-1}} \left( \polymerpf^\cE + \polymerpf^\cO \right)
\end{equation}
is an $o(1)$-in-probability fraction of all independent sets.
Further, by Proposition \ref{prop:polymerdecomp}, which states that $\Cnt$ is equal to \eqref{eq:sumpolypfs} times a
$(1+o(1))$-in-probability factor, we observe that the number of independent sets which are \emph{double-counted} in
\eqref{eq:sumpolypfs} is also an $o(1)$-in-probability fraction of all independent sets.
In other words, all but an $o(1)$-in-probability fraction of independent sets are polymer decomposable on exactly
one of the two sides, which proves part (1).

Of these sets, there are $2^{2^{d-1}} \polymerpf^\cE (1-o(1))$ which are decomposable on the even side
and similarly for the odd side.
And, as mentioned at the beginning of Section \ref{sec:iop} (in particular see \eqref{eq:numindS}), the number of independent
sets which have even side $S$ is $2^{2^{d-1} -\N(S)}$, which proves part (2) for the even side; the odd case follows
by the same reasoning.

Finally, by Corollary \ref{cor:approxinprob_iop}, for $p > \fr{1}{2} - \gamma$ we have
\begin{equation}
	\polymerpf^\cE \psim \fakepolymerpf^\cE \cdot \Exp{\Delta^\cE - \tilde{\Delta}^\cE},
\end{equation}
and similarly for $\cO$, which yields part (3), recalling the definition \eqref{eq:qhdef} of $q^\cE$ and $q^\cO$.
\end{proof}

Now let us establish that total variation closeness of defect sides is sufficient to yield total variation closeness of the
full independent sets.

\begin{lemma}
\label{lem:stoi}
Let $I$ denote a uniformly random independent set in $\Qp$ and let $\hat{I}$ denote the output of $\AS$,
given that $\AS$ does not fail.
Let $S$ denote the random subset of $\cE$ given by $I \cap \cE$, conditioned on $I \cap \cE \in \Sep$.
Let $\hat{S}$ denote the random subset of $\cE$ given in step \ref{step:fail2_shatconstruct} of $\AS$, conditioned on choosing
$\cH = \cE$ in step \ref{step:hpick} (and on $\AS$ not failing).
If $\TVbin{S}{\hat{S}} \pto 0$ and the analogous statement holds for $\cO$, then $\TVbin{I}{\hat{I}} \pto 0$ as well.
\end{lemma}

\begin{proof}[Proof of Lemma \ref{lem:stoi}]
We need to exhibit a coupling between $I$ and $\hat{I}$ such that $\P[I \neq \hat{I}] \pto 0$.
By Lemma \ref{lem:proper} (1), we may couple the choice of $\cH$ in step \ref{step:hpick} of $\AS$ to agree with the polymer-decomposable
defect side of $I$ with probability $\pto 1$ (we may simply ignore the event where both $I \cap \cE$ and $I \cap \cO$ are polymer-decomposable,
since this happens with probability $\pto 0$).
The hypothesis of the lemma is that we may couple the defect sides of $I$ and $\hat{I}$ with high probability,
and given that $I$ has defect side $S$, the distribution of the non-defect side is a uniformly random subset of the vertices which do not
neighbor $S$.
This is exactly what is achieved in step \ref{step:final} of $\AS$, meaning that on the event that the defect sides $S$ and $\hat{S}$
are coupled, the entire independent sets $I$ and $\hat{I}$ may be coupled perfectly, finishing the proof.
\end{proof}

Thus, to prove Theorem \ref{thm:sampler}, it suffices to show that $\AS$ fails with probability $\pto 0$, and that,
given that $\AS$ does not fail, the defect side constructed by the algorithm is close in total variation distance to the defect
side of a uniformly random set.
These are captured in the following two lemmas, which together finish the proof via an application of Lemma \ref{lem:stoi}.

\begin{lemma}
\label{lem:samplerdoesntfail}
For all $p > \fr{1}{2} - \gamma$, where $\gamma > 0$ is defined as in Proposition \ref{prop:dimers},
\begin{equation}
	\P[\AS \text{ fails}] \pto 0.
\end{equation}
\end{lemma}

\begin{lemma}
\label{lem:sshatclose}
Let $S$ denote $I \cap \cE$, where $I$ is a uniformly random independent set conditioned on the event $I \cap \cE \in \Sep^\cE$.
Let $\hat{S}$ denote the result of step \ref{step:fail2_shatconstruct} in $\AS$, conditioned on choosing $\cH = \cE$ in step \ref{step:hpick}
and on the algorithm not failing at steps \ref{step:fail1_s1construct} or \ref{step:fail2_shatconstruct}.
For all $p > \fr{1}{2} - \gamma$, where $\gamma > 0$ is defined as in Proposition \ref{prop:dimers}, we have
\begin{equation}
	\TVbin{S}{\hat{S}} \pto 0.
\end{equation}
Additionally, the analogous statement with $\cE$ replaced by $\cO$ also holds.
\end{lemma}

We will prove Lemma \ref{lem:samplerdoesntfail} in Section \ref{sec:sampling_fail}, and we will prove Lemma \ref{lem:sshatclose} in Section \ref{sec:sampling_close} below.

\subsection{The algorithm rarely fails}
\label{sec:sampling_fail}

Let us assume that in step \ref{step:hpick}, we chose $\cH = \cE$.
So we will drop all superscripts of $\cH$ to simplify the notation.
The case where $\cH = \cO$ follows by identical reasoning.
For $\AS$ to fail at all, it must fail at step \ref{step:fail1_s1construct} or step \ref{step:fail2_shatconstruct},
so Lemma \ref{lem:samplerdoesntfail} follows from the following pair of lemmas, which we prove in turn.

\begin{lemma}
\label{lem:samplerdoesntfail_1}
For all $p > 2 - 2^{2/3} \approx 0.413$,
\begin{equation}
	\P[\AS \text{ fails at step \ref{step:fail1_s1construct}}] \pto 0.
\end{equation}
\end{lemma}

\begin{lemma}
\label{lem:samplerdoesntfail_2}
For all $p > \f12 - \gamma$, where $\gamma > 0$ is defined as in Proposition \ref{prop:dimers},
\begin{equation}
	\P\left[\AS \text{ fails at step \ref{step:fail2_shatconstruct}} \middle| \AS \text{ does not fail at step \ref{step:fail1_s1construct}} \right] \pto 0
\end{equation}
\end{lemma}

\begin{proof}[Proof of Lemma \ref{lem:samplerdoesntfail_1}]
The probability that $\AS$ fails at step \ref{step:fail1_s1construct} is the same as the probability that
a sample from $\pi$ is not in $\Sep_{\leq 2}$.
Recalling the definition of $\pi$ from \eqref{eq:pidef} as well as the definitions of $\fakepolymerpf$ and $\fakepolymerpf_{\leq 2}$ from 
\eqref{eq:fakepolymerpfdef} and \eqref{eq:smallpolymerpfdef} respectively, we see that
\begin{equation}
\label{eq:failearly}
	\P[\AS \text{ fails at step \ref{step:fail1_s1construct}}]
	= 
	\fr{\fakepolymerpf - \fakepolymerpf_{\leq 2}}{\fakepolymerpf} \pto 0,
\end{equation}
invoking Proposition \ref{prop:smallpolymers_full} which states that $\fakepolymerpf \psim \fakepolymerpf_{\leq 2}$
for all $p > 2 - 2^{2/3}$.
\end{proof}

\begin{proof}[Proof of Lemma \ref{lem:samplerdoesntfail_2}]
$\AS$ could fail in two different ways at step \ref{step:fail2_shatconstruct}: $S_1 \cup S_2 \notin \Sep_{\leq 2}$ only if either
$S_2 \notin \Sep_2$ or $S_2 \not \sse \Cpt_2(S_1)$.

Let us consider the case $\S_2 \notin \Sep_2$ first.
By a union bound,
\begin{equation}
\label{eq:s2notpd}
	\P[S_2 \notin \Sep_2] \leq \sum_{\fd_1 \sim \fd_2} \phi_{\fd_1} \phi_{\fd_2};
\end{equation}
here we have also used the bound $\fr{\phi_\fd}{1 + \phi_\fd} \leq \phi_\fd$.
Now, by a simple moment calculation done in Lemma \ref{lem:dimermoments} (b) below, the right-hand side in \eqref{eq:s2notpd}
has expectation (in terms of the percolation configuration) tending to $0$ for $p > 0.391$; this shows immediately that
\begin{equation}
\label{eq:s2notpdsmall}
	\P[S_2 \notin \Sep_2] \pto 0.
\end{equation}

Let us next consider the failure via $S_2 \not \sse \Cpt_2(S_1)$.
Using the bound $\fr{\phi_\fd}{1+\phi_\fd} \leq \phi_\fd$ again, as well as another union bound, conditioning on the value of the singleton set $S_1$ constructed in
step \ref{step:fail1_s1construct}, we have
\begin{equation}
\label{eq:s2notcpt}
	\P\left[S_2 \not \sse \Cpt_2(S_1) \middle| \text{step \ref{step:fail1_s1construct} yields } S_1 \right]
	\leq \sum_{\fd \sim S_1} \phi_\fd = \Adj(S_1),
\end{equation}
recalling the notation $\Adj$ from Definition \ref{def:adj}.
Since probabilities are at most $1$, we may replace this upper bound by $\min\{1,\Adj(S_1)\}$, which will be important in the sequel for turning
smallness in probability into smallness in expectation.

The upper bound of $\Adj(S_1)$ given in \eqref{eq:s2notcpt} is actually \emph{not} small under the worst-case choice of $S_1$,
as first alluded to in Remark \ref{rmk:worstvsavg} and discussed in further detail in Section \ref{sec:dimers_firststrategy}.
Instead, we need to use our average-case analysis and leverage the fact that $\Adj(S_1)$ is small with high probability for $S_1 \sim \pi$ conditioned to be in $\Sep_1$,
as discussed in Section \ref{sec:dimers_secondstrategy}; in particular, we will make use of \eqref{eq:adjconditioned},
which is a consequence of Lemmas \ref{lem:seplowerbound} and \ref{lem:adjupperbound} and which states that,
for all $p > \f12 - \gamma$,
\begin{equation}
\label{eq:adjconditioned_again}
	\pi \left[ \Adj(S_1) > \eps \middle| S_1 \in \Sep_1 \right] \pto 0
\end{equation}
for some choice of $\eps = \eps_d \to 0$ as $d \to \infty$.
Note that the same statement holds for $\tAdj$ as well, which is the same as $\Adj$ but with each $\phi_\fd$ replaced by $\fphi_\fd$;
this is simply because $\fphi_\fd \leq \phi_\fd$.
These facts will also be relevant in Section \ref{sec:sampling_close} below, where we prove Lemma \ref{lem:samplerdoesntfail_2}, showing
that the defect side generated by the sampler is close to the defect side of a uniform independent set.

To make use of \eqref{eq:adjconditioned_again}, we will need to understand the distribution of the set $S_1$ constructed in
step \ref{step:fail1_s1construct} of $\AS$.
Since $\AS$ fails if the set $\tilde{S} \sim \pi$ is not in $\Sep_{\leq 2}$,
in the case where the algorithm does not fail, $\tilde{S}$ is conditioned to be in $\Sep_{\leq 2}$, meaning that the probability
of a particular value of $S_1$ constructed in step \ref{step:fail1_s1construct} is
\begin{equation}
	\P\left[\text{step \ref{step:fail1_s1construct} yields } S_1 \middle| \text{step \ref{step:fail1_s1construct} does not fail} \right]
	=
	\fr{1}{\fakepolymerpf_{\leq 2}} \prod_{v \in S_1} \phi_v
	\sum_{\substack{S_2 \in \Sep_2 \\ S_2 \sse \Cpt_2(S_1)}}
	\prod_{\fd \in S_2} \fphi_\fd.
\end{equation}
By Lemma \ref{lem:dimer_sep}, the inner sum over $S_2$ is approximated in probability by $e^{\tilde{\Delta} - \tAdj(S_1)}$,
and this approximation is \emph{uniform} in the choice of $S_1$.
So
\begin{equation}
	\P\left[\text{step \ref{step:fail1_s1construct} yields } S_1 \middle| \text{step \ref{step:fail1_s1construct} does not fail} \right]
	\psim
	\fr{\prod_{v \in S_1} \phi_v}{\fakepolymerpf_{\leq 2}} \cdot \Exp{\tilde{\Delta} - \tAdj(S_1)}.
\end{equation}
Now, by Proposition \ref{prop:dimers} which states that $\fakepolymerpf_{\leq 2} \psim \polymerpf_1 \cdot e^{\tilde{\Delta}}$
for $p > \fr{1}{2} - \gamma$, we have
\begin{equation}
	\label{eq:probofs1}
	\P\left[\text{step \ref{step:fail1_s1construct} yields } S_1 \middle| \text{step \ref{step:fail1_s1construct} does not fail} \right]
	\psim
	\fr{\prod_{v \in S_1} \phi_v}{\polymerpf_1} \cdot e^{-\tAdj(S_1)}
\end{equation}
\emph{uniformly} for any $S_1 \in \Sep_1$ (sets $S_1 \notin \Sep_1$ are not possible at this stage).
Combining \eqref{eq:probofs1} with \eqref{eq:s2notcpt} and recalling the definition of $\pi$ from \eqref{eq:pidef}, we obtain
\begin{align}
	\P\left[S_2 \not \sse \Cpt_2(S_1) \middle| \text{step \ref{step:fail1_s1construct} does not fail} \right]
	&= \E \left[ \P\left[S_2 \not \sse \Cpt_2(S_1) \middle| \text{step \ref{step:fail1_s1construct} yields } S_1 \right]\right] \\
	&\leq \E \left[ \min \{ 1, \Adj(S_1) \} \right] \\
	&\psim \sum_{S_1 \in \Sep_1} \fr{\prod_{v \in S_1} \phi_v}{\polymerpf_1} \cdot e^{-\tAdj(S_1)} \cdot \min\{1,\Adj(S_1)\} \\
	&\leq \E_{S_1 \sim \pi} \left[ \min\{1,\Adj(S_1)\} \middle| S_1 \in \Sep_1 \right],
\end{align}
where the last inequality follows simply from the fact that $\tAdj(S_1) \geq 0$.
Now, using \eqref{eq:adjconditioned_again}, we find that the right-hand side converges to $0$ in probability.
Combining this with \eqref{eq:s2notpdsmall} and \eqref{eq:failearly} finishes the proof.
\end{proof}

\subsection{The defect distributions are close}
\label{sec:sampling_close}

Now we turn to the proof of Lemma \ref{lem:sshatclose}, showing that the defect side from the sampler is close to the defect side of a uniform independent set,
in total variation distance.
Recall that in the context of that lemma, $S$ denotes $I \cap \cE$, where $I$ is a uniformly random independent set conditioned on the event $I \cap \cE \in \Sep^\cE$.
Additionally, $\hat{S}$ denote the result of step \ref{step:fail2_shatconstruct} in $\AS$, conditioned on choosing $\cH = \cE$ in step \ref{step:hpick}
and on the algorithm not failing at steps \ref{step:fail1_s1construct} or \ref{step:fail2_shatconstruct}.
We only focus on the $\cE$ case in this section as the $\cO$ case is identical, and
henceforth we drop all $\cE$ superscripts.

\begin{proof}[Proof of Lemma \ref{lem:sshatclose}]
Given that $\AS$ does not fail, the defect side $\hat{S}$ obtained in step \ref{step:fail2_shatconstruct}
is guaranteed to be in $\Sep_{\leq 2}$, meaning that it can be decomposed as $S_1 \in \Sep_1$ and $S_2 \in \Sep_2$ with $S_2 \sse \Cpt_2(S_1)$.
Additionaly, since $S$ is guaranteed to be in $\Sep$ by our conditioning assumption, we have
\begin{equation}
\label{eq:tvbin1}
	\TVbin{S}{\hat{S}} =
	\sum_{S_1 \in \Sep_1} \sum_{\substack{S_2 \in \Sep_2 \\ S_2 \sse \Cpt_2(S_1)}}
	\left| \P[S = S_1 \cup S_2] - \P[\hat{S} = S_1 \cup S_2] \right|
	+ \sum_{S' \in \Sep \setminus \Sep_{\leq 2}} \P[S = S'].
\end{equation}

By Lemma \ref{lem:proper} (2), for any $S' \in \Sep$,
\begin{equation}
\label{eq:polymerrep}
	\P[S = S'] \psim \fr{\prod_{\fp \in S'} \phi_\fp}{\polymerpf},
\end{equation}
and this approximation is uniform among $S' \in \Sep$.
This implies that the second sum in \eqref{eq:tvbin1} is 
\begin{equation}
\label{eq:trimersumsmall}
	\sum_{S' \in \Sep \setminus \Sep_{\leq 2}} \P[S = S'] \psim \fr{\polymerpf - \polymerpf_{\leq 2}}{\polymerpf} \pto 0,
\end{equation}
recalling the definitions of $\polymerpf$ and $\polymerpf_{\leq 2}$ from \eqref{eq:polymerpfdef} and \eqref{eq:smallpolymerpfdef}
and
invoking Proposition \ref{prop:smallpolymers} which states that $\polymerpf \psim \polymerpf_{\leq 2}$ for $p > 2 - 2^{2/3}$.

Now let us turn to the first sum.
For $S$ the defect side of the uniform independent set, \eqref{eq:polymerrep} yields
\begin{equation}
\label{eq:uniforms1s2prob}
	\P[S = S_1 \cup S_2] \psim \fr{\prod_{v \in S_1} \phi_v \prod_{\fd \in S_2} \phi_\fd}{\polymerpf}
	\psim \fr{\prod_{v \in S_1} \phi_v \prod_{\fd \in S_2} \phi_\fd}{\polymerpf_1 \cdot e^\Delta},
\end{equation}
where we have invoked Proposition \ref{prop:smallpolymers} (which states that $\polymerpf \psim \polymerpf_{\leq2}$)
and Proposition \ref{prop:dimers} (which states that $\polymerpf_{\leq 2} \psim \polymerpf_1 \cdot e^\Delta$)
to approximate $\polymerpf$ by $\polymerpf_1 \cdot e^\Delta$.
It's crucial to note that all of these approximations are uniform in $S_1$ and $S_2$, in the sense that the error terms are bounded by quantities tending (in probability)
to $0$ which do not depend on $S_1$ or $S_2$.

Now let us examine the corresponding probability for $\hat{S}$ constructed in step \ref{step:fail2_shatconstruct} of $\AS$.
Recall from \eqref{eq:probofs1} that
\begin{equation}
\label{eq:algs1prob}
	\P\left[\text{step \ref{step:fail1_s1construct} yields } S_1 \middle| \text{step \ref{step:fail1_s1construct} does not fail} \right]
	\psim
	\fr{\prod_{v \in S_1} \phi_v}{\polymerpf_1} \cdot e^{-\tAdj(S_1)}.
\end{equation}
Now, given that we obtain $S_1$ in step \ref{step:fail1_s1construct} and that $\AS$ does not fail at step \ref{step:fail2_shatconstruct},
the probability of obtaining
a particular value of $S_2$ in that step (which must satisfy $S_2 \in \Sep_2$ and $S_2 \sse \Cpt_2(S_1)$) is
\begin{equation}
\label{eq:algs2prob}
	\P\left[\text{step \ref{step:fail2_shatconstruct} yields } S_2 \middle| \text{step \ref{step:fail1_s1construct} yields } S_1 \right]
	= \fr{\prod_{\fd \in S_2} \phi_\fd}{\sum_{\substack{S_2' \in \Sep_2 \\ S_2' \sse \Cpt_2(S_1)}} \prod_{\fd \in S_2'} \phi_\fd}
	\psim \fr{\prod_{\fd \in S_2} \phi_\fd}{e^{\Delta}} e^{\Adj(S_1)},
\end{equation}
where we have applied Lemma \ref{lem:dimer_sep} which approximates the big sum in the denominator by $e^{\Delta - \Adj(S_1)}$
 and the approximation here is again uniform in $S_1$.

Combining \eqref{eq:uniforms1s2prob}, \eqref{eq:algs1prob} and \eqref{eq:algs2prob}, we find that the first sum in \eqref{eq:tvbin1}
satisfies
\begin{align}
	\sum_{S_1 \in \Sep_1} \sum_{\substack{S_2 \in \Sep_2 \\ S_2 \sse \Cpt_2(S_1)}}
	&\left| \P[S = S_1 \cup S_2] - \P[\hat{S} = S_1 \cup S_2] \right| \\
	&\psim \sum_{S_1 \in \Sep_1} \sum_{\substack{S_2 \in \Sep_2 \\ S_2 \sse \Cpt_2(S_1)}}
	\fr{\prod_{v \in S_1} \phi_v \prod_{\fd \in S_2} \phi_\fd}{\polymerpf_1 \cdot e^\Delta} \left| 1 - e^{\Adj(S_1) - \tAdj(S_1)} \right| \\
	&= \sum_{S_1 \in \Sep_1} \fr{\prod_{v \in S_1} \phi_v}{\polymerpf_1} \left|1 - e^{\Adj(S_1) - \tAdj(S_1)} \right|
	\cdot \underbrace{\fr{\sum_{\substack{S_2 \in \Sep_2 \\ S_2 \sse \Cpt_2(S_1)}} \prod_{\fd \in S_2} \phi_\fd}{e^\Delta}}_{(*)} \\
\end{align}
To handle the factor labeled $(*)$, we apply Lemma \ref{lem:dimer_sep}, which states that
\begin{equation}
	\sum_{\substack{S_2 \in \Sep_2 \\ S_2 \sse \Cpt_2(S_1)}} \prod_{\fd \in S_2} \phi_\fd \psim \Exp{\sum_{\fd \not \sim S_1} \phi_\fd} = \Exp{\Delta - \Adj(S_1)},
\end{equation}
where the approximation is uniform in $S_1$.
So, continuing from the previous display, the first sum in \eqref{eq:tvbin1} is
\begin{equation}
	\psim \sum_{S_1 \in \Sep_1} \fr{\prod_{v \in S_1} \phi_v}{\polymerpf_1} \left|1 - e^{\Adj(S_1) - \tAdj(S_1)} \right| e^{-\Adj(S_1)}
	= \E_\pi \left[ \left| e^{-\Adj(S_1)} - e^{-\tAdj(S_1)} \right| \middle| S_1 \in \Sep_1 \right].
\end{equation}
Now, since $\Adj(S_1), \tAdj(S_1) \geq 0$ and the derivative of $t \mapsto e^{-t}$ is bounded by $1$ for $t \geq 0$, the right-hand side above is at most
\begin{equation}
	\E_\pi \left[ | \Adj(S_1) - \tAdj(S_1) | \middle| S_1 \in \Sep_1 \right],
\end{equation}
which tends to $0$ in probability by \eqref{eq:adjconditioned_again} as well as the corresponding fact for $\tAdj$.
In combination with \eqref{eq:trimersumsmall}, this finishes the proof of the lemma.
\end{proof}

%% file: sections/moments.tex
\section{Moment calculations}
\label{sec:moments}

Recall from \eqref{eq:phidef} and \eqref{eq:fakephidef} that for each polymer $\fp$, we have defined $\phi_\fp = 2^{-\N(\fp)}$, and
$\fphi_\fp = \prod_{v \in \fp} 2^{-\N(v)}$.
Additionally we introduce the notation
\begin{equation}
    \Phi_k = \sum_{v \in \cH} \phi_v^k
\end{equation}
for each $k \geq 0$, and we recall from \eqref{eq:philogdef} and \eqref{eq:deltadef} the notation
\begin{equation}
    \Phi_{\log} = \sum_{v \in \cH} \log(1 + \phi_v), \qquad
    \Delta = \sum_{\fd \in \Dim^\cH} \phi_\fd
    \qquad \text{and} \qquad
    \tilde{\Delta} = \sum_{\fd \in \Dim^\cH} \fphi_\fd.
\end{equation}
Note that these quantities depend on the side $\cH \in \{\cE, \cO\}$, but for most of the following
moment calculations, the result will be identical for both sides and for notational convenience
we will not mention the side explicitly.
The exception will be the covariance of $\Phi_{\log}$ for the even and odd side, and there we will
use the notation $\Phi_{\log}^\cE$ and $\Phi_{\log}^\cO$ to distinguish between the two sides.

\subsection{Singleton moments}

Recall the following constants defined in \eqref{eq:mu1kdef} which turn out to be the expectations of $\Phi_k$,
or equivalently $2^{d-1}$ times the $k$th moments of $\phi_v$.
\begin{equation}
    \mu_1^k = \fr{1}{2} \Rnd{2 - \fr{2^k-1}{2^{k-1}} p}^d.
\end{equation}
We will also use the notation
\begin{equation}
    \sigma^2 = \mu_1^2 = \fr{1}{2} \Rnd{2 - \fr{3}{2} p}^d
    \qquad \text{and} \qquad
    \sigma = \sqrt{\sigma^2}.
\end{equation}

We first collect the following moments of singleton variables and sums of singleton variables.

\begin{lemma}
\label{lem:moments}
The following facts hold for all $p \in (0,1)$.
\begin{enumerate}[(a)]
    \item For each $v$ in $\cH$ and each $k \geq 0$, we have
    \begin{align}
        \E[\phi_v^k] &= \left(1 - \frac{2^k-1}{2^k} p\right)^d, \\
        \Var(\phi_v^k) &= \E[\phi_v^{2k}](1 - o(1)).
    \end{align}

    \item For each $k \geq 0$, we have
    \begin{align}
        \E[\Phi_k] &= \mu_1^k, \\
        \Var(\Phi_k) &= \mu_1^{2k} (1-o(1)).
    \end{align}
\end{enumerate}
\end{lemma}

\begin{proof}[Proof of Lemma \ref{lem:moments}]
\qquad
\begin{enumerate}[(a)]
    \item 
    Since $\N(v)$ is a sum of $d$ independent $\Ber(p)$ variables,
    \begin{equation}
        \E[\phi_v^k] = \left((1-p) + 2^{-k} p\right)^d = \left(1 - \frac{2^k-1}{2^k} p\right)^d.
    \end{equation}
    For the second claim, it suffices to show that $\E[\phi_v^k]^2 \ll \E[\phi_v^{2k}]$,
    or in other words, that
    \begin{equation}
        \left(1 - \frac{2^k-1}{2^k} p\right)^2 < 1 - \frac{2^{2k}-1}{2^{2k}} p.
    \end{equation}
    The validity of this inequality for $0 < p < 1$ can be checked by multiplying
    both sides by $2^{2k}$ and expanding the square; this gives
    \begin{align}
        2^{2k} - 2 \cdot 2^k (2^k-1) p + (2^k-1)^2 p^2 &< 2^{2k} - (2^{2k} - 1) p \\
        - 2 \cdot 2^k (2^k-1) + (2^k - 1)^2 p &< -(2^k+1) (2^k-1) \\
        (2^k-1)^2 p &< (2^k-1)(2 \cdot 2^k - (2^k + 1)) \\
        (2^k-1) p &< 2 \cdot 2^k - (2^k+1) = 2^k-1,
    \end{align}
    which holds as $p < 1$.

    \item
    These claims follow from the previous part, using the fact that the number of vertices
    in each side is $2^{d-1}$.
    The claim about the variance also uses the independence of $\phi_v$ amongst $v \in \cH$.
    \qedhere
\end{enumerate}
\end{proof}

Next we obtain approximations for the mean and variance of $\Phi_{\log}$ which are used in Section \ref{sec:clt}.

\begin{lemma}
\label{lem:logmoments}
The following approximations hold for $\Phi_{\log}$.
\begin{enumerate}[(a)]
    \item
    For all $p \in (0,1)$, we have
    \begin{equation}
        \Var(\Phi_{\log}) = \sigma^2 (1 + o(1))
    \end{equation}
    \item 
    For $p > 0.455$, we have
    \begin{equation}
        \E[\Phi_{\log}] = \mu_1^1 - \fr{\mu_1^2}{2} + \fr{\mu_1^3}{3} + o(\sigma)
    \end{equation}
\end{enumerate}
\end{lemma}

\begin{proof}[Proof of Lemma \ref{lem:logmoments}]
\qquad
\begin{enumerate}[(a)]
\item 
Since $x \geq \log(1 + x) \geq x - \fr{x^2}{2}$ for all $x \in [0,1]$, 
using Lemma \ref{lem:moments} (a), we have
\begin{equation}
    \Rnd{1 - \fr{1}{2}p}^d \geq
    \E[\log(1 + \phi_v)] \geq
    \Rnd{1 - \fr{1}{2}p}^d - \fr{1}{2} \Rnd{1 - \fr{3}{4}p}^d,
\end{equation}
and since $x^2 \geq \log^2(1+\phi_v) \geq x^2 - x^3 + \fr{1}{4} x^4$, we also have
\begin{equation}
    \Rnd{1 - \fr{3}{4}p}^d \geq
    \E[\log^2(1+\phi_v)] \geq
    \Rnd{1 - \fr{3}{4}p}^d - \Rnd{1 - \fr{7}{8}p}^d + \fr{1}{4} \Rnd{1 - \fr{15}{16}p}^d.
\end{equation}
Combining these, we obtain
\begin{align}
    \label{eq:varlogUB}
    \Rnd{1 - \fr{3}{4}p}^d &- \Rnd{1 - \fr{1}{2}p}^{2d} 
    + \Rnd{1 - \fr{1}{2}p}^d \Rnd{1 - \fr{3}{4}p}^d - \fr{1}{4} \Rnd{1 - \fr{3}{4}p}^{2d} \\
    &\geq \Var(\log(1+\phi_v)) \geq \\
    \label{eq:varlogLB}
    \Rnd{1 - \fr{3}{4}p}^d &- \Rnd{1 - \fr{7}{8}p}^d + \fr{1}{4} \Rnd{1 - \fr{15}{16}p}^d
    - \Rnd{1 - \fr{1}{2}p}^{2d}.
\end{align}
Now for both the upper bound \eqref{eq:varlogUB} and the lower bound \eqref{eq:varlogLB},
each term after the first is exponentially smaller than the first term.
This is clear by inspection for all terms other than $\Rnd{1 - \fr{1}{2}p}^{2d}$,
and the fact that this term is also exponentially smaller than $\Rnd{1 - \fr{3}{4}p}^d$
follows from the same argument as in Lemma \ref{lem:moments} above.
This means that
\begin{equation}
    \Var(\log(1+\phi_v)) = \Rnd{1 - \fr{3}{4}p}^d (1 + o(1)),
\end{equation}
and so since $\Phi_{\log}$ is a sum of $2^{d-1}$ independent random variables of the form
$\log(1+\phi_v)$, we have
\begin{equation}
    \Var(\Phi_{\log}) = \sigma^2(1 + o(1)),
\end{equation}
since $\sigma^2 = 2^{d-1} \Rnd{1 - \fr{3}{4}p}^d$.

\item
This time we use the fact that
\begin{equation}
    x - \fr{x^2}{2} + \fr{x^3}{3}
    \geq \log(1+x) \geq x - \fr{x^2}{2} + \fr{x^3}{3} - \fr{x^4}{4}
\end{equation}
for $x \geq 0$.
This gives
\begin{equation}
    0 \geq \E[\Phi_{\log}] - \Rnd{\mu_1^1 - \fr{\mu_1^2}{2} + \fr{\mu_1^3}{3}} \geq - \fr{\mu_1^4}{4}
\end{equation}
We thus want to show that $\mu_1^4 \ll \sigma$, which is implied by
\begin{equation}
    \Rnd{2 - \fr{15}{8} p}^2 < \Rnd{2 - \fr{3}{2} p}.
\end{equation}
The difference between the left-hand side and the right-hand side is
\begin{equation}
    2 - 6 p + \fr{225}{64} p^2,
\end{equation}
which has roots
\begin{equation}
    \fr{6 \pm \sqrt{36 - 225/8}}{225/32} \approx 0.4542, 1.2524.
\end{equation}
Therefore, for $0.455 < p < 1$, we have $\mu_1^4 \ll \sigma$, finishing the proof.
\qedhere
\end{enumerate}
\end{proof}

The following lemma gives the independence input for Lemma \ref{lem:clt} allowing us to show
that the scaling limit of $(\Phi_{\log}^\cE, \Phi_{\log}^\cO)$ is a pair of independent Gaussians.

\begin{lemma}
\label{lem:cov}
For all $p \in (0,1)$,
\begin{equation}
    \Cov(\Phi_{\log}^\cE, \Phi_{\log}^\cO) \ll \Var(\Phi_{\log}).
\end{equation}
\end{lemma}

\begin{proof}[Proof of Lemma \ref{lem:cov}]
First note that $\Cov(\Phi^\cE_{\log}, \Phi^\cO_{\log}) \geq 0$ by the FKG inequality since both
random variables are decreasing in the percolation configuration.
Now, since $\phi_v$ and $\phi_w$ are independent when $v \not\sim w$ in $\Q$, we have
\begin{align}
    \Cov(\Phi_{\log}^\cE, \Phi_{\log}^\cO)
    &= \sum_{\substack{v \in \cE, u \in \cO \\ v \sim u}}
    \Rnd{\E[\log(1+\phi_v) \log(1 +\phi_u)] -
        \E[\log(1+\phi_v)] \E[\log(1+\phi_u)]} \\
    &\leq \sum_{\substack{v \in \cE, u \in \cO \\ v \sim u}}
        \E[\log(1 + \phi_v) \log(1 + \phi_u)] \\
    &\leq \sum_{\substack{v \in \cE, u \in \cO \\ v \sim u}}
        \E[\phi_v \phi_u],
\end{align}
using the inequality $\log(1+x) \leq x$ for $x \in [0,1]$ at the last step.
Now, when $v \sim w$, there is exactly one edge connecting them in $\Q$, and there are
$2d-2$ edges which are incident to only one of the two vertices.
So $\phi_v \phi_w \sim 2^{-2 \cdot \Ber(p)} 2^{-\Bin(2d-2,p)}$, where the two random variables 
in the exponent are independent. 
Additionally, for each of the $2^{d-1}$ vertices $v \in \cE$, there are $d$ vertices $u \in \cO$ such that
$v \sim u$, and so we have
\begin{align}
    \Cov(\Phi_{\log}^\cE, \Phi_{\log}^\cO)
    &= 2^{d-1} d \Rnd{1 - \fr{3}{4}p} \Rnd{1 - \fr{1}{2}p}^{2d-2},
\end{align}
which is $\ll \sigma^2$ since
\begin{equation}
    2 \cdot \Rnd{1 - \fr{1}{2}p}^2 = 2 - 2 p + \fr{1}{2} p^2 < 2 - 2p + \fr{1}{2} p = 2 - \fr{3}{2} p.
\end{equation}
The conclusion thus follows by Lemma \ref{lem:logmoments} (a).
\end{proof}

\subsection{Dimer moments}

Recall the following constants defined in \eqref{eq:mu2defintro} and \eqref{eq:mu2tildedefintro},
which turn out to be the expectations of $\Delta$ and $\tilde{\Delta}$ respectively:
\begin{align}
    \mu_2 &= \fr{d(d-1)}{4} \Fr{(2-p)^2}{2}^{d} \cdot \Fr{1 + (1-p)^2}{2 - p}^2, \\
    \tilde{\mu}_2 &= \fr{d(d-1)}{4} \Fr{(2-p)^2}{2}^{d}.
\end{align}
Notice that the above two expressions differ only by a multiple which does not depend on $d$.
Additionally, the exponential orders of $\mu_2$ and $\tilde{\mu}_2$ are the same as that of $\fr{1}{2^d} (\mu_1^1)^2$.
Intuitively, this is because each $\phi_\fd$ and $\fphi_\fd$ behaves approximately like a product of two independent
singleton variables $\phi_v$, since most of the neighbors of $\fd$ are not neighbors of both vertices in $\fd$.
Moreover, there are the same number of singletons and dimers, up to a $\poly(d)$ factor.

\begin{lemma}
\label{lem:dimermoments}
Recall the definitions of $\Delta$ and $\tilde{\Delta}$ from \eqref{eq:deltadef} above.
\begin{enumerate}[(a)]
    \item For all $p \in (0,1)$, we have
    \begin{equation}
        \E[\Delta] = \mu_2
        \qquad \text{and} \qquad
        \E[\tilde{\Delta}] = \tilde{\mu}_2.
    \end{equation}

    \item For all $p > 0.391$, we have
    \begin{equation}
        \Var(\Delta) \ll 1
        \qquad \text{and} \qquad
        \Var(\tilde{\Delta}) \ll 1.
    \end{equation}
    Additionally for $p > 0.391$, we have
    \begin{equation}
        \E[ \sum_{\fd_1 \sim \fd_2} \phi_{\fd_1} \phi_{\fd_2} ] \ll 1
        \qquad \text{and} \qquad
        \E[ \sum_{\fd_1 \sim \fd_2} \fphi_{\fd_1} \fphi_{\fd_2} ] \ll 1,
    \end{equation}
    where $\fd_1 \sim \fd_2$ if the neighborhoods of the two dimers overlap.
\end{enumerate}
\end{lemma}

\begin{proof}[Proof of Lemma \ref{lem:dimermoments}]
\qquad
\begin{enumerate}[(a)]
    \item 
    Let us consider $\tilde{\Delta}$ first.
    For $\fd = \{u,v\}$, we have
    \begin{equation}
        \fphi_\fd = 2^{-\N(u)-\N(v)} \stackrel{d}{=} 2^{-\Bin(2d,p)},
    \end{equation}
    and so by the same reasoning as in Lemma \ref{lem:moments} (a), we have
    \begin{equation}
        \E[\fphi_\fd] = \Rnd{1 - \fr{1}{2}p}^{2d}.
    \end{equation}
    Now there are $2^{d-1} \fr{d(d-1)}{2}$ possible dimers $\fd \in \Dim^\cH$, and so we have
    \begin{equation}
        \E[\tilde{\Delta}] = 2^{d-1} \fr{d(d-1)}{2} \Rnd{1 - \fr{1}{2}p}^{2d} = \tilde{\mu}_2.
    \end{equation}
    As for $\Delta$, we must consider $\phi_\fd$ instead of $\fphi_\fd$.
    Now for $\fd = \{u,v\}$, we know that $u$ and $v$ share a neighbor; in this case, in fact, 
    $u$ and $v$ share exactly two neighbors since their binary string representations differ in exactly two bits,
    and swapping either one of these bits leads to a shared neighbor.
    Each of these common neighbors is a neighbor of $\fd$ in $\Qp$ with probability $1-(1-p)^2$,
    since it is not a neighbor exactly when both relevant edges are absent.
    There are also $2d-2$ vertices in the opposite side which are neighbors of exactly one of $u$ or $v$.
    So we have
    \begin{equation}
        \phi_\fd \stackrel{d}{=} 2^{-\Bin(2,1-(1-p)^2) - \Bin(2d-2,p)},
    \end{equation}
    where the two Binomials are independent.
    So we have
    \begin{align}
        \E[\phi_\fd] &= \Rnd{(1-p)^2 + \fr{1}{2} (1- (1-p)^2)}^2 \Rnd{1 - \fr{1}{2}p}^{2d-2} \\
        &= \Fr{\fr{1}{2} + \fr{1}{2} (1-p)^2}{1 - \fr{1}{2}p}^2 \Rnd{1 - \fr{1}{2}p}^{2d} \\
        &= \Fr{1 + (1-p)^2}{2 - p}^2 \Rnd{1 - \fr{1}{2}p}^{2d}.
    \end{align}
    This differs from $\E[\fphi_\fd]$ by the same factor as $\mu_2$ differs from $\tilde{\mu}_2$, and so
    we obtain $\E[\Delta] = \mu_2$.

    \item
    First note that $\phi_{\fd_1}$ is independent of $\phi_{\fd_2}$ if $\fd_1 \cap \fd_2 = \emptyset$,
    and for each $\fd_1$ there are at most $2 d^2$ dimers $\fd_2$ for which this does not hold.
    So we have
    \begin{align}
        \Var(\Delta) &\leq \sum_{\substack{\fd_1, \fd_2 \in \Dim \\ \fd_1 \cap \fd_2 \neq \emptyset}} \E[\phi_{\fd_1} \phi_{\fd_2}] \\
        &\leq \sum_{\substack{\fd_1, \fd_2 \in \Dim \\ \fd_1 \cap \fd_2 \neq \emptyset}} \sqrt{\E[\phi_{\fd_1}^2] \E[\phi_{\fd_2}^2]} \\
        &\leq \poly(d) \cdot 2^d \cdot \E[\phi_\fd^2].
    \end{align}
    Now using the expression for $\phi_\fd$ from the previous part, we see that
    \begin{equation}
        \E[\phi_\fd^2] = c_p \cdot \E[2^{-2 \cdot \Bin(2d,p)}]
    \end{equation}
    for some constant $c_p$ depending only on $p$.
    Now, as in Lemma \ref{lem:moments} (a), we have
    \begin{equation}
        \E[2^{-2 \cdot \Bin(2d,p)}] = \Rnd{1 - \fr{3}{4}p}^{2d}.
    \end{equation}
    Thus
    \begin{equation}
        \Var(\Delta) \lesssim \poly(d) \cdot 2^d \Rnd{1 - \fr{3}{4}p}^{2d},
    \end{equation}
    and this is $\ll 1$ when $\Rnd{1 - \fr{3}{4}p}^2 < \fr{1}{2}$, which happens for
    $p > \fr{2}{3} (2 - \sqrt{2}) \approx 0.3905$.
    Additionally, $\tilde{\Delta}$ is a sum of i.i.d.\ random variables $\fphi_\fd \leq \phi_\fd$, and so by the
    same reasoning we have $\Var(\tilde{\Delta}) \ll 1$ as well.

    The second statement, about the sums over $\fd_1 \sim \fd_2$, follows by the same reasoning since for any dimer
    $\fd_1$ there are at most $\poly(d)$ dimers $\fd_2$ which satisfy $\fd_1 \sim \fd_2$.
    \qedhere
\end{enumerate}
\end{proof}

%% file: sections/entropy.tex
\section{Entropy approximations for binomials}
\label{sec:entropy}

Recall the binary (relative) entropy between two Bernoulli distributions with parameters $p$ and $q$,
defined in \eqref{eq:hpdef} above:
\begin{equation}
	H_p(q) = q \log_2\f{q}{p} + (1 - q) \log_2\f{1 - q}{1 - p}.
\end{equation}
The following lemma is standard; see for instance \cite[Page 353]{cover}.

\begin{lemma}
	\label{lem:bin}
	For $p,q \in [0,1]$ with $q \geq p$, we have $\P[\Bin(n, p) \geq nq] \leq 2^{-n \cdot H_p(q)}$.
\end{lemma}

Next, we need a lemma which bounds the binary relative entropy in terms of a more easily accessible quantity.

\begin{lemma}
	\label{lem:entropy}
	For any $0 < p < x < 1$ with $x \geq 10p$, we have
	\[
		H_p(x) \geq \f12 \cdot x \log_2(x/p)
	\]
\end{lemma}

We remark that the constant 10 is not the best possible constant here, although 5 seems to be closer to being optimal.

\begin{proof}[Proof of Lemma \ref{lem:entropy}]
	Recall from \eqref{eq:hpdef} that 
	\[
		H_p(x) = x \log_2\f{x}{p} + (1 - x) \log_2\f{1 - x}{1 - p}.
	\]
	Since $x > p$, $1 -x < 1 - p$, and thus the second summand is negative. Further, the ratio between the (negative) second 
	summand and the first summand is
	\begin{align}
		\f{(1 - x) \log_2 \f{1 - p}{1 - x}}{x \log_2 \f{x}{p}} &= \f{(1 - x) \log_2\Rnd{1 + \f{x - p}{1 - x}}}{x
		\log_2 \f x p} \\
															   &= \f{(1 - x) \ln\Rnd{1 + \f{x - p}{1 - x}}}{x
															   \ln\f x p} \\
															   &\leq \f{(1 - x) \Rnd{\f{x - p}{1 - x}}}{x \ln \f x p} \tag*{($\ln (1 + t) \leq t$)}\\
															   &= \f{1 - \f p x}{\ln \f x p} = \fr{\fr{p}{x} - 1}{\ln \f p x}.
	\end{align}
	Now since we assumed that $\fr{p}{x} \leq \fr{1}{10}$ and $\fr{t-1}{\ln t}$ is an increasing function of $t > 0$,
	the right-hand side above is at most $\fr{0.1 - 1}{\ln 0.1} \approx 0.3908 < \f12$,
	finishing the proof.
\end{proof}

The following lemma is used repeatedly in Section \ref{sec:dimers} to show that various optimization problems have valid
solutions for a range of $p$ which includes values slightly below $\fr{1}{2}$.

\def\ss{s^\star}
\begin{lemma}
	\label{lem:minentropy}
	For every fixed $m \geq 1$, let 
	\[
		f(p) \coloneq \inf_{s \in (0, 1)} ms + H_p(s), \quad p \in (0, 1).
	\]
	Then $f$ is continuous and strictly increasing in $p$, and
	\[
		f\Rnd{\tfrac{1}{2}} = m + 1 - \log_2\Rnd{2^{m} + 1}.
	\]
\end{lemma}

\begin{proof}[Proof of Lemma \ref{lem:minentropy}]
	\def\pd{\partial}
	Let us write $g(p, s) = ms + H_p(s)$ for
	the objective function, and $\ss(p)$ for the minimizer. Then, to compute $\ss = \ss(p)$, note that $g$ is convex in
	$s$, and thus it suffices to equate its derivative to zero. The
	derivative of $H_p(x)$ is $\log_2(x/p) - \log_2((1 - x)/(1 - p))$. Setting $\pd g/\pd s$ to zero, we get
	\begin{align}
		m + \Rnd{\log_2\f{\ss}{p} - \log_2\f{1 - \ss}{1 - p}} = 0 & \iff \f{\ss}{1 - \ss} = 2^{-m} \f{p}{1 - p} \\
																  & \iff \ss = \f{p}{(1 - p) 2^m + p}.
	\end{align}
	In particular, $\ss < p$ always. Observe that continuity of $f$ follows from the continuity of $H_p$ and $\ss(p)$.
	The derivative of $f$ is
	\begin{equation}
		f'(p) = \f{\pd g}{\pd p}\Big|_{s = \ss(p)} + \f{\pd g}{\pd s}\Big|_{s = \ss(p)} \cdot (\ss)'(p)
	\end{equation}
	The second term is zero due to the optimality of $\ss$. Further, since $\pd H_p(x)/\pd p$ is positive when $x < p$, 
	we have that $f'(p) > 0$ always, showing that $f$ is strictly increasing.

	Setting $p = 1/2$, we get $\ss = 1/(2^m + 1)$, and thus
	\begin{align}
		f(1/2) &= g(1/2, \ss(1/2)) \\ &= \f{m}{2^m + 1} + \f{1}{2^m + 1} \log_2 \F{1}{1/2 \cdot (2^m + 1)} + \f{2^m}{2^m +
		1}\log_2\F{2^m}{1/2 \cdot (2^m + 1)} \\
								  &= m + 1 - \log_2(2^m + 1),
	\end{align}
	as required.
\end{proof}

Finally, the following lemma expresses the exponential quantity appearing in the dimer moment calculations
of Lemma \ref{lem:dimermoments} as the solution of an optimization problem.

\begin{lemma}
\label{lem:dimersumoptimization}
For all $p \in (0,1)$, 
\begin{equation}
	\fr{1}{\poly(d)} \cdot 2^{d (1-2\ss-2H_p(\ss))} \leq 
	\Fr{(2-p)^2}{2}^{d} \leq \poly(d) \cdot 2^{d (1-2\ss-2H_p(\ss))},
\end{equation}	
where $\ss$ maximizes $1 - 2 s - 2 H_p(s)$ over $s \in [0,1]$.
In particular, this implies that
\begin{equation}
	\fr{(2-p)^2}{2} = 2^{1 - 2\ss - 2 H_p(\ss)}.
\end{equation}
\end{lemma}

\def\ks{k^\star}
\begin{proof}[Proof of Lemma \ref{lem:dimersumoptimization}]
By the same argument as in Lemma \ref{lem:moments},
\begin{equation}
\label{eq:binsum}
	\fr{(2-p)^{2d}}{2^d} = 2^d \E [2^{-\Bin(2d,p)}]
	= 2^d \sum_{k=0}^{2d} 2^{-k} \P[\Bin(2d,p) = k].
\end{equation}
Now a standard derivation using Stirling's approximation (see e.g.\ \cite[Page 353]{cover}),
we have
\begin{equation}
	\fr{1}{2d+1} 2^{2d H(k/2d)} \leq \binom{2d}{k} \leq 2^{2d H(k/2d)},
\end{equation}
where
\begin{equation}
	H(q) = - q \log_2 q - (1-q) \log_2 (1-q)
\end{equation}
is the standard binary entropy function.
This immediately implies that
\begin{equation}
	\fr{1}{2d+1} 2^{-2dH_p(k/2d)} \leq \P[\Bin(2d,p) = k] \leq 2^{-2dH_p(k/2d)}.
\end{equation}
Now using the fact that there are polynomially many terms in the sum
\eqref{eq:binsum}, we find that
\begin{equation}
\label{eq:kstar}
	\fr{1}{\poly(d)} \cdot 2^{d(1-2 (\ks/2d)-2H_p(\ks/2d))} \leq
	\fr{(2-p)^{2d}}{2^d} \leq \poly(d) \cdot 2^{d(1-2(\ks/2d)-2H_p(\ks/2d))}
\end{equation}
for some fixed polynomials in the upper and lower bounds, where $\ks \in \{0,\dotsc,2d\}$ maximizes
the exponent $1 - 2(\ks/2d) - 2 H_p(\ks/2d)$.

Now, for any fixed $p \in (0,1)$, the function $g(s) = 1 - 2 s - 2 H_p(s)$ is concave.
Thus it has a unique maximum $\ss \in [0,1]$, and the optimizer $\ks$ is one of the two points
in $\{0,\dotsc,2d\}$ nearest to $\ss$ on either side.
This means that $\left|\fr{\ks}{2d} - \ss\right| \leq \fr{1}{2d}$, and so since $g'(\ss) = 0$ we must have
\begin{equation}
	\left| g(\ss) - g(\ks) \right| \leq \fr{C}{d^2}
\end{equation}
for some constant $C$, by Taylor expansion.
This combined with \eqref{eq:kstar} finishes the proof.
\end{proof}

%% file: ref.bib
@article{ks,
  title={Independent sets in random subgraphs of the hypercube},
  author={Kronenberg, Gal and Spinka, Yinon},
  journal={arXiv preprint arXiv:2201.06127},
  year={2022}
}

@article{galvinthresh,
  title={A threshold phenomenon for random independent sets in the discrete hypercube},
  author={Galvin, David},
  journal={Combinatorics, probability and computing},
  volume={20},
  number={1},
  pages={27--51},
  year={2011},
  publisher={Cambridge University Press}
}

@incollection{stein,
  title={A normal approximation for the number of local maxima of a random function on a graph},
  author={Baldi, Pierre and Rinott, Yosef and Stein, Charles},
  booktitle={Probability, statistics, and mathematics},
  pages={59--81},
  year={1989},
  publisher={Elsevier}
}

@article{ross,
  title={Fundamentals of {S}tein’s method},
  author={Ross, Nathan},
  journal={Probability Surveys},
  volume={8},
  pages={210--293},
  year={2011}
}

@article{brcgw,
  title={Gaussian to log-normal transition for independent sets in a percolated hypercube},
  author={Chowdhury, Mriganka Basu Roy and Ganguly, Shirshendu and Winstein, Vilas},
  journal={arXiv preprint arXiv:2410.07080},
  year={2024}
}

@article{bo,
  title={Complete matchings in random subgraphs of the cube},
  author={Bollob{\'a}s, B{\'e}la},
  journal={Random Structures \& Algorithms},
  volume={1},
  number={1},
  pages={95--104},
  year={1990},
  publisher={Wiley Online Library}
}

@book{cover,
  title={Elements of information theory},
  author={Cover, Thomas M},
  year={1999},
  publisher={John Wiley \& Sons}
}

@inproceedings{ch,
  title={Hamiltonicity of random subgraphs of the hypercube},
  author={Condon, Padraig and D{\'\i}az, Alberto Espuny and Gir{\~a}o, Antonio and K{\"u}hn, Daniela and Osthus, Deryk},
  booktitle={Proceedings of the 2021 ACM-SIAM Symposium on Discrete Algorithms (SODA)},
  pages={889--898},
  year={2021},
  organization={SIAM}
}

@incollection{be,
  title={The evolution of the cube},
  author={Bollob{\'a}s, B{\'e}la},
  booktitle={North-Holland Mathematics Studies},
  volume={75},
  pages={91--97},
  year={1983},
  publisher={Elsevier}
}

@article{bp,
  title={The probability of connectedness of a random subgraph of an n-dimensional cube},
  author={Burtin, Ju D},
  journal={Problemy Peredaci Informacii},
  volume={13},
  number={2},
  pages={90--95},
  year={1977}
}

@article{ee,
  title={Evolution of the n-cube},
  author={Erd{\"o}s, Paul and Spencer, Joel},
  journal={Computers \& Mathematics with Applications},
  volume={5},
  number={1},
  pages={33--39},
  year={1979},
  publisher={Elsevier}
}

@article{ksbinary,
  title={The number of binary codes with distance 2},
  author={Korshunov, Aleksej D and Sapozhenko, Alexander A},
  journal={Problemy Kibernet},
  volume={40},
  pages={111--130},
  year={1983}
}

@article{jp,
  title={Independent sets in the hypercube revisited},
  author={Jenssen, Matthew and Perkins, Will},
  journal={Journal of the London Mathematical Society},
  volume={102},
  number={2},
  pages={645--669},
  year={2020},
  publisher={Wiley Online Library}
}

@article{container,
  title={On the number of connected subsets with given cardinality of the boundary in bipartite graphs},
  author={Sapozhenko, Aleksandr Antonovich},
  journal={Metody Diskret. Analiz},
  volume={45},
  number={45},
  pages={42--70},
  year={1987}
}

@inproceedings{sly,
  title={Computational transition at the uniqueness threshold},
  author={Sly, Allan},
  booktitle={2010 IEEE 51st Annual Symposium on Foundations of Computer Science},
  pages={287--296},
  year={2010},
  organization={IEEE}
}

@inproceedings{count,
  title={Counting independent sets in unbalanced bipartite graphs},
  author={Cannon, Sarah and Perkins, Will},
  booktitle={Proceedings of the Fourteenth Annual ACM-SIAM Symposium on Discrete Algorithms},
  pages={1456--1466},
  year={2020},
  organization={SIAM}
}

@article{x,
  title={Computational complexity of counting problems on 3-regular planar graphs},
  author={Xia, Mingji and Zhang, Peng and Zhao, Wenbo},
  journal={Theoretical Computer Science},
  volume={384},
  number={1},
  pages={111--125},
  year={2007},
  publisher={Elsevier}
}

@article{sapozhenko,
  title={On the number of connected subsets with given cardinality of the boundary in bipartite graphs},
  author={Sapozhenko, Aleksandr Antonovich},
  journal={Metody Diskret. Analiz},
  volume={45},
  number={45},
  pages={42--70},
  year={1987}
}
